\documentclass[notitlepage,11pt,reqno]{amsart}  
\usepackage[foot]{amsaddr}
\usepackage[numbers,sort]{natbib}
\usepackage{amssymb,nicefrac,bm,upgreek,mathtools,verbatim}
\usepackage[final]{hyperref}
\usepackage[mathscr]{eucal}
\usepackage{dsfont} 
\usepackage{graphicx}
\usepackage[margin=10pt,font=small,labelfont=bf]{caption}

\usepackage[margin=1in]{geometry} 
\allowdisplaybreaks  
\usepackage{color}  
\definecolor{dmagenta}{rgb}{.4,.1,.5}       
\definecolor{dblue}{rgb}{.0,.0,.5}     
\definecolor{mblue}{rgb}{.0,.0,.8}     
\definecolor{ddblue}{rgb}{.0,.0,.4}            
\definecolor{dred}{rgb}{.6,.0,.0}   
\definecolor{dgreen}{rgb}{.0,.5,.0}  
\definecolor{Eeom}{rgb}{.0,.0,.5}

\usepackage[normalem]{ulem}

\newtheorem{lemma}{Lemma}[section]
\newtheorem{theorem}{Theorem}[section]
\newtheorem{proposition}{Proposition}[section]
\newtheorem{corollary}{Corollary}[section]

\theoremstyle{definition}
\newtheorem{definition}{Definition}[section]
\newtheorem{assumption}{Assumption}[section]

\newtheorem{example}{Example}[section]

\theoremstyle{remark}
\newtheorem{remark}{Remark}[section]
\numberwithin{equation}{section}
\newcommand{\eps}{\epsilon}
\hypersetup{
  colorlinks=true,
  citecolor=dblue,
  linkcolor=dblue,
  frenchlinks=false,
  pdfborder={0 0 0},
  naturalnames=false, 
  hypertexnames=false,
  breaklinks}
\usepackage[capitalize,nameinlink]{cleveref}

\crefname{section}{Section}{Sections}
\crefname{subsection}{Section}{Sections}
\crefname{condition}{Condition}{Conditions}
\crefname{hypothesis}{Hypothesis}{Conditions}
\crefname{assumption}{Assumption}{Assumptions} 
\crefname{lemma}{Lemma}{Lemmas} 
  
\Crefname{figure}{Figure}{Figures}

\crefformat{equation}{\textup{#2(#1)#3}}
\crefrangeformat{equation}{\textup{#3(#1)#4--#5(#2)#6}}
\crefmultiformat{equation}{\textup{#2(#1)#3}}{ and \textup{#2(#1)#3}}
{, \textup{#2(#1)#3}}{, and \textup{#2(#1)#3}}
\crefrangemultiformat{equation}{\textup{#3(#1)#4--#5(#2)#6}}%
{ and \textup{#3(#1)#4--#5(#2)#6}}{, \textup{#3(#1)#4--#5(#2)#6}}%
{, and \textup{#3(#1)#4--#5(#2)#6}}

\Crefformat{equation}{#2Equation~\textup{(#1)}#3}
\Crefrangeformat{equation}{Equations~\textup{#3(#1)#4--#5(#2)#6}}
\Crefmultiformat{equation}{Equations~\textup{#2(#1)#3}}{ and \textup{#2(#1)#3}}
{, \textup{#2(#1)#3}}{, and \textup{#2(#1)#3}}
\Crefrangemultiformat{equation}{Equations~\textup{#3(#1)#4--#5(#2)#6}}%
{ and \textup{#3(#1)#4--#5(#2)#6}}{, \textup{#3(#1)#4--#5(#2)#6}}%
{, and \textup{#3(#1)#4--#5(#2)#6}}

\crefdefaultlabelformat{#2\textup{#1}#3}
\newcommand{\upu}{\Upupsilon}

\newcommand{\cG}{{\mathcal{G}}}  
\newcommand{\cH}{{\mathcal{H}}}  
\newcommand{\cK}{{\mathcal{K}}}  
\newcommand{\Lg}{\mathcal{L}}    

\newcommand{\cP}{{\mathcal{P}}}  
\newcommand{\sV}{{\mathcal{V}}}  
\newcommand{\cX}{{\mathcal{X}}}  
\newcommand{\sC}{\mathscr{C}}
\newcommand{\veps}{\varepsilon}
\newcommand{\frB}{\mathfrak{B}}
\newcommand{\frG}{\mathfrak{G}}
\newcommand{\frK}{\mathfrak{K}}

\newcommand{\beql}[1]{\begin{equation}\label{#1}}

\newcommand{\beq}{\begin{displaymath}}
\newcommand{\eeqno}{\end{displaymath}}
\newcommand{\eeq}{\end{equation}}

\newcommand{\E}{\mathbb{E}}
\newcommand{\PP}{\mathbb{P}}

\newcommand{\RR}{\mathds{R}}
\newcommand{\NN}{\mathds{N}}
\newcommand{\ZZ}{\mathds{Z}}

\newcommand{\Rd}{\mathds{R}^{d}}

\newcommand{\D}{\mathrm{d}}

\newcommand{\Act}{{\mathbb{U}}}
\newcommand{\Uadm}{\mathfrak{U}}
\newcommand{\Usm}{\mathfrak{U}_{\mathrm{SM}}}
\newcommand{\Ussm}{\mathfrak{U}_{\mathrm{SSM}}}

\newcommand{\Ind}{\mathds{1}}   
\newcommand{\Cc}{\mathcal{C}}   

\newcommand{\transp}{^{\mathsf{T}}}

\DeclareMathOperator*{\diag}{diag}

\newcommand{\grad}{\nabla}

\newcommand{\calP}{\mathcal{P}}

\newcommand{\cC}{\mathcal{C}}
\newcommand{\cV}{\mathcal{V}}

\newcommand{\cF}{\mathcal{F}}
\newcommand{\calO}{\mathcal{O}}

\newcommand{\calC}{\mathcal{C}}
\newcommand{\lsm}{\Lambda_{\text{SM}}}

\newcommand{\frC}{\mathfrak{C}}

\newcommand{\calK}{\mathcal{K}}
\newcommand{\frV}{\mathfrak{V}}
\newcommand{\cZ}{\mathcal{Z}}
\newcommand{\srO}{\mathscr{O}}
\newcommand{\fro}{\mathfrak{o}}

\makeatletter
\DeclareRobustCommand\widecheck[1]{{\mathpalette\@widecheck{#1}}}
\def\@widecheck#1#2{%
    \setbox\z@\hbox{\m@th$#1#2$}%
    \setbox\tw@\hbox{\m@th$#1%
       \widehat{%
          \vrule\@width\z@\@height\ht\z@
          \vrule\@height\z@\@width\wd\z@}$}%
    \dp\tw@-\ht\z@
    \@tempdima\ht\z@ \advance\@tempdima2\ht\tw@ \divide\@tempdima\thr@@
    \setbox\tw@\hbox{%
       \raise\@tempdima\hbox{\scalebox{1}[-1]{\lower\@tempdima\box
\tw@}}}%
    {\ooalign{\box\tw@ \cr \box\z@}}}
\makeatother

\usepackage{accents}
\newlength{\dhatheight}

\newcommand{\cA}{\mathcal{A}}
\newcommand{\bU}{\mathbb{U}}
\newcommand{\Wadm}{\mathfrak{W}}

\newcommand{\Wsm}{\mathfrak{W}_{\mathrm{SM}}}
\newcommand{\sW}{\mathscr{W}}
\newcommand{\wtau}{\widecheck\tau}

\let\oldtocsection=\tocsection
\let\oldtocsubsection=\tocsubsection
\let\oldtocsubsubsection=\tocsubsubsection
\renewcommand{\tocsection}[2]{\hspace{0em}\oldtocsection{#1}{#2}}
\renewcommand{\tocsubsection}[2]{\hspace{1em}\oldtocsubsection{#1}{#2}}
\renewcommand{\tocsubsubsection}[2]{\hspace{2em}\oldtocsubsubsection{#1}{#2}}

\newcommand{\ttl}{\Large Ergodic Risk Sensitive Control of Diffusions \\[5pt] 
under a General Structural Hypothesis}

\begin{document}

\title[]{\ttl}

\author{Sumith Reddy Anugu$^\dagger$}
\author{Guodong Pang$^\ddag$}

\address{$^\dagger$Institut F\"ur Mathematik, Technische Universit\"at Ilmenau, Ilmenau, Germany 98693}
\email{sumith-reddy.anugu@tu-ilmenau.de}
\address{$^\ddag$Department of Computational Applied Mathematics and Operations Research,
George R. Brown School of Engineering,
Rice University,
Houston, TX 77005}
\email{gdpang@rice.edu}

\begin{abstract} 
We study the infinite-horizon average (ergodic) risk sensitive control problem for diffusion processes under a general structural hypothesis: there is a partition of state space into two subsets, where the controlled diffusion process satisfies a Foster-Lyapunov type drift condition in one subset, under any stationary Markov control, while the near-monotonicity condition is satisfied with the running cost function being inf-compact in its complement. In particular, under these conditions we show that among all admissible controls, the optimal ergodic risk sensitive cost is attained for stationary Markov controls which are characterized as minimizers to the corresponding Hamilton-Jacobi-Bellman equation.  The proof involves considering an inf-compact perturbation to the running cost over the entire space such that the resulting ergodic risk sensitive control problem is well-defined. We then  use the existing results in the case of inf-compact running cost to characterize the optimal Markov controls among all the admissible controls and also show that the limit of  the optimal values of the perturbed problems coincides with the optimal value of the original problem. The heart of the analysis lies in exploiting the variational formula of exponential functionals of Brownian motion and applying it to the objective exponential cost function of the controlled diffusion. This representation facilitates us to view the risk sensitive cost for any stationary Markov control as the optimal value of a control problem of an extended diffusion involving a new auxiliary control where the optimal criterion is to maximize the associated long-run average cost criterion that is a difference of the original running cost and an extra term that is quadratic in the auxiliary control. The main difficulty in using this approach lies in the fact that  tightness of mean empirical measures of the extended diffusion is not a priori implied by the analogous tightness property of the original diffusion.  We overcome this by establishing a priori estimates for the extended diffusion associated with the nearly optimal auxiliary controls. 
\end{abstract}

\keywords{Ergodic risk sensitive control of diffusions, general structural hypothesis, uniform stability, near-monotonicity, inf-compact perturbation, variational representation, extended diffusion, ground diffusion, characterization of optimal control}

\date{\today}

\maketitle

\allowdisplaybreaks

\section{Introduction} \label{sec-intro}
For a running cost $r$, a control $U$, and a controlled diffusion $X$, the problem of ergodic risk sensitive control (ERSC)  minimizes 
$$ \limsup_{T\to\infty} \frac{1}{T} \log \E\Big[\exp\Big(\int_0^T r(X_t,U_t) \D t\Big)\Big]$$
over a set of admissible controls $U$- 
unlike  in the case of conventional ergodic control (CEC) problem where we minimize
$$ \limsup_{T\to\infty} \frac{1}{T} \E\Big[\int_0^T r(X_t,U_t) \D t\Big]\,.$$
The origin of  ERSC problems dates back to \cite{howard1972risk}, where the authors studied the problem in the setup of controlled Markov chains with finite state and control sets. 
ERSC problems for Markov processes ({ both in} discrete and continuous time settings) have been extensively studied since then. 
We refer the reader to 
\cite{biswas2022survey,bauerle2023markov} for more extensive and recent survey of the results on both continuous and discrete time Markov chains. { Here, we only focus on  ERSC problems for diffusions.} 
In addition to ergodic cost problems, other objective criteria have been considered in the literature, such as finite horizon problems (see, e.g., \cite{bensoussan1985optimal,nagai1996bellman,da2002finite, fleming1991risk,james1992asymptotic}) and  infinite horizon discounted cost problems (see, e.g., \cite{FM95,nagai1996bellman,runolfsson94}). Risk sensitive controls have found applications in many areas such as portfolio optimization \cite{fleming2000risk,fleming2002risk,bielecki2005risk,fleming1999optimal,nagai2003optimal}, insurance \cite{grandits2007optimal,eisenberg2023measuring},
and  robust control theory \cite{dupuis2000robust,whittle1990risksensitive,whittle1981risk}.

As was the case in the context of CEC problems (see \cite[Chapter 3]{arapostathis2012ergodic}), the ERSC problems in the case of diffusions have been studied under the assumptions that can be broadly divided into two categories: (i) Blanket stability: the diffusion is assumed to be uniformly stable for all stationary Markov controls. See  \cite{FM95,runolfsson90,runolfsson94,biswas2011risk,biswas-eigenvalue,biswas2010risk,ari2018strict,AB20,ABBK20}  for the relevant literature using this assumption. ERSC problem is also studied under the blanket stability conditions in the context of switching diffusions \cite{biswas2022ergodic}, jump diffusions \cite{ari2022risk,pradhan2021riskjump}, reflecting diffusions \cite{pradhan2021risk,ghosh2018risk,ghosh2022nonzero}, and stochastic differential games \cite{biswas2020zero,ghosh2018risk,ghosh2022nonzero,ghosh2023nonzero,ghosh2021ergodic,basu2012zero}. 
(ii) Near-monotonicity: in addition to a well-posedness condition, the running cost is assumed to be strictly greater than the optimal value outside a compact set, which includes inf-compact functions on the entire space. 
See~\cite{AB18}, where the authors additionally assume that controlled diffusion is recurrent.

In this paper, we study the ERSC problem under the assumption that the underlying diffusion satisfies a more general structural hypothesis \emph{viz.,} the state space is partitioned into two subsets, where the running cost is inf-compact in one subset and the controlled diffusion satisfies a Foster-Lyapunov drift condition in its complement.  Under this assumption, we show that the associated Hamilton-Jacobi-Bellman (HJB) equation is well-posed and also characterize the optimal stationary Markov controls. 
A similar study in the context of CEC problems was carried out in  \cite{ABP15} (in the case of continuous diffusions) and in \cite{APZ20} (in the case of jump diffusions). Such a general structural assumption is motivated from optimal control under ergodic cost criteria of parallel server networks in the Halfin-Whitt asymptotic regime, where the controlled diffusion may not satisfy either the blanket stability or the near-monotonicity condition \cite{ABP15,AP2016,AHP21,hmedi2022uniform}. It is worth highlighting that the structural assumption in the ERSC setting is properly adapted for the multiplicative HJB equation, from that in  \cite{ABP15,AP2016} under the CEC setting, see Assumption \ref{a-main}. 

 In the CEC setting, the authors in  \cite{ABP15,AP2016,APZ20} construct an inf-compact function on the entire space that satisfies the following two properties 1.) it is appropriately comparable to the running cost function and 2.) the CEC cost associated with this function is finite whenever the CEC cost associated with the running cost is finite. The aforementioned inf-compact function is then used to set up a family of CEC problems with a perturbed running cost that is the sum of the original running cost and a small perturbation of the inf-compact function. This family of CEC problems clearly satisfies the near-monotone condition (in fact, the perturbed running cost is inf-compact). From here, the authors study the limit of this family of CEC problems as the perturbation parameter goes to zero which is then shown to reduce to the original CEC problem. 
The techniques in these works, extensively use the notion of mean empirical measure and the convex analytic approach \cite[Section 3.2]{arapostathis2012ergodic} as the CEC cost can be represented as a linear functional of mean empirical measure and hence convex.  In contrast, the ERSC cost can neither be represented as a linear functional of mean empirical measure nor is it convex. Therefore, none of the techniques in the aforementioned works have any immediate applicability in the ERSC case and hence, we introduce a new set of techniques to overcome this difficulty.  

We now briefly discuss the methodology and the technical challenges that we shall encounter in studying the ERSC problem in our setup.  We proceed {with} the following keys steps: 

{\bf Step 1.} Defining a perturbed ERSC problem: We achieve this by constructing an inf-compact function over the entire space with the desired property that the associated ERSC cost is finite whenever the ERSC {cost} associated with the original running cost $r$ is finite and is appropriately comparable to the running cost (see Lemma~\ref{lem-fin-cost}).  Then, define the perturbed running cost $r^\veps$ (with $\veps$ being the perturbation parameter) as an appropriate linear combination of the original running cost and the constructed inf-compact function (see~\eqref{def-r-cost-pert}).  Our goal is then to take $\veps\to 0$ and show that the corresponding optimal ERSC cost converges to the original optimal ERSC cost and then characterize the optimal Markov controls of the original ERSC problem. This is in a similar spirit as in  \cite{ABP15, AP2016, APZ20} for the CEC problems, however, 
 after this, the techniques and the methodology used in this paper differ significantly.  {One of the} two important consequences of this construction is that whenever ERSC cost for a stationary Markov control is finite,  such a Markov control is stable (see Corollary~\ref{cor-stable}) and the running cost $r^\veps$ is exponentially `uniformly' integrable (see Lemma~\ref{lem-ui}). This `uniform' integrability of the running cost $r^\veps$ is fundamental in making the analysis of  the perturbed ERSC problem \emph{via} variational representation (the next step) useful - in contrast, it is not clear if such a property is satisfied by the original running cost. 
 
{\bf Step 2.}  Application of variational representation:  Here,
 we exploit a well-known variational formula of exponential functions of  Brownian motion $W$ (see formula in \eqref{eq-var-rep}, e.g., \cite[Theorem 5.1]{boue1998} and also \cite{budhiraja2019analysis}), and  write
 \begin{align}\label{1} \limsup_{T\to\infty} \frac{1}{T}\log \E\Big[\exp\Big(\int_0^T r^\veps (X_t,U_t) \D t \Big)\Big]=  \limsup_{T\to\infty} \sup_{w\in \cA}\E\Big[\frac{1}{T}\int_0^T \Big( r^\veps(Z_t, \widetilde U_t)-\frac{1}{2}\|w_t\|^2 \Big) \D t\Big],\end{align}
 where $\cA$ is an appropriately defined set, $Z$ is an ``extended" process (see~\eqref{X-control} for its definition) with a properly modified control process $\widetilde U$, and an auxiliary control $w$.  See the details in Section \ref{sec-var-BM}. 
A significant amount of work in this paper will involve showing that the limit superior and supremum can be interchanged and moreover, one can replace $\cA$ by a set of stationary Markov controls, \emph{i.e.,} we will show that 
$$ \limsup_{T\to\infty} \frac{1}{T} \log\E\Big[\exp\Big(\int_0^T r^\veps(X_t,U_t) \D t \Big)\Big]=   \sup_{w\in \widetilde \cA}\limsup_{T\to\infty}\E\Big[\frac{1}{T}\int_0^T \Big( r^\veps(Z_t, \widetilde U_t)-\frac{1}{2}\|w_t\|^2 \Big)\D t\Big],$$
where $\widetilde \cA$ is an appropriate set of stationary Markov controls. 

The main difficulty is  that it turns out that  tightness of the mean empirical measures ({MEMs}) of the `extended' process $Z$ associated with nearly optimal $w$ is necessary to obtain the above interchangeability. However, {such a} tightness property of the {MEMs} of $Z$ is not at all immediate from the analogous tightness property corresponding to $X$.  
 We prove  tightness of the {MEMs} of $Z$ by considering a truncated version of the running cost $r^\veps$ owing to the fact that $r^\veps$ exponentially `uniformly' integrable as mentioned above (see Lemma~\ref{lem-comp-X-pert}).
 Using this truncated version, one immediately arrives at a uniform in $T$ estimate of the second term on the right hand side of~\eqref{1} in Lemma~\ref{lem-comp-w-pert}. 
  This consequently, gives us the desired  tightness property in Lemma~\ref{lem-comp-X-pert}  with the help of the constructed inf-compact function.
From here, we proceed to show that whenever $U_t$ happens to be a stationary Markov control $v(\cdot)$, the ERSC cost can be written as the CEC cost where the optimal criterion is to maximize the running cost function $r(x,v(x))-\frac{1}{2}\|w\|^2$ with an ``extended" diffusion.  
As a consequence, we can represent the ERSC cost in a form which is linear in $r$ and more importantly, the optimal ERSC value can be written as an optimal value of an inf-sup problem with the above long-run average cost criterion (see Lemma~\ref{lem-limit-truc-pert}(iii)). However, it is not a priori clear if the inf-sup operations can be interchanged as the usefulness of the variational representation relies heavily on this fact. It turns out that such an interchange is indeed possible. This is proved by studying a family of  ergodic two-person zero-sum (TP-ZS) stochastic games and borrowing the existing results from \cite{borkar1992stochastic}. An important property of this family of ergodic TP-ZS stochastic games is that the maximizing strategies vanish outside a large compact set which is later on used frequently in the proof of Theorem~\ref{thm-diffusion}.
 In Theorem~\ref{thm-lin-rep-pert}, we show that the values of this family of ergodic TP-ZS stochastic games converge to that of the optimal value of ERSC problem associated with $r^\veps$.  This representation of the ERSC optimal cost associated with $r^\veps$ is extensively used subsequently in showing that the ERSC optimal {costs associated} with both $r$ and $r^\veps$ over all admissible controls is  achieved by stationary Markov controls.  This step can easily be regarded as the major novelty of the paper. 

{\bf Step 3.} Analyzing the limiting behavior as $\veps\to 0$: The starting point of this step is to establish certain uniform (in $\veps$) estimates in Lemma~\ref{lem-comp-X-w-pert} that will then be used to show that the optimal ERSC values associated with $r^\veps$ converge to the optimal value of the original ERSC problem (which is associated with $r$) in Theorem~\ref{thm-pert-limit}.  
 Moreover, we show that the ERSC cost associated with $r^\veps$ under any stationary Markov control also converges to its original counterpart (see Theorem~\ref{thm-pert-limit}).

 The analysis of the limiting behavior, particularly, the proof of the main result - Theorem~\ref{thm-diffusion},  is further divided into four parts. 
\begin{itemize}
\item[(i)]  {\it Well-posedness of the HJB equation}: Existence of solutions to the multiplicative  HJB equation involves a direct application of the standard elliptic regularity theory. For uniqueness, we exploit the stability of the so-called ``ground" diffusion (a particular case of the ``extended" diffusion), which turns out to be difficult to prove directly under our hypothesis. We overcome this difficulty by examining tightness of the {MEMs} of the ``extended" process under nearly optimal auxiliary controls.
\item[(ii)] {\it Characterization of optimal Markov controls:}   To prove that the minimizers of HJB are optimal, the main difficulty is that it is not a priori  clear if the Markov controls given as the minimizers to the HJB equation result in a finite ERSC cost.
 We prove this claim is true by again using the variational formulation and proving  tightness of the {MEMs} associated with nearly optimal auxiliary controls.  To prove that optimal stationary Markov controls are minimizers to HJB equation, we prove this by contradiction and using the fact that the ERSC problem associated with $r$ can be written as a limit of a family of ergodic TP-ZS stochastic games. 
 \item[(iii)] {\it Stochastic representation of the HJB solution:}   Because of the work in \cite{ari2018strict}, it is well known that existence of a stochastic representation for the solutions of a multiplicative Poisson equation is very closely related to the stability of the ``ground" diffusion. We use this close connection and the stability of the ``ground" diffusion from step (i) to infer the existence of a stochastic representation of the HJB solution. 
  \item[(iv)] {\it Minimum over admissible controls is achieved by Markov controls:}  We again exploit the CEC problem from the variational formulation associated with stationary Markov controls.   We first establish that just like in the perturbed case, the original optimal ERSC cost over stationary Markov controls can also be written as the limit of a family of ergodic TP-ZS stochastic games mentioned earlier, but with $r^\veps$ replaced by $r$. This in conjunction with variational formulation (associated with admissible controls) is used to prove the result. 
  \end{itemize}

We highlight that the last part of the main result is another novel part of the paper. 
 It is well known that under the assumption of uniform stability (see \cite[Theorem 4.1]{ari2018strict}), the infimum of ERSC cost over admissible controls is the same as the infimum over stationary Markov controls. However, it is not known if this is the case under other conditions like {near-monotonicity} (see \cite[Remark 1.3]{AB18}). In contrast, it is easy to show the analogous result in the case of CEC problem (see \cite[Theorem 3.4.7]{arapostathis2012ergodic}, \cite[Theorem 3.1(b)]{ABP15}). 
 As mentioned already, the convex analytic approach cannot be directly applied to the ERSC problem, but it can be applied to the family of ergodic TP-ZS stochastic games {that arise} from {the} variational formulation - in particular, we use the results from \cite{borkar1992stochastic} where authors use the convex analytic approach to analyze TP-ZS  stochastic games under ergodic cost criterion.

We remark that relating the ERSC cost (for a particular control) to the CEC problem that is mentioned above, has been studied in the existing literature.  In the case of a finite-horizon risk sensitive control problem, such a variational formulation was derived using the theory of large deviations in \cite{whittle91}. One of the first works studying the variational formulation of the ERSC problem for diffusions is in \cite{FM95,runolfsson90,runolfsson94,boue2001}, which was then followed by \cite{biswas2010risk,AB20,ABBK20,ari2018strict}. 
In \cite{runolfsson94}, the variational formulation was derived under restricted conditions of Markov controls that are continuous in their arguments 
 and  inf-compactness of the running cost function, whereas, in \cite{runolfsson90} the case of a linear-quadratic control  problem is studied.  
 However, the above works derived the aforementioned relation in more restrictive settings and are not useful immediately in our setting of ergodic criteria. On the contrary, we derive this relation using the variational formula for exponential functions of Brownian motion 
 (see~\eqref{eq-var-rep})
 which holds under very mild assumptions on the functional - it is only required to be a non-negative Borel measurable functional.

\subsection{Organization of the paper}

The rest of the paper is organized as follows: We conclude this section by introducing the necessary notation used through out the paper. In Section~\ref{sec-model}, we set up the model and give the assumptions, and state the main result of the paper. 
In Section~\ref{sec-est}, we prove various estimates that help us construct an inf-compact perturbation to the original running cost. This section also contains the existing results in the ERSC problems under near-monotonicity condition which are stated in the case of an inf-compact running cost. 
Section~\ref{sec-var-BM} develops the variational formulations in the context of both the perturbed and original ERSC problems. This section contains certain crucial uniform in time estimates that are used extensively from thereon.  
Section~\ref{sec-proof-pert} then develops the relevant results for the perturbed ERSC problem using the variational formulation. 
This section in particular, re-casts the already existing results using this formulation in a way that is amenable to certain techniques in \cite{ABP15}. 
Finally, in Section~\ref{sec-proof-main}, we take the limit as the perturbation goes to zero and analyze the limiting behavior of optimal values of the perturbed ERSC cost and the solutions of the associated HJB equations.

\subsection{Notation} We use $(\Omega, \cF, \PP)$ to denote the underlying abstract probability space with $\E$ as the associated expectation. $\E_x$ denotes the expectation when the underlying process starts at $x$. The standard Euclidean norm in $\RR^d$ is denoted by $\| \cdot \|$, $ x\cdot y $  denotes the inner
product of $x,y\in \RR^d$, and $x\transp$ denotes the transpose of $x \in \RR^d$. The set of nonnegative real numbers (integers) is denoted by $\RR_+$ ($\ZZ_+$), $\NN$ stands for the set of natural numbers, and $\Ind_{ A}(\cdot)$ denotes the indicator function corresponding to set $A$. The minimum (maximum) of two real numbers $a$ and $b$ is denoted by $a \wedge b$ ($a \vee b$), respectively, and
$a^{\pm} \doteq  (\pm a) \vee 0$. The closure, boundary, and complement of a set $A \subset \RR^d$ are denoted by $\bar  A$,
$\partial A$, and $A^c$, respectively. 
The term domain in $\RR^d$ refers to a nonempty, connected open subset of $\RR^d$. For a domain $D \subset \RR^d$, the space $\cC^k (D)$ ($\cC^\infty (D)$, respectively), $k \geq 0$, refers to the class of all real-valued
functions on $D$ whose partial derivatives up to order $k$ (any order, respectively) exist and are continuous.  {$\cC_b(\RR^d)$} and  $\cC^\infty_c(\RR^d)\subset \cC(\RR^d)$ denote the set of  bounded continuous functions and set of compactly supported smooth functions, respectively. By
$\cC^{ k,\alpha}(\RR^d)$,  we denote the set of functions that are $k$-times continuously differentiable and whose $k$-th
derivatives are locally H\"{o}lder continuous with exponent $\alpha$. The space $L^p (D)$, $p \in [1, \infty)$, stands for
the Banach space of (equivalence classes of) measurable functions $f$ satisfying $\int_D |f (x)|^p \D x < \infty$,
and $L^\infty (D)$ is the Banach space of functions that are essentially bounded in $D$. The standard
Sobolev space of functions on $D$ whose generalized derivatives up to order $k$ are in $L^p (D)$, equipped
with its natural norm, is denoted by $W^{k,p} (D)$, $k \geq 0, p \geq 1$. In general, if $\cX$ is a space of real-valued
functions on a set $Q$, $\cX_\text{loc}$ consists of all functions $f$ such that $f \phi \in \cX$  for every $\phi $ that is compactly supported smooth function on $Q$.  Here, $f\phi$ is simply the scalar multiplication of the functions $f$ and $\phi$. {For $T>0$, $\frC_T^d$ denotes the set of $\RR^d$--valued continuous functions on $[0,T]$ equipped with uniform topology.}  If $T=\infty$, we simply write $\frC^d$.

 For a Polish space $\cX$, $\cP(\cX)$ is the set of Borel probability measures on $\cX$ equipped with the topology of weak convergence. 
Let $\sC_0$ denote the set of all non-negative bounded functions on $\RR^d$ that vanish at infinity and not identically equal to zero. For a positive function $g\in \cC(\RR^k)$, $\srO(g)$ denotes the set of all functions $f\in\cC(\RR^k)$ which have the property $ \limsup_{\|x\|\to\infty} \frac{|f(x)|}{g(x)}<\infty$
and $\fro(g)$ denotes the set of all functions $f\in\cC(\RR^k)$ that have the property $ \limsup_{\|x\|\to\infty} \frac{|f(x)|}{g(x)}=0$. For a Borel set $A\subset \RR^d$, $\tau(A)$ and $\widecheck \tau(A)$ are first exit and hitting times of $A$, respectively of the underlying process. The underlying process will be evident from the context. $B_R\subset\RR^d$ denotes the open ball of radius $R$ around origin. For short, we write $\tau_R$ and $\widecheck \tau_R$ for $\tau(B_R)$ and $\widecheck \tau(B_R)$, respectively. 
We refer to a measure $\pi_T\in \cP(\cX)$ as the mean empirical measure (MEM) of an $\cX$--valued process $Y$ (starting at $y\in \cX$) on $[0,T]$, if it is defined as  $ \pi_T(A)\doteq T^{-1}\E_y\Big[\int_0^T \Ind_{A}(Y_t) \D t\Big]$,  for a Borel set $A\subset \cX$.  
{We say a family $\{\pi_n \}_{n\in \NN}\subset \calP(\cX)$ is tight, if for every $\epsilon>0$, there exists a compact set $K_\epsilon\subset \cX$ such that $\pi_n(K_\epsilon^c)<\epsilon$, for every $n$. }

\medskip

\section{Model and Results}\label{sec-model}
We consider a $\RR^d$--valued controlled diffusion $X=\{X_t: t\ge 0\}$ given as the solution to  
\begin{equation} \label{eqn-X}
X_t = {X}_0 + \int_0^t b(X_s, U_s)\D s + \int_0^t\Sigma(X_s) \D W_s\,.
\end{equation}
{Here, $W$ is a $\RR^d$--valued Brownian motion, the process  $U$ (referred to as control) is assumed to take values in a compact metric space $\bU$ and the coefficients $b$ and $\Sigma$ satisfy the following conditions:
\begin{enumerate}
\item[(i)] (Local Lipschitz continuity) $b:\RR^d\times \bU\rightarrow \RR^d$ and $\Sigma:\RR^d\rightarrow\RR^{d\times d}$ is continuous and for every $R>0$, there exists $C_R>0$ such that 
\begin{align}\label{eq-cond-loc-lip}
\|b(x,u)-b(y,u)\|+ \|\Sigma(x)-\Sigma(y)\|\leq C_R\|x-y\|, \text{ for  $x,y\in B_R$\,.}
\end{align}
\item[(ii)] (Linear growth) There exists a constant $C_{b,\Sigma}>0$ such that 
\begin{align}\label{eq-cond-lin-growth} \|b(x,u)\|^2 + \|\Sigma(x)\|^2 \leq C_{b,\Sigma}(1+\|x\|^2), \text{ for $x\in \RR^d$\,.}\end{align}
\item [(iii)](Non-degeneracy) There exists $\sigma>0$ such that 
\begin{align}\label{eq-cond-non-deg} z\transp\Sigma(x)\big(\Sigma(x)\big)\transp z\geq \sigma \|z\|^2, \text{ for $z\in \RR^d$\,.}\end{align}
\end{enumerate}}
For simplicity, assume that  ${X}_0={x}$ is a deterministic constant.

\begin{definition}\label{def-adm-cont}
A $\bU$--valued process $U$ is said to be admissible if it satisfies the following: if $U_t=U_t(\omega)$ is jointly measurable in $(t,\omega)\in \RR^+\times \Omega$ and  for every $0\leq s< t$, $W_t-W_s$ is independent of the completed filtration (with respect to $(\cF,\PP)$) generated by $\{X_0,U_r, W_r: r\leq s\} $. The set of all such controls is denoted by $\Uadm$.
\end{definition}

Given a process  $U=\{U_t:t\geq 0\} $ that is admissible, the controlled diffusion ${X}$ in \eqref{eqn-X} has a unique { strong} solution \cite[Theorem 2.2.4]{arapostathis2012ergodic} - see \cite[Definition 5.2.1]{karatzas1991brownian} for definition of strong solution. On the other hand, if $U_t= v(t,X_{[0,t]})$ for a Borel measurable map $v: \RR_+\times \frC^d\rightarrow \bU$ (such a control is called a feedback control), then~\eqref{eqn-X} reduces to 
\begin{equation} \label{eqn-X-feed-0}
 X_t = x + \int_0^t b\big(X_s, v(s, X_{[0,s]})\big)\D s + \int_0^t\Sigma( X_s) \D W_s\,.
\end{equation}
 In this case,~\eqref{eqn-X-feed-0} is only guaranteed to have a unique weak solution; see \cite[Theorem 2.2.11]{arapostathis2012ergodic}.  Let the weak solution be given by the  pair of processes $( X, \widehat W)$ and a filtered probability space $(\widehat \Omega, \widehat \cF,\widehat F_t,\widehat \PP)$ - see \cite[Definition 5.3.1]{karatzas1991brownian} for the  definition of weak solution.   However, using \cite[Theorem 2.3.4 (b)]{arapostathis2012ergodic}, we can infer the following: there exists a probability space $(\widetilde \Omega, \widetilde \cF,\widetilde \PP)$ (which is possibly an augmented version of $(\Omega, \cF,\PP)$) on which for the given $\RR^d$--valued Brownian motion $W$ and  $x\in \RR^d$, there exists a process $\widetilde X=\{\widetilde X_t:t\geq 0\}$ given as 
\begin{equation} \label{eqn-X-feed}
\widetilde X_t = x + \int_0^t b\big(\widetilde X_s, v(s,\widetilde X_{[0,s]})\big)\D s + \int_0^t\Sigma(\widetilde X_s) \D W_s\,.
\end{equation}
such that the laws of $ X$ and $\widetilde X$ are identical. Using \cite[Lemma 3.22]{kallenberg2002foundations} and the notion of regular conditional probability (see \cite[Pg. 306--308]{karatzas1991brownian}), we can conclude that there exists a measurable function (depending on $U_{\cdot}= v(\cdot,X_{[0,\cdot]})$; we suppress the dependence, however) $F: \frC^d\times [0,1] \rightarrow \frC^d$ such that 
$$ \widetilde X_{\cdot}= F(W_{\cdot}, \vartheta )$$
with $\vartheta$ being a uniform random variable on $[0,1]$ that is independent of Brownian motion $W$.  Since we are only concerned with the ERSC cost in this paper (see~\eqref{eq-ersc-cost}) which only involves expectation, it is equivalent to work with either $ X$ or $\widetilde X$ as their laws are identical. However, we choose to work with $\widetilde X$ as this is in the  more amenable form in order to apply   Lemma~\ref{lem-var-rep-BM-gen} to the ERSC cost in~\eqref{eq-ersc-cost}.  
\begin{remark}\label{rem-notation}To keep the notation simple, without further mention, we use $X$ for $\widetilde X$ with the understanding that the underlying probability space is possibly an augmented version of $(\Omega, \cF,\PP)$.\end{remark}

Let $\Usm\subset \Uadm$ denote the set of stationary Markov controls. In order to study the convergence of stationary Markov controls or existence of optimal stationary Markov controls, it is useful to consider a weaker notion of a stationary Markov control, \emph{viz.,} relaxed control - the control is defined in the sense of distribution. To be more precise, a {stationary Markov control} $v$ is said to be a relaxed Markov control if $v=v(\cdot)$ is a Borel measurable map from $\RR^d$ to $\cP(\bU)$. In this case, we write $v(\D u|x)$ to distinguish the relaxed Markov control $v$ from other {stationary Markov controls} which are referred to as precise Markov controls. Clearly, the set of relaxed Markov controls contains $\Usm$. But with slight abuse of notation, we represent the set of relaxed Markov controls also by $\Usm.$   Under $v\in\Usm$, there exists a unique strong solution to~\eqref{eqn-X}. Moreover, in this case, the solution ${X}$ is strong Markov  (from \cite[Theorem 2.2.12]{arapostathis2012ergodic}) and the transition probabilities are locally H{\"o}lder continuous (see \cite[Theorem 4.1]{bogachev2001}). 
For every $u\in\bU$, we denote the generator $\Lg^{u}:\Cc^{2}(\RR^{d})\mapsto\Cc(\RR^{d})$ of the controlled diffusion ${X}$ as 
\begin{align*}
\Lg^{u} f(x) & \doteq   \sum_{i=1}^d b_i(x,u) \frac{\partial}{\partial x_i} f(x) + \frac{1}{2}\sum_{i,j=1}^d A_{ij}(x) \frac{\partial^2}{\partial x_i \partial x_j} f(x)\,, 
\end{align*}
where $A(x)\doteq \Sigma(x)\Sigma(x)\transp$.
It is the generator of a strongly-continuous
semigroup on $\Cc_{b}(\RR^{d})$, which is strong Feller.
We denote by $\Ussm$ the subset of $\Usm$ that consists
of \emph{stable controls}, i.e.,
under which the controlled process $X$ is positive recurrent.
In the following, whenever we are dealing with a generic admissible control, we denote it by $U$ and a generic stationary Markov control is denoted by $v$.

Let $r:\RR^d\times\bU\rightarrow \RR_+$ be a continuous function that is locally Lipschitz in the first argument (uniformly in the second). The ergodic risk-sensitive cost function is given by
{\begin{align}\label{eq-ersc-cost} 
J(x, U) \doteq \limsup_{T\to\infty} \frac{1}{T} \log \E_x^U\Big[ \exp\Big( \int_0^T r({X}_t, U_t )\D t \Big)\Big]\,.
\end{align}}
In the above and in what follows, we emphasize that the underlying controls are $U\in \Uadm$ and $v\in \Usm$ by writing $\E_x^{U}$ and $\E_x^v$, respectively. 
The associated ERSC cost minimization problem is given by
\begin{equation*}
{\Lambda}({x}) \doteq  \inf_{U\in \Uadm}  J({x}, U)\, \quad \text{ and } \quad \Lambda\doteq \inf_{x\in \RR^d} \Lambda(x)\,.
\end{equation*}
$\Lambda(x)$ is the optimal value function for the ERSC problem given the initial state ${x}$. 
In addition, let 
\begin{equation*}
{\Lambda}_{\text{SM}}({x}) \doteq \inf_{v\in \Usm}  J({x}, v) \quad \text{ and } \quad  \Lambda_{\text{SM}}\doteq \inf_{x\in \RR^d}\Lambda_{\text{SM}}(x)\,.
\end{equation*}
 It is easy to see that $\Lambda\leq \Lambda_{\text{SM}}$.
For notational convenience, let $
{\Lambda}_v(x) \doteq  J(x, v), \quad \text{for} \,\, v\in \Usm$. {To keep the expressions concise, we let 
$$ r^v(x)\doteq r(x,v(x)) \,\,\text{ and }\,\, \Lg^{v} f(x) \doteq   \sum_{i=1}^d b_i(x,v(x)) \frac{\partial}{\partial x_i} f(x) + \frac{1}{2}\sum_{i,j=1}^d A_{ij}(x) \frac{\partial^2}{\partial x_i \partial x_j} f(x)\,, $$
whenever the underlying control is $v\in \Usm$. When $v\in \Usm$ is a relaxed Markov control, we replace $r(x,v(x))$ and $b(x,v(x))$ by $\int_{\bU}r(x,u)v(du|x)$ and $\int_{\bU}b(x,u)v(du|x)$, respectively.}

In what follows, we encounter ERSC problems {associated with} various running costs. So to emphasize the dependence  on the running cost $r$, we write $J(x,U)[r]$ (for $U\in \Uadm$) and $\Lambda_{x,v}[r]$ (for $v\in \Usm$) when we are referring to the ERSC cost under $U\in \Uadm$ and under $v\in \Usm$, respectively. It turns out that under our setup,  the initial condition $x$ is irrelevant in the case of stationary Markov controls. Hence, for $v\in \Usm$, we simply $\Lambda_v[r]$.  Also,  we write $\Lambda[r]$ and $\Lambda_{\text{SM}}[r]$ if we are referring to  
the optimal values of the ERSC problem corresponding to the running cost $r$ over $\Uadm$ and $\Usm$, respectively over all initial conditions.   We remark that for $v\in \Usm$, even though the running cost function depends on $v\in \Usm$ as $r^v(\cdot)=r(\cdot, v(\cdot))$, we drop this dependence when we write $\Lambda_v[r]$ (and not write $\Lambda_v[r^v]$), as it is clear from the subscript that the underlying control is $v$. 

Also, define $ \Usm^o\doteq \left\{v\in \Usm: \Lambda _v[r]=\Lambda_{\text{SM}}[r]\right\}$. 
In other words, $\Usm^o$ is the set of optimal stationary Markov controls. We remark that a priori this set may be empty.
\begin{definition} For a $\delta>0$, we say $U\in \Uadm$ is $\delta$--optimal RS control for $\Lambda[r]$, if  $ J(x,U)[r]\leq \Lambda[r] +\delta$. 
Similarly, for a $\delta>0$, we say $v\in \Usm$ is $\delta$--optimal RS control for $\lsm[r]$, if  $ \Lambda_v[r]\leq \lsm[r]+\delta$.
\end{definition}

\begin{definition} \label{def-nm}A continuous function $f:\RR^d\times\bU\rightarrow \RR$ is said to be near-monotone relative to  $\lambda \in \RR$, if there exists $\eps>0$ such that $K_\eps\doteq \{x\in \RR^d: \min_{u\in \bU} f(x,u)\leq \lambda +\eps\}$ is either compact or empty. Also, we say $f$ is inf-compact on an open set $\calO\subset\RR^d$ if $\{x: \min_{u\in \bU} f(x,u)\leq l\}\cap \overline \calO$ is compact (or empty) set of $\RR^d$, for every $l\in\RR$.  If $\calO=\RR^d$, then we simply say $f$ is inf-compact.
\end{definition} 
{
\begin{remark}\label{rem-nm} Any inf-compact function $f:\RR^d\times \bU\rightarrow \RR_+$ is near-monotone relative to every $\lambda\in \RR$.
\end{remark}
}
Let $f:\RR^d\rightarrow \RR$  be a locally bounded function {that is uniformly} bounded from below. 
The principal eigenvalue $\lambda_v^*[f]$ for $v\in \Usm$ is defined  as 
\begin{align}\label{def-princp-eig} \lambda_v^*[f]\doteq \inf\Big\{\lambda\in \RR: \exists \psi\in W^{2,d}_{\text{loc}} (\RR^d) \text{ such that } \psi>0, \Lg^v\psi + (f-\lambda) \psi\leq 0 \text{ a.e. } x\in \RR^d\Big\}\,.\end{align}
The associated $\psi_v^f\in W^{2,d}_{\text{loc}} (\RR^d) $ that satisfies 
$$\Lg^v\psi_v^f(x) + f(x)\psi_v^f(x)=\lambda_v^*[f] \psi_v^f(x)\,,\text{ for a.e. } x\in \RR^d\,.  $$
is referred to as the principal eigenfunction. The pair $(\psi_v^f,\lambda_v^*[f])$ is referred to as the principal eigenpair of the operator $\Lg^v+f$. We remark that such a pair may not be necessarily unique. See~\cite[Lemma 2.4]{ari2018strict} for sufficient conditions that imply uniqueness of the principal eigenpair. From \cite[Lemma 2.2 and 2.3]{AB18}, the ERSC cost of the function $f$ is related to $\lambda^{*}_v[f]$ as follows
$$ \lambda_v^*[f]\leq \inf_{x\in \RR^d}\limsup_{T\to\infty} \frac{1}{T} \log\E_x^v\Big[\exp\Big(\int_0^T f(X_t)\D t\Big)\Big].$$
\cite[Theorem 1.4]{AB18} and \cite[Theorem 3.2]{ari2018strict} give sufficient conditions for the equality to hold in the above display. To be more precise, under the assumption of near-monotonicity of $f$ {relative to $\Lambda_v[f]$} or  $X$ being appropriately exponentially ergodic, the equality holds above.

We now state the assumptions made in this paper.
\begin{assumption}\label{a-main}
For some open set $\mathcal{K} \subset \RR^d$, the following hold:
\begin{itemize}
\item[(i)] the running cost $r$ is inf-compact on $\mathcal{K}$;
\item[(ii)] there exist constants $C_i>0$, $i=1,2,3$ with $C_3<1$,  and inf-compact functions 
$\sV\in \Cc^2(\RR^d)$ and ${\bar h}\in \cC(\RR^d\times \bU)$ 
such that  { $\cV\geq 1$ and}
\begin{equation}\label{eq-assump-1}
\begin{split}
\Lg^u \sV (x) &\,\le\, \bigl(C_1 - {\bar h}(x,u)\bigr)\sV (x)
\quad \forall (x,u) \in \mathcal{K}^c \times \Act\,,\\[5pt]
\Lg^u \sV(x) &\,\le\, \bigl(C_2 + C_3r(x,u)\bigr)\sV (x)
\quad \forall (x,u) \in \mathcal{K} \times \Act\,.
\end{split}
\end{equation}
\end{itemize}
\end{assumption}

\begin{remark}We remark that the assumption that $C_3<1$ is important in constructing an inf-compact function $h$ such that $r\in \srO(h)$ and the ERSC {cost} associated with $h$ (whenever the ERSC {cost} associated with $r$ is finite) is well-defined and also constructing the inf-compact perturbation to the running cost function using $h$. Since we are dealing with the ERSC problem, multiplicative constants play a crucial role, which is not the case with the CEC problem. \end{remark}

\begin{remark} Observe that when $\calK=\emptyset$, Assumption~\ref{a-main} reduces to the uniform stability assumption and when $\calK=\RR^d$, it reduces to the near-monotonicity assumption with inf-compactness on the running cost. 
 Therefore, Assumption~\ref{a-main} should be considered as a mixed condition that is ``in between" 
 the two frameworks, (a) blanket (uniform) stability of the controlled diffusion, {and} (b) near monotonicity in the case of an inf-compact running cost.
 \end{remark}
 {
 \begin{remark}Assumption~\ref{a-main} is analogous to the general structural hypothesis for CEC of diffusions and jump diffusions 
 studied in  \cite{ABP15,AP2016,APZ20}.
  In the CEC case,
 using the same notation as in Assumption~\ref{a-main}, the structural hypothesis is the following:  there exists an open set $\mathcal {K}\subset \RR^d$ such that 
\begin{itemize}
\item[(i)] the running cost $r$ is inf-compact on $\mathcal{K}$;
\item[(ii)] there exist  inf-compact functions 
$\sV\in \Cc^2(\RR^d)$ and ${\bar h}\in \cC(\RR^d\times \bU)$ 
such that 
\begin{equation}\label{eq-rem-assump-1}
\begin{split}
\Lg^u \sV (x) &\,\le\, 1 - {\bar h}(x,u)
\quad \forall (x,u) \in \mathcal{K}^c \times \Act\,,\\[5pt]
\Lg^u \sV(x) &\,\le\, 1 + r(x,u)
\quad \forall (x,u) \in \mathcal{K} \times \Act\,.
\end{split}
\end{equation}
\end{itemize}
The constants $1$ are chosen without compromising the generality as we can always scale the functions $\bar h$ and $\sV$ accordingly. Comparing with \eqref{eq-rem-assump-1}, the conditions in \eqref{eq-assump-1} are adapted to take into account the multiplicative nature of the HJB equation for the ERSC problem. 

 \end{remark}
 }

\begin{assumption}\label{a-well-defined}
There exists a $v^*\in \Usm$  such that $\Lambda_{v^*}[r]<\infty.$
\end{assumption}
\begin{remark} The above assumption implies  that the ERSC problem is well-defined. We remark that this is a necessary assumption 
even in either of two frameworks mentioned above. 
To see this, consider the case of uniform stability ($\calK=\emptyset$).
 If ${\bar h}(x,u)- r(x,u)$ is not inf-compact,  then we still cannot guarantee that Assumption~\ref{a-well-defined} holds. 
 In the case of near monotonicity with inf-compact running cost,
  it is obvious to see why Assumption~\ref{a-well-defined} is still necessary. In \cite{ABP15} where a similar problem in the context of the CEC problem is considered (see Assumption 3.2 of that paper), the CEC cost for an admissible control is assumed to be finite. However, in the our case, we have assumed that ERSC cost for a stationary Markov control is finite. This is needed because the existing results that we will be using \emph{viz.,} \cite[Proposition 1.3]{AB18} assume finiteness of the ERSC cost for some stationary Markov control. Therefore, such a requirement is enforced. 
\end{remark}

Since not every admissible control $U$ a priori gives rise to a finite ERSC cost, \emph{i.e.,} $J(x,U)[r]<\infty$, we define the various classes of controls $U$ for which $J(x,U)[r]<\infty$. From now on, we fix $\beta^*\doteq \Lambda_{v^*}[r]$ with $v^*\in\Usm$ given by Assumption~\ref{a-well-defined}. { For $\beta>\beta^*$, let
\begin{align}\label{def-u*}
\Uadm^{*,\beta}&\doteq \{U\in \Uadm: J(x,U)[r]\leq \beta, \text{ for some $x\in \RR^d$ }\}\end{align}
and $\Usm^{*,\beta}\doteq \Usm\cap \Uadm^{*,\beta}$.  From Assumption~\ref{a-well-defined}, $\Usm^{*,\beta}$  and $\Uadm^{*,\beta}$ are non-empty for $\beta>\beta^*$. 
The following relations are then evident: 
\begin{align*}
\Lambda[r]=\inf_{U\in \Uadm^{*,\beta}} J(x,U)[r]\leq\lsm[r]=\inf_{v\in \Usm^{*,\beta}} \Lambda_v[r]\,.
\end{align*}
Due to this, without loss of generality we only confine ourselves to $\Uadm^{*,\beta}$ or $\Usm^{*,\beta}$,} instead of $\Uadm$ or $\Usm$, respectively. Later, we will show that $\Lambda[r]=\lsm[r]$ (see Theorem~\ref{thm-diffusion}(iv)).

\begin{example}
Here we will give an example of the limiting controlled diffusion of a particular parallel server network, the ``W" network. 
{The ``W" network has three classes (denoted by $1,2,3$) of jobs and two server pools (denoted by $1,2$) with $\mu_{ij} $ for $i=1,2$ and $j=1,2,3$ being the limiting service rate of servers from pool $j$ when they serve customers from class $i$, and  $\lambda_i$ is the limit of appropriately scaled arrival rate of customers from class $i$.   In }the Halfin-Whitt regime, the limiting controlled diffusion $X$ in~\eqref{eqn-X}  has the following drift and diffusion coefficients (see a derivation in \cite{AP2016}): for  $l\in \RR^3$, 
$$ b(x,u)\doteq l-M_1\big(x-(e\cdot x)^+ u^c\big) + (e\cdot x)^- M_2u^s,\quad {\Sigma(x)=\diag(\sqrt{2\lambda_1},\sqrt{2\lambda_2},\sqrt{2\lambda_3})} \text{ where, }$$
$$ M_1=\begin{pmatrix}\mu_{11}&0&0\\
\mu_{22}-\mu_{21}&\mu_{22}&0\\
0&0&\mu_{32}\end{pmatrix},\, M_2=\begin{pmatrix} 0&0\\\mu_{21}-\mu_{22}&0\\ 0&0\end{pmatrix}$$
and $$ u\in \bU=\Big\{u=(u^c,u^s)\in \RR^3_+\times \RR^2_+: \langle e\cdot u^c\rangle=\langle e\cdot u^s\rangle =1\Big\}\,.$$ 

The running cost function is given by
$ r(x,u) = \sum_{i=1}^3c_i\big[( e\cdot x)^+ u_i^c\big]$ with $c_i>0$ (penalizing the queueing cost, but can also include idling cost $\sum_{j=1}^2 d_j \big[( e\cdot x)^- u_j^s\big]$ with $d_i>0$). 
Uniform stability of this controlled diffusion is an open question (see \cite{hmedi2022uniform} for the recent overview on the recent development on uniform stability for parallel server networks without abandonment).
However, if we define $ \cK=\cK_\delta\doteq \{x\in \RR^3: |(e\cdot x)|>\delta\|x\|\} \text{ with $\delta>0$}$, then it is clear that $r(\cdot,\cdot)$ restricted to $\cK$ is inf-compact. For a positive definite matrix $Q$, choose $g:\RR^3\rightarrow \RR$ to be a smooth function that agrees with $x\transp Qx$ on $B_1^c$. Finally, define $\cV_{Q}(x)=\exp\big(g(x)\big)$.  Following \cite[Lemma 3.1]{APZ20}, we can conclude that there exists a diagonal matrix $Q$, a small enough $\delta>0$, and a constant $C$ such that Assumption~\ref{a-main} is satisfied for $\cV=\cV_Q$, {$\bar h= C\|x\|^2$} with $\cK$ and $r$ chosen as above. 
To show that Assumption~\ref{a-well-defined} is satisfied, following \cite[Proposition 3.1]{hmedi2022uniform} we choose $v^*$ to be a constant control such that $u_3^c=1$ and $u^s_2=1$.
The theory can be applied to more general multiclass multi-pool networks, which we will study in a followup paper together with asymptotic optimality. 
\end{example}

\subsection{The main result}

In this section, we present the main results of the paper which include showing that the optimal cost $\Lambda[r]$ {is} attained by stationary Markov controls,
 establishing the well-posedness of the associated Hamilton-Jacobi-Bellman (HJB) equation and then characterizing the stationary optimal Markov controls.

\begin{theorem} \label{thm-diffusion}Under Assumptions~\ref{a-main} and~\ref{a-well-defined}, we have the following.
	\begin{enumerate}
	
		\item[(i)] The HJB equation
		\begin{equation} \label{eqn-HJB}
		\min_{u \in \Act} \bigl[\Lg^{u} V(x) + r(x,u)\,V(x)\bigr] \;=\; \lsm[r]\,V(x)
		\qquad\forall\,x\in\Rd
		\end{equation}
		has a unique positive solution  $V\in\Cc^{2}(\RR^{d})$,
		satisfying $V(0)=1$.
		\item[(ii)] A stationary Markov control $v$ is optimal \emph{i.e.,} $v\in \Usm^o$  if and only if $v$ satisfies 
		\begin{equation}\label{eqn-optimality1}
		\Lg^v V(x) + r^v(x)\,V(x)\;=\;
		\min_{u\in\Act}\; \bigl[\Lg^{u} V(x) + r(x,u)\,V(x)\bigr]
		\quad \text{a.e.\ }x\in\Rd\,. 
		\end{equation}
		
		\item[(iii)] 
		The function $V$  has the following stochastic representation
		\begin{equation}\label{eq-V*rep}
		V(x) =\E_{x}^{v}\Bigl[\exp\Big(\int_{0}^{\widecheck\tau_R}
		(r^v(X_{t})-\lsm[r])\,\D{t}\Big)\,V(X_{\widecheck\tau_R})\Bigr]
		\qquad\forall\, x\in \Bar{B}_R^c\,,
		\end{equation}
		for all $R>0$, and $v\in {\Usm^o}$. Additionally, if $v\in \Usm^{o}$ satisfies~\eqref{eq-V*rep}, for some $R>0$, then $\Lambda_v[r+f]>\Lambda_v[f]$, for all $f\in \sC_0$. 
		\item [(iv)] $\Lambda [r]=\lsm[r]$.
	\end{enumerate}

\end{theorem} 
\begin{remark}\label{rem-limit-lambda} Theorem~\ref{thm-diffusion}(iv) is new in the literature even under the assumptions of inf-compact running cost (when $\cK=\RR^d$). See \cite[Remark 1.3]{AB18}. In the case of uniform stability (when $\cK=\emptyset)$, using \cite[Theorem 4.1]{ari2018strict}, it is clear that $\Lambda[r]=\Lambda_{\text{SM}}[r]$. 
\end{remark}
{\begin{remark} We briefly discuss how the optimal ERSC cost is related to the optimal CEC cost associated with the controlled diffusion $X$ and the running cost $r$ with optimization over $U\in \Uadm$. To that end, for $0<\kappa\leq 1$, let $$\Lambda^\kappa[r] \doteq \inf_{x\in \RR^d}\inf_{U\in \Uadm} \limsup_{T\to\infty} \frac{1}{\kappa T} \log \E_x^U\left[ \exp\left( \kappa\int_0^T r({X}_t, U_t )\D t \right)\right]\,.$$
It is clear that the case of $\kappa=1$ corresponds to the  ERSC problem studied in this section. For $0<\kappa\leq 1$,  in the appendix we show that  for open set $\mathcal{K} \subset \RR^d$, inf-compact function $\bar h$ and positive constants $C_i$, $i=1,2,3$ (which are taken from Assumption~\ref{a-main}), and  $\cV^\kappa\doteq \exp\big(\kappa \log \sV(x)\big)$, we have that (a) the running cost $\kappa r$ is inf-compact on $\mathcal{K}$ and  (b) the following conditions hold:
\begin{equation}\label{eq-mixed-condition-rs}
\begin{split}
\Lg^u \sV^\kappa (x) &\,\le\, \bigl(\kappa C_1 - {\kappa\bar h}(x,u)\bigr)\sV^\kappa (x)
\quad \forall (x,u) \in \mathcal{K}^c \times \Act\,,\\[5pt]
\Lg^u \sV^\kappa(x) &\,\le\, \bigl(\kappa C_2 + C_3\kappa r(x,u)\bigr)\sV^\kappa (x)
\quad \forall (x,u) \in \mathcal{K} \times \Act\,.
\end{split}
\end{equation}
Namely, a version analogous to Assumption~\ref{a-main} holds. For  $\delta>0$, suppose that there exists $0<\kappa^\delta\leq 1$,  $v^\delta\in \Usm$ and $x^\delta \in \RR^d$ such that  
\begin{align}\label{eq-mixed-1}J^0(x^\delta, v^\delta)\leq \Lambda^0[r]+\delta\quad \text{ and } \quad  \limsup_{T\to\infty} \frac{1}{\kappa^\delta T} \log \E_{x^\delta}^{v^\delta}\left[ \exp\left( \kappa^\delta\int_0^T r({X}_t, v^\delta(X_t) )\D t \right)\right]<\infty\,.\end{align}
Then,  we have \begin{align}\label{eq-limit-lambda}\lim_{\kappa\to 0} \Lambda^\kappa[r]= \Lambda^0[r]\,.\end{align} 
In the above, for $U\in \Uadm$,
$$ J^0(x,U)[r] \doteq  \limsup_{T\to\infty} \frac{1}{ T}  \E_x^U\left[\int_0^T r({X}_t,U_t )\D t \right] \quad \text{ and } \quad \Lambda^0[r]\doteq \inf_{x\in \RR^d}\inf_{U\in \Uadm}  J^0(x,U)[r]  \,.$$
The proof of this result is provided in the appendix, for the sake of completeness. To the best of authors' knowledge, the result analogous to~\eqref{eq-limit-lambda} in the context of the finite horizon risk-sensitive case is well-known in the literature; see \cite{bauerle2023markov}. 
\end{remark}
}

The proof of this result is given in Section~\ref{sec-proof-main}. Below, we give an overview of the proof which includes the key ideas to be used. In a broad sense, we first construct and study ERSC problems associated with a certain  perturbed running cost $r^\veps$ such that $r^\veps$ is inf-compact and for any $U\in \Uadm$, the perturbed ERSC cost is finite whenever the original ERSC cost is finite. From here, using Assumption~\ref{a-well-defined}, we then proceed to show using the results of \cite{AB18} (in particular, Proposition 1.3 of that paper) that the original ERSC problem can be  completely solved in the sense that we have the well-posedness of the associated HJB equation and characterization of the optimal stationary Markov controls. 

Next, we move on to show that the limit of the optimal ERSC cost associated with $r^\veps$ is in fact, $\lsm[r]$ which is proved in Theorem~\ref{thm-pert-limit}. The proof  relies heavily on variational formulation of the ERSC cost for any admissible control which is introduced in Section~\ref{sec-var-BM}. The implication of using this formulation is that we can write the ERSC cost for any $v\in\Usm^*$  as the optimal cost of a new CEC problem associated with an extended process that involves an auxiliary control and an extended running cost that is the difference of the original running cost and a term that is quadratic in the auxiliary control.  After this, we prove the existence of solution to HJB by proving the convergence of solutions to the HJB equation associated with the perturbed ERSC problem. We then move on to prove the characterizations of optimal stationary Markov controls.

\medskip

\section{ERSC problem associated with perturbed running cost}\label{sec-est}
As mentioned earlier, the key ingredient of the proof of Theorem~\ref{thm-diffusion} involves studying the ERSC problem associated with a  perturbed running cost. In this section, we construct this perturbation, define the associated ERSC problem and state the existing results from literature. To that end,  define
\begin{align*}
 \cH&\doteq (\cK\times  \bU)\cup \left\{(x,u)\in \RR^d\times \bU: r(x,u)> {\bar h}(x,u) \right\}\,.\end{align*}
Then,   \cite[Lemma 3.3]{ABP15} gives us the  following. 
\begin{lemma}\label{lem-inf-comp}
There exists an inf-compact function $h:\RR^d\times \bU\rightarrow \RR_+$ such that 
\begin{align}\label{eq-inf-comp-1}   r(x,u)\leq &h(x,u)\leq 2+ 2{\bar h}(x,u)\Ind_{\cH^c}(x,u) +2  r(x,u) \Ind_{\cH}(x,u)
\,,\end{align}
and 
\begin{align}\label{eq-inf-comp-2}\Lg^u \sV(x)\leq \Big(C_1\wedge C_2-{\bar h}(x,u)\Ind_{\cH^c}(x,u)+C_3 r(x,u)\Ind_{\cH}(x,u)\Big)\sV(x)\,.
\end{align}
\end{lemma}
\begin{proof}
~\eqref{eq-inf-comp-1} follows directly from \cite[Lemma 3.3]{ABP15}.  To finish the proof,~\eqref{eq-inf-comp-2}  follows from the fact that $\cH\supset \cK\times \bU$.
\end{proof}
In the following, we use $h$ to construct a running cost function {that is inf-compact} over the entire space $\RR^d\times \bU$ (see~\eqref{def-r-cost-pert}). In the rest of the section, we use the inf-compact function $h$ from Lemma~\ref{lem-inf-comp} to perform the aforementioned construction of an inf-compact running cost whose associated ERSC problem is well defined. Before we state the next lemma, we give the well-known Young's inequality for the product of two non-negative real numbers: for $a,b\geq 0$ and $p,q>1$ such that $p^{-1}+q^{-1}=1$, we have
\begin{align}\label{eq-young-ineq}
ab\leq \frac{a^p}{p}+\frac{b^q}{q}
\end{align}
and equality holds if and only if $a^p=b^q$.
\begin{lemma}\label{lem-fin-cost} For $0<\veps<\frac{1-C_3}{4 }$ and {$U\in \Uadm^{*,\beta}$},  the following hold: 
\begin{align}\label{eq-b0}  \limsup_{T\to\infty} \frac{1}{T} \log \E_x^{U}\left[\exp\left(\int_0^T r(X_t,U_t)\Ind_{\calK\times \bU}(X_t,U_t)\D t\right)\right]\leq {\beta}\,,\end{align}
\item
\begin{align}\label{eq-b00} \limsup_{T\to\infty} \frac{1}{T} \log \E_x^{U}\left[\exp\left(\veps \int_0^T h(X_t,U_t)\D t\right)\right]\leq 2+\max\big\{C_1\wedge C_2, {\beta}\big\}\,.\end{align}
\end{lemma}
\begin{proof}Fix $U\in {\Uadm^{*,\beta}}$.~\eqref{eq-b0} follows immediately from { the definition of $\Uadm^{*,\beta}$ in~\eqref{def-u*}} and the fact that $r(x,u)\Ind_{\calK\times \bU}(x,u)\leq r(x,u)$
{which follows from the non-negativity of $r(\cdot,\cdot)$.} 

To prove~\eqref{eq-b00}, we apply  It{\^o}'s formula to 
$$ \exp\Big(\int_0^{{T\wedge\tau_R}} \Big({\bar h}(X_t,U_t)\Ind_{\calK^c\times \bU}(X_t,U_t)-C_3r(X_t,U_t)\Ind_{\calK\times \bU}(X_t,U_t)-C_1\wedge C_2\Big)\D t\Big) \sV(X_{T\wedge \tau_R})$$
to get
\begin{align*}
\sV(x)\geq \E_x^{U}&\Big[ \exp\Big(\int_0^{T\wedge \tau_R} \Big({\bar h}(X_t,U_t)\Ind_{\calK^c\times \bU}(X_t,U_t)\\
&-C_3r(X_t,U_t)\Ind_{\calK\times \bU}(X_t,U_t)-C_1\wedge C_2\Big)\D t\Big) \sV(X_{T\wedge \tau_R}) \Big]\,.
\end{align*}
Upon taking $R\to\infty$, using the fact that $\sV\geq 1$ and applying Fatou's lemma, we have
\begin{align*} \limsup_{T\to\infty} \frac{1}{T}\log \E_x^{U}&\left[\exp\left(\int_0^{T} \Big({\bar h}(X_t,U_t)\Ind_{\calK^c\times \bU}(X_t,U_t)-C_3r(X_t,U_t)\Ind_{\calK\times \bU}(X_t,U_t)\Big)\D t\right)\right]\\
&\leq C_1\wedge C_2\,.\end{align*}

For $0<\veps<\frac{1-C_3}{4}$ and large enough $T$, consider 
\begin{align*} \E_x^{U}&\left[\exp\left(\int_0^{T} \left(2\veps  {\bar h}(X_t,U_t)\Ind_{\calK^c\times \bU}(X_t,U_t)+2\veps  r(X_t,U_t)\Ind_{\calK\times \bU}(X_t,U_t)\right)\D t\right)\right]\\
&=  \E_x^{U}\Bigg[\exp\Big(\int_0^{T} \Big(2\veps {\bar h}(X_t,U_t)\Ind_{\calK^c\times \bU}(X_t,U_t)+2\veps r(X_t,U_t)\Ind_{\calK\times \bU}(X_t,U_t)\\
&\qquad\qquad-\frac{C_3}{2} r(X_t,U_t)\Ind_{\calK\times \bU}(X_t,U_t)+\frac{C_3}{2} r(X_t,U_t)\Ind_{\calK\times \bU}(X_t,U_t)\Big)\D t\Big)\Bigg]\,.
\end{align*}
Applying~\eqref{eq-young-ineq} with $$a= \exp\Big(\int_0^{T} \Big(2\veps  {\bar h}(X_t,U_t)\Ind_{\calK^c\times \bU}(X_t,U_t)-\frac{C_3}{2}r(X_t,U_t)\Ind_{\calK\times \bU}(X_t,U_t)\Big)\D t\Big)\,,$$ $$b=\exp\Big(\int_0^{T} \Big(\frac{C_3+4\veps }{2}r(X_t,U_t)\Ind_{\calK\times \bU}(X_t,U_t)\Big)\D t\Big)\,,  $$
and $p=q=2$, we get
\begin{align*}
 \E_x^{U}&\left[\exp\left(\int_0^{T} \left(2\veps {\bar h}(X_t,U_t)\Ind_{\calK^c\times \bU}(X_t,U_t)+2\veps  r(X_t,U_t)\Ind_{\calK\times \bU}(X_t,U_t)\right)\D t\right)\right]\\
 &\leq \frac{1}{2}\E_x^{U}\Bigg[\exp\Big(\int_0^{T} \Big(4\veps {\bar h}(X_t,U_t)\Ind_{\calK^c\times \bU}(X_t,U_t)-C_3r(X_t,U_t)\Ind_{\calK\times \bU}(X_t,U_t)\Big)\D t\Big)\Bigg]\\
 &\qquad+ \frac{1}{2}\E_x^{U}\Bigg[ \exp\Big(\int_0^{T} \Big((C_3+4\veps )r(X_t,U_t)\Ind_{\calK\times \bU}(X_t,U_t)\Big)\D t\Big) \Bigg]\,.
\end{align*}
Since $0<\veps < \frac{1-C_3}{4 }$ and $0<C_3<1$, it is also clear that $4\veps<1$. Therefore, we have
\begin{align}\nonumber
&\limsup_{T\to\infty}\frac{1}{T}\log \E_x^{U}\left[\exp\left(\int_0^{T} \left(2\veps {\bar h}(X_t,U_t)\Ind_{\calK^c\times \bU}(X_t,U_t)+2\veps  r(X_t,U_t)\Ind_{\calK\times \bU}(X_t,U_t)\right)\D t\right)\right]\\\nonumber
&\leq \max\Bigg\{ \limsup_{T\to\infty}\frac{1}{T} \log \E_x^{U}\Bigg[\exp\Big(\int_0^{T}\widehat h(X_t,U_t)\D t\Big)\Bigg], \\\nonumber
&\qquad\qquad\qquad \limsup_{T\to\infty}\frac{1}{T}\log \E_x^{U}\Bigg[ \exp\Big(\int_0^{T} \Big((C_3+4\veps)r(X_t,U_t)\Ind_{\calK\times \bU}(X_t,U_t)\Big)\D t\Big) \Bigg]\Bigg\}\\\nonumber
&\leq \max\Bigg\{ C_1\wedge C_2, \limsup_{T\to\infty} \frac{1}{T}\log \E_x^{U}\Bigg[ \exp\Big(\int_0^{T} \Big(r(X_t,U_t)\Ind_{\calK\times \bU}(X_t,U_t)\Big)\D t\Big) \Bigg]\Bigg\}\\\label{eq-b1}
&\leq \max\big\{ C_1\wedge C_2,{\beta}\big\}\,.
\end{align}
In the above, $\widehat h(x,u)\doteq   {\bar h}(x,u)\Ind_{\calK^c\times \bU}(x,u)-C_3r(x,u)\Ind_{\calK\times \bU}(x,u)$. To arrive at the final equation, we use {the} fact that $C_3+4\veps<1$ and~\eqref{eq-b0}. 
From the right hand side of~\eqref{eq-inf-comp-1}, we have~\eqref{eq-b00}. Finally, using~\eqref{eq-inf-comp-1}, this gives us
\begin{align*}
\limsup_{T\to\infty}\frac{1}{T}\log \E_x^{U}\left[\exp\left(\veps \int_0^{T} h(X_t,U_t)\D t\right)\right]\leq 2+ \max\big\{ C_1\wedge C_2,{\beta}\big\}\,.
\end{align*}
\end{proof}
\begin{corollary}\label{cor-inf-comp-finite}
For  $v\in {\Usm^{*,\beta}}$ and $0<\veps<\frac{1-C_3}{4}$, the following holds:
\begin{align}\label{eq-inf-comp-cost-markov} \limsup_{T\to\infty} \frac{1}{T} \log \E_x^{v}\left[\exp\left(\veps \int_0^T h(X_t,v(X_t))\D t\right)\right]\leq 2+ \max\big\{ C_1\wedge C_2, {\beta} \big\}\,. \end{align}
\end{corollary}
\begin{proof}
The proof is exactly along the same lines as that of~\eqref{eq-b00}. The only change is that we replace $U_t$ by $v(X_t)$.
\end{proof}

Let $0< \veps<\veps_0\doteq \frac{1-C_3}{8}$  and define
\begin{align}\label{def-r-cost-pert}r^\veps\doteq \Big(1-\frac{\veps}{\veps_0}\Big)r+ \veps h\,.\end{align} In the following, we study the ERSC problem associated with $r^\veps$, for $\veps\in (0,\veps_0)$. To that end, define 
$$J (x,U)[r^\veps]\doteq \limsup_{T\to\infty} \frac{1}{T} \log \E_x^U\left[ \exp\left( \int_0^T r^\veps({X}_t, U_t )\D t \right)\right]\,.$$
The lemma below shows that the ERSC problem associated with $r^\veps$ is well-defined, whenever $0<\veps<\veps_0.$

\begin{lemma}\label{lem-pert-well-defined} The following holds.
 \begin{align}\label{eq-pert-well-defined-1}\sup_{0\leq \veps<\veps_0}\sup_{U\in {\Uadm^{*,\beta}}}J(x,U)[r^\veps]&\leq 2+ \max\big\{C_1\wedge C_2,\beta \big\}\,.\end{align}
\end{lemma}

\begin{proof} 
Fix $U\in {\Uadm^{*,\beta}}$ and consider
\begin{align*}
\E_x^U\left[ \exp\left( \int_0^T r^\veps({X}_t, U_t )\D t \right)\right]= \E_x^U\left[ \exp\left( \int_0^T \left(\big(1-\frac{\veps}{\veps_0}\big)r({X}_t, U_t ) + \veps h(X_t,U_t)\right)\D t \right)\right]\,. 
\end{align*}
Applying~\eqref{eq-young-ineq} with
\begin{align*}
a&= \exp\big((1-\frac{\veps}{\veps_0}) \int_0^T r({X}_t, U_t )\D t \big), \quad
 b= \exp\big(\veps \int_0^T h(X_t,U_t)\D t \big),\\
p&= \big(1-\frac{\veps}{\veps_0}\big)^{-1} \quad \text{ and } \quad q= \frac{p}{p-1}= \frac{\veps_0}{\veps}\,,
\end{align*}
we get
\begin{align*}
J(x,U)[r^\veps]&\leq \max\Bigg\{ \limsup_{T\to\infty} \frac{1}{T} \log \E_x^U\left[ \exp\left( \int_0^T r({X}_t, U_t ) \D t \right)\right],\\
&\qquad \qquad\qquad  \limsup_{T\to\infty} \frac{1}{T}\log \E_x^U\left[ \exp\left(\veps_0\int_0^T  h(X_t,U_t)\D t \right)\right] \Bigg\}\\
&= \max\Bigg\{J(x,U)[r], \,\,\, \limsup_{T\to\infty} \frac{1}{T}\log \E_x^U\left[ \exp\left(\veps_0 \int_0^T   h(X_t,U_t)\D t \right)\right] \Bigg\}\,. 
\end{align*} 
From Lemma~\ref{lem-fin-cost}, taking supremum over $U\in {\Uadm^{*,\beta}}$ and then over $0<\veps<\veps_0$, we get the desired result.  \end{proof}

In light of Lemma~\ref{lem-pert-well-defined} {(which implies that  $  \Lambda[r^\veps]\leq \lsm[r^\veps]<\infty$)},   it is clear that $r^\veps$ is inf-compact and in particular, near-monotone relative to  
$ \Lambda[r^\veps] \text{ and } \lsm[r^\veps]$; {recall the definition of near-monotonicity in Definition~\ref{def-nm} and see Remark~\ref{rem-nm}.}
Note that Assumption~\ref{a-well-defined} and Lemma~\ref{lem-pert-well-defined} imply that $\Lambda[r^\veps]$ and $\lsm[r^\veps]$ are  finite for $0<\veps<\veps_0$.
We use  \cite[Proposition 1.3]{AB18} extensively which we now state below in the context of the ERSC problem associated with $r^\veps$.   Define 
$$\Usm^{o,\veps}\doteq \{v\in \Usm: \Lambda_v[r^\veps]=\lsm[r^\veps]\}\,.
$$
{From hereon, we always consider $\beta>2+ \max\big\{C_1\wedge C_2,\beta^* \big\}$. From the above lemma, this means that for $0\leq \veps<\veps_0$, $\Usm^{o,\veps},\Usm^{o}\subset \Usm^{*,\beta}$. } 
\begin{remark}
In the rest of the paper, we always assume that $0<\veps< \veps_0$. Also, to keep the expressions simple, we write $r^{\veps,v}(\cdot)=r^\veps(\cdot,v(\cdot)) $ for every $v\in {\Usm^{*,\beta}}$.
\end{remark}

Since  we are interested in studying the ERSC problem associated with $r^\veps$ for $0<\veps<\veps_0$ and then taking $\veps\to 0$, we state {relevant results in this case below}. 
\begin{theorem}\label{thm-HJB-pert} Suppose Assumptions~\ref{a-main} and~\ref{a-well-defined} hold. 
Then, there exists a {unique}  pair $(V^\veps, \widetilde \Lambda^\veps)\in \calC^2(\RR^d)\times \RR_+$  such that $V^\veps(0)=1$, $\inf_{x\in \RR^d}V^\veps(x)>0$ and 
\begin{align}\label{eq-HJB-pert}
\min_{u\in \bU} \Big\{\Lg^uV^\veps(x)+	r^\veps(x,u) V^\veps(x)\Big\}=\widetilde \Lambda^\veps V^\veps(x), \text{ for every $x\in \RR^d$ and $\lsm[r^\veps]=\widetilde \Lambda^\veps$.}
\end{align}
Moreover, the following hold. 
\begin{enumerate} \item [(i)]A stationary Markov control $v$  is optimal \emph{i.e.,} $v\in \Usm^{o,\veps}$ if and only if    it satisfies 
		\begin{equation}\label{eqn-optimality-pert}
		\Lg^v V^\veps(x) + r^{\veps,v}(x)\,V^\veps(x)\;=\;
		\min_{u\in\Act}\; \bigl[\Lg^{u} V^\veps(x) + r^\veps(x,u)\,V^\veps (x)\bigr]
		\quad \text{a.e.\ }x\in\Rd\,.
		\end{equation}
		\item[(ii)] 
		The function $V^\veps$ has the following stochastic representation
		\begin{equation}\label{eq-V*rep-pert}
		V^\veps(x) =\E_{x}^{v}\Bigl[\exp\Big(\int_{0}^{\widecheck\tau_R}
		(r^{\veps,v}(X_{t})-\lsm[r^\veps])\,\D{t}\Big)\,V^\veps(X_{\widecheck\tau_R})\Bigr]
		\qquad\forall\, x\in \Bar{B}_R^c\,,
		\end{equation}
		for all $R>0$, and $v\in \Usm^{o,\veps}$.
		Additionally, if $v\in\Usm^{o,\veps}$ {satisfies~\eqref{eq-V*rep-pert}} for some $R>0$,
		then $\Lambda_v[r^\veps+f] > \Lambda_v[r^\veps]=\lsm[r^\veps]$ for all $f\in\sC_0 $.
		\end{enumerate}
\end{theorem}

\begin{lemma}\label{lem-monotonicity}
Suppose Assumptions~\ref{a-main} and~\ref{a-well-defined} hold. Then, for any open ball $\frB$ and  $\delta>0$, we have $\Lambda_{\text{SM}}[r^\veps]< \Lambda_{\text{SM}}[r^\veps + \delta \Ind_\frB]$. 
\end{lemma}

\begin{proof} [Proof of Theorem~\ref{thm-HJB-pert}] In light of Lemma~\ref{lem-monotonicity}, the result {follows} from  \cite[Proposition 1.4]{AB18}.
\end{proof}

The proof of Lemma~\ref{lem-monotonicity}  (which is deferred to Section~\ref{sec-proof-pert}) involves an extensive use of variational formulation of ERSC problem. This formulation is the content of the next section. We end this section by giving an important uniform integrability {and a positive recurrence result} in the context of the ERSC problem corresponding to $r^\veps$.
\begin{lemma}\label{lem-ui}
For every $0<\veps<\veps_0$, there exists $\eta_\veps>0$ (depending only on $\veps$) such that 
\begin{align*}
\sup_{U\in {\Uadm^{*,\beta}}}\limsup_{T\to\infty} \frac{1}{T}\log \E_x^U\Big[\exp\Big((1+\eta_\veps)\int_0^T r^\veps (X_t,U_t)\D t\Big)\Big] \leq 2 + \max\{C_1\wedge C_2, {\beta}\}\,.
\end{align*}
\end{lemma}
\begin{proof}Fixing $U\in {\Uadm^{*,\beta}}$ and following exactly the same proof of Lemma~\ref{lem-pert-well-defined} with $r^\veps + \frac{\veps}{2} h$, we have 
\begin{align*}
\E_x^U&\left[ \exp\left( \int_0^T (r^\veps({X}_t, U_t )+\frac{\veps}{2} h(X_t,U_t))\D t \right)\right]\\
&= \E_x^U\left[ \exp\left( \int_0^T \left((1-\frac{\veps}{\veps_0})r({X}_t, U_t ) + \veps h(X_t,U_t) +\frac{\veps}{2} h(X_t,U_t)\right)\D t \right)\right]\,. 
\end{align*}
Applying~\eqref{eq-young-ineq} for 
\begin{align*}
 a= \exp\left((1-\frac{\veps}{\veps_0}) \int_0^T r({X}_t, U_t )\D t \right),&\quad
 b= \exp\left(\frac{3\veps}{2}\int_0^T h(X_t,U_t)\D t \right),\\
p= \frac{1}{1-\frac{\veps}{\veps_0}} \quad &\text{ and } \quad q= \frac{p}{p-1}= \frac{\veps_0}{\veps}\,,
\end{align*}
we get
\begin{align*}
J(x,U)[r^\veps+\frac{\veps}{2}h]&\leq \max\Bigg\{ \limsup_{T\to\infty} \frac{1}{T} \log \E_x^U\left[ \exp\left( \int_0^T r({X}_t, U_t ) \D t \right)\right]\,,\\
&\qquad \qquad  \limsup_{T\to\infty} \frac{1}{T}\log \E_x^U\left[ \exp\left(\frac{3\veps_0}{2} \int_0^T  h(X_t,U_t)\D t \right)\right] \Bigg\}\\
&= \max\Bigg\{J(x,U)[r],\,\,\,
 \limsup_{T\to\infty} \frac{1}{T}\log \E_x^U\left[ \exp\left(\frac{3\veps_0}{2} \int_0^T   h(X_t,U_t)\D t \right)\right] \Bigg\}\\
& \leq 2+ \max\{C_1\wedge C_2,{\beta}\}\,. 
\end{align*} 
To arrive at the last inequality, we use the fact that $U\in {\Uadm^{*,\beta}}$ and Lemma~\ref{lem-fin-cost}. We now claim that for some $\eta_\veps>0$, $(1+\eta_\veps)r^\veps\leq r^\veps +\frac{\veps}{2} h$. {Then, t}ogether with the above display, taking supremum over $U\in {\Uadm^{*,\beta}}$ will prove the result. The claim immediately follows as shown below:
\begin{align*}
 \frac{h}{r^\veps}= {\Big((1-\frac{\veps}{\veps_0}) \frac{r}{h}+ \veps\Big) ^{-1}}\geq {\Big(1-\frac{ \veps}{\veps_0} +\veps\Big)^{-1}}>0\,.
\end{align*}
In the above, we use the fact that $r\leq h$ (see~\eqref{eq-inf-comp-1}). The desired $\eta_\veps$ is $\frac{\veps }{2}{\big(1-\frac{ \veps}{\veps_0} +\veps\big)^{-1}}$, which completes the proof.
\end{proof}

From Corollary~\ref{cor-inf-comp-finite}, we can conclude that  $ {\Usm^{*,\beta}}\subset \Ussm$ (see the corollary below). 
We will make extensive use of the associated Foster-Lyapunov function. 

\begin{corollary}\label{cor-stable}
For  $v\in {\Usm^{*,\beta}}$, there exists an inf-compact $\sW_v\in W^{2,p}_{\text{loc}}(\RR^d)$, $p\geq d$  satisfying $\sW_v(0)=1$, $\inf_{x\in\RR^d}\sW_v(x)>0$ such that 
\begin{align}\label{eq-h-wv} \Lg^{v}\sW_v(x) +\veps_0 h^v(x) \sW_v(x)= \lambda^*_v [\veps_0 h^v]\sW_v(x)\,, \text{ for a.e. } x\in \RR^d\end{align}
with $h^v(\cdot)\doteq h(\cdot,v(\cdot))$ and {$\lambda^*_v[\veps_0 h(\cdot,v(\cdot))]= \Lambda_v[\veps_0 h(\cdot,v(\cdot))]<\infty$}. Moreover, $v\in \Ussm$ and $X$ is positive recurrent.
\end{corollary}
\begin{proof} Fix $v\in {\Usm^{*,\beta}}$.
 The existence of $\sW_v$ immediately follows from \cite[Lemma 3.1]{ari2018strict}.
  From \cite[Lemma 3.1]{ari2018strict}, we have $\lambda^*_v[\veps_0 h(\cdot,v(\cdot))]= \Lambda_v[\veps_0 h(\cdot,v(\cdot))]<\infty$. { Since we have 
  $$ \Lg^{v}\sW_v(x) = -\big(\veps_0 h^v(x)- \lambda^*_v [\veps_0 h^v]\big)\sW_v(x)\,, \text{ for a.e. } x\in \RR^d,$$
 inf-compactness} of $\veps_0 h^v-\lambda^*_v[\veps_0h^v]$ 
 and the fact that $\inf_{x\in\RR^d}\sW_v(x)>0$ implies that {$\sW_v$} is a Foster-Lyapunov function for $X$ under control $v\in {\Usm^{*,\beta}}$.  To infer that $X$ is positive recurrent, we use the above  display in conjunction with \cite[Theorem 2.6.10]{arapostathis2012ergodic}.   This completes the proof of the corollary. 
\end{proof} 
\begin{remark} 
Clearly, the above Foster-Lyapunov function is dependent on  $v\in \Usm^*$. For our purposes, this is sufficient as will be seen later.
\end{remark}
From hereon, Assumptions~\ref{a-main} and~\ref{a-well-defined} are enforced without further mention.

\medskip

\section{Variational formulation of ERSC problem}\label{sec-var-BM}
We now develop a variational formulation of the ERSC problem. 
The fundamental result we use is the  variational representation of exponential functionals of Brownian motion given below as stated in \cite[Theorem 5.1]{boue1998}. This was first introduced by  Fleming \cite{fleming1977exit}  to study certain large deviation problems; see also \cite{budhiraja2019analysis} for its extensive applications. To that end, define $\cA$ as the set of all $\cG_t$--progressively measurable functions $w:\RR_+\rightarrow \RR^d$ such that $T^{-1}\E\Big[\int_0^T\|w_t\|^2\D t\Big]<\infty$,  for every $T>0$.
 Here, $\cG_t$ is the filtration generated by $\{W_s:0\leq s\leq t\}$ such that $\cG_0$ includes all the $\PP$--null sets. { Recall that  $\frC_T^d$ denotes the set of $\RR^d$--valued continuous functions on $[0,T]$ equipped with uniform topology.}
\begin{lemma}\label{lem-var-rep-BM-gen} For $T>0$, suppose that $G:\frC^d_T\rightarrow \RR$ is a non-negative Borel measurable function. Then the following holds: 
	\begin{align}\label{eq-var-rep}
	\frac{1}{T}\log \E[e^{TG(W)}]= \sup_{w\in\cA} \E\bigg[ G\bigg(W+\int_0^{\cdot}w_t \D t\bigg) - \frac{1}{2T}\int_0^T\|w_t\|^2\D t\bigg]\,.
	\end{align}
	\end{lemma}

\begin{remark}\label{rem-any-bm} We note that Brownian motion $W$ in the right hand side of~\eqref{eq-var-rep} can be replaced by any $d$--dimensional Brownian motion $\widetilde W$ as long as the set $\cA$ (in~\eqref{eq-var-rep}) is defined with respect to the filtration of $\widetilde W$. However to avoid introducing extra notation, we restrict ourselves to the Brownian motion $W$ in the right hand side of~\eqref{eq-var-rep}.
 \end{remark}
 
 Lemma~\ref{lem-var-rep-BM-gen} has the most important consequence in the context of our ERSC problem. Before we state it, we introduce a new process $Z$ which from hereon referred to as the `extended'  process:  for every $v\in \Usm^{*,\beta}$ and $w\in \cA$ (which we refer to as auxiliary control from hereon)  with $\widetilde w_t\doteq \int_0^t w_s\D s$, the process $Z$ satisfies   
 \begin{align}\label{X-control-0} d{Z}_t=  b\big({Z}_t, v(Z_t)\big)\D t +\Sigma(Z_t) w_t\D t + \Sigma(Z_t) \D W_t,\,\, Z_0=x\,. \end{align}
 This extended diffusion has been also investigated in earlier works. See  for instance \cite{FM95} where this extended diffusion is used for both discounted and ergodic risk-sensitive costs, and  \cite{boue2001}, where a robust risk-sensitive escape problem is studied. See also  \cite{FM95,AB18, biswas2010risk,ABBK20,ari2018strict}, where 
 ERSC problem is studied but the so-called ``ground" diffusion is used (a special case of the extended diffusion with $U=v\in \Usm$ and  a particular choice of $w$ as a Markov control, obtained from the principal eigenfunction of the operator $\Lg^v$). 
  In addition, the authors of \cite{AB18, ari2018strict} investigate the problem of well-posedness for the ``ground" diffusion.     We however could not find a reference where the explicit analysis of the well-posedness of~\eqref{X-control}  is addressed at the level of generality that is needed in our analysis, albeit the analysis being standard. We therefore analyze this well-posedness of~\eqref{X-control} in the next two lemmas, for the sake of completeness.

 In the following, we first  address the question of existence of the process $Z$  under a certain moment condition on the process $w$. Recall Remark~\ref{rem-notation}. In the following, in regards to the notion of weak solution, we always denote the underlying Brownian motion by $W,$ to keep the notation simple.
\begin{lemma}\label{lem-u-e-ext} For $T>0$, $U\in \Uadm^{*,\beta}$ and $w\in \cA$,  of the form $w_t=w_t(W_{[0,t]})$, the equation 
\begin{align}\label{X-control} d{Z}_t=  b\big({Z}_t, U_t(W_{[0,t]}+ \widetilde w_{[0,t]})\big)\D t +\Sigma(Z_t) w_t\D t + \Sigma(Z_t) \D W_t  \end{align}
with $\widetilde w_t\doteq \int_0^t w_s\D s$ admits a unique weak solution on $[0,T]$.
\end{lemma}
\begin{remark} The proof follows very closely the arguments of the proof of \cite[Theorem 2.2.11]{arapostathis2012ergodic}. 
\end{remark}
\begin{proof} Since $w\in \cA$, we have 
\begin{align}\label{eq-E-fin} \E\Big[\int_0^T\|w_t\|^2\D t\Big]<\infty\,.\end{align}
Let $\overline W$ be a $d$--dimensional Brownian motion. Then, for every $U\in \Uadm^{*,\beta}$ the equation
 \begin{align}\label{X-cont-int} d{\overline Z}_t=  b\big({\overline Z}_t, U_t\big)\D t + \Sigma({\overline Z}_t) \D {\overline W}_t,\quad  {\overline Z}_0=x \end{align}
 admits a unique weak solution $\overline Z$ for an augmented (if needed) probability space $(\overline \Omega,\overline \cF,\overline \PP)$, according to \cite[Theorem 2.2.4]{arapostathis2012ergodic}.  For the sake of the rest of the proof, we set $\Omega= \frC([0,T],\RR^d)$, $\cF$ is the Borel $\sigma$-algebra of $\Omega$ and $\overline \PP$ is the law of $\overline Z$. Due to this construction, it is clear that the canonical process of $\frC([0,T],\RR^d)$ given by $Z_t(\omega)=\omega(t)$, for $\omega\in\frC([0,T],\RR^d)$,  has the same law as that of $\overline Z$.  Let $\cF_t$ be the natural filtration of $Z$ defined above.

Now suppose that $\int_0^T\|w_t\|^2 \D t\leq M$, for some $M>0$ and define a non-negative random variable:
$$ \Pi_t\doteq \exp\Big( \int_0^t \langle w_s,\D W_s\rangle -\frac{1}{2}\int_0^t \|w_s\|^2 \D s\Big)\,.$$ 
Since $\int_0^T\|w_t\|^2 \D t\leq M$, using Novikov's criterion (see \cite[Proposition 3.5.12]{karatzas1988brownian}), we can conclude that $\Pi_t$ is a martingale and $\E\big[\Pi_t\big]=1$. This means that the new measure $\PP $ that is defined through its restrictions $\PP_t$ on $(\Omega,\cF_t)$, for $t\geq 0$, \emph{via.}
$$ \frac{\D \PP_t	}{\D \overline \PP_t}=\Pi_t$$ 
is indeed, a probability measure and defines a consistent family of restrictions $\PP_t$. Here, $\overline \PP_t$ is the restriction of $\overline \PP$ to $(\Omega,\cF_t)$, for $t\geq 0$. From here, using the Girsanov's theorem (see \cite[Theorem 3.5.1]{karatzas1988brownian}), we can conclude that 
$$ W_t\doteq \overline W_t -\int_0^t w_s\D s$$
is $d$--dimensional Brownian motion under $\PP$. In other words, the process $Z$ (which is the canonical process of $(\Omega,\cF)$) satisfies
 \begin{align}\label{X-cont-int-1} d{Z}_t=  b\big({ Z}_t, U_t\big)\D t + \Sigma (Z_t) w_t\D t+ \Sigma({ Z}_t) \D { W}_t, \quad { Z}_0=x \,.\end{align}
This proves that there exists a filtered probability space which is $(\Omega,\cF,\cF_t, \PP)$ and a $d$--dimensional Brownian motion which is $W$ as defined above such that $Z$  satisfies~\eqref{X-cont-int-1} -- the existence of  weak solution (see \cite[Definition 5.3.1]{karatzas1988brownian}). The proof of uniqueness follows exactly the analogous arguments of uniqueness in the proof of \cite[Theorem 2.2.11]{arapostathis2012ergodic} and hence we omit it.  This proves  the lemma when $\int_0^T\|w_t\|^2 \D t$ is bounded.

Now we extend it to the case where $\int_0^T\|w_t\|^2 \D t$ is allowed to be unbounded. For $N\in \NN$, we define $w^N_t\doteq w_t \Ind_{[0,N]}(\int_0^t\|w_s\|^2\D s) $. Clearly, $$ \int_0^T \|w^N_t\|^2\D t\leq N$$ and $w^N$ belongs to the above case. From the above analysis, we already know that there exists a process $Z^N$ which is the unique weak solution to 
 \begin{align}\label{X-cont-int-2} d{Z}^N_t=  b\big({ Z}^N_t, U_t\big)\D t + \Sigma (Z^N_t) w^N_t\D t+ \Sigma({ Z}^N_t) \D { W}_t,\quad { Z}^N_0=x\,. \end{align}
 Note that Brownian motion $W$ can vary with $N$, but we suppress this dependence as it is not important for the analysis below. From the above construction, the law of $Z^N$ is given  by 
 $$ \PP^N(A) =\int_A \Pi^N_t(\omega) \D \overline \PP(\omega), \text{ for $A\in \cF_t$},$$
 where $ \Pi^N_t\doteq  \exp\Big( \int_0^t \langle w^N_s,\D W_s\rangle -\frac{1}{2}\int_0^t \|w^N_s\|^2 \D s\Big)$. For a stopping time $\tau^N\doteq \inf\{t\geq 0: \int_0^t\|w_s\|^2\D s>N\}$, the process $Z^N_{t\wedge \tau^N}$ satisfies 
  \begin{align}\label{X-cont-int-3} {Z}^N_{t\wedge \tau^N}= \int_0^{t\wedge \tau^N} b\big({ Z}^N_s, U_s\big)\D s + \int_0^{t\wedge \tau^N}\Sigma (Z^N_s) w_s\D s+ \int_0^{ t\wedge \tau^N}\Sigma({ Z}^N_s) \D { W}_s,\,\,\, { Z}^N_0=x\,. \end{align}

For $M>0$, define $ O_M(t)\doteq \big\{\xi\in \frC([0,T],\RR^d): \int_0^t\|w_s\|^2\D s\leq M \big\}\in \cF_t$. Then, for $N_1,N_2\in \NN$ such that $N_1\wedge N_2>M$, we have 
$$ \PP^{N_1}\big(A\cap O_M(t)\big)=\PP^{N_2}\big(A\cap O_M(t)\big)\,.$$
As $O_M(t)\in \cF_t$, we have $ \PP^{N_1}\big( O_M(t)\big)=\PP^{N_2}\big( O_M(t)\big)$ and this implies that $\PP^*\doteq \lim_{N\to\infty}\PP^N\big(O_M(t)\big)$ exists for all $M>0$. Moreover, 
$$ \PP^*\big(O_M(t)\big) = \lim_{N\to\infty}\int_{O_M(t)} \Pi^N_t(\omega) \D \overline \PP(\omega)= \int_{O_M(t)} \lim_{N\to\infty}\Pi^N_t(\omega) \D \overline \PP(\omega)= \int_{O_M(t)} \Pi^*_t(\omega) \D \overline \PP(\omega), $$
where $\Pi^*_t\doteq \exp\Big( \int_0^t \langle w_s,\D W_s\rangle -\frac{1}{2}\int_0^t \|w_s\|^2 \D s\Big)$.  From~\eqref{eq-E-fin} and   Markov's inequality, we have 
$$ \PP(O_M(t)^c)\leq \frac{\E\big[\int_0^T \|w_s\|^2\D s\big]}{M}\,.$$
Therefore, from the non-negativity of $\Pi^*_t$, we have 
$$ \PP^*(O_M(t))\geq \PP(O_M(t))=1-\PP(O^c_M(t))\geq 1-\frac{\E\big[\int_0^T \|w_s\|^2\D s\big]}{M}\,.$$
Taking $M\to\infty$, give us $ \liminf_{M\to\infty} \PP^*(O_M(t))\geq 1\,.$ In other words, $\E[\Pi^*_t]=1$. This proves that $\PP^*$ is indeed a probability measure and also proves  the existence of the weak solution. Again, uniqueness of this solution follows exactly along the same lines as the uniqueness in the proof of \cite[Theorem 2.2.11]{arapostathis2012ergodic}.  This completes the proof of the lemma. 
\end{proof}  
 Let $\Wadm$ be the set of admissible $\RR^d$--valued controls (here we use admissibility as in  Definition~\ref{def-adm-cont} with $\bU$ replaced by $\RR^d$) and we denote $\Wsm\subset \Wadm$ as the set of stationary Markov controls (including the relaxed controls). { For a relaxed Markov control $w=\mu(\D w|x)$, we set $\|w(x)\|\doteq  \int_{\RR^d} \|w\| \mu(\D w|x)$; it is clear that when $w$ is a precise Markov control, $\|w(x)\|$ reduces to the usual Euclidean norm.} For $l>0$, let $\Wsm(l)$ be the set of all stationary Markov controls including relaxed controls (which is a subset of $\Wadm$) which are  such that $\sup_{x\in \RR^d}\|w(x)\|\leq l$ and of the form $w=w(\cdot )$ on $B_l$ and $w=0$ on $B_l^c$.
\begin{lemma}\label{lem-u-e-ext-sm} For $U\in \Uadm^{*,\beta}$ and $w\in \Wsm$,  
there exists a unique $\frC([0,T],\RR^d)$--valued process ${ Z}^M$  that satisfies  
\begin{align*}
 {Z}^M_{t\wedge \tau^M} &=  \int_0^{t\wedge \tau^M}b\big({ Z}^M_s, U_s\big)\D s + \int_0^{t\wedge \tau^M}\Sigma({Z}^M_s)  {w}( Z_s)\D s+ \int_0^{t\wedge \tau^M}\Sigma({ Z}^M_s) \D { W}_s,
\end{align*}
where,  $ \tau^M\doteq \inf\{t\geq 0: \int_0^t \|w( Z^M_s)\|^2 \D s>M\}.$ 
Additionally,  if  $$\sup_{M>0}\E\Big[\int_0^T\|w( Z^M_t)\|^2 \D t\Big]<\infty,$$ for $T>0$,
then \eqref{X-control} admits a unique weak solution on $[0,T]$.
\end{lemma}
\begin{proof} For $M>0$ and $\xi\in \frC([0,T],\RR^d)$, define $  w^M (\xi_{[0,t]})\doteq w(\xi_t)\Ind_{[0,M]}(\int_0^t \|w(\xi_s)\|^2\D s).$ 
For $ w^M$, the equation 
 \begin{align}\label{X-cont-int-4} d{ Z}^M_t=  b\big({ Z}^M_t, U_t\big)\D t + \Sigma({ Z}^M_t)  { w}^M( Z^M_{[0,t]})\D t+ \Sigma({ Z}^M_t) \D { W}_t,\, { Z}^M_0=x \end{align}
 admits a unique weak solution, following the arguments of the proof of Lemma~\ref{lem-u-e-ext}.
  For the stopping time  $\tau^M$,  it is clear that
 \begin{align*}
 { Z}^M_{t\wedge \tau^M}&=  \int_0^{t\wedge \tau^M}b\big({ Z}^M_s, U_s\big)\D s + \int_0^{t\wedge \tau^M}\Sigma({ Z}^M_s)  { w}^M( Z^M_{[0.s]})\D s+ \int_0^{t\wedge \tau^M}\Sigma({ Z}^M_s) \D { W}_s\\
 &=  \int_0^{t\wedge \tau^M}b\big({ Z}^M_s, U_s\big)\D s + \int_0^{t\wedge \tau^M}\Sigma({ Z}^M_s)  {w}( Z^M_s)\D s+ \int_0^{t\wedge \tau^M}\Sigma({ Z}^M_s) \D { W}_s\,.
 \end{align*}
This proves the first part of the lemma. From here, following the arguments in the proof of Lemma~\ref{lem-u-e-ext}, gives us the second part of the lemma.
\end{proof} 
\begin{proposition}\label{prop-var-cost}For $U\in {\Uadm^{*,\beta}}$, the following statements hold: 
\begin{enumerate}
\item [(i)] Suppose that~\eqref{eqn-X} under $U$ has a unique strong solution. Then,
	\begin{align}\label{eqn-erg-cont-var-rep}
	J(x,U)[r]=  \limsup_{T\to\infty}\sup_{w\in\cA} \E_x^{U,w}\Bigg[ \frac{1}{T}\int_0^T \Big( r\big(Z_t, U_t(W_{[0,t]}+ \widetilde w_{[0,t]})\big) - \frac{1}{2}\|w_t\|^2\Big)\D t\Bigg]
	\end{align} 
with $\widetilde w_t\doteq \int_0^t w_s\D s$. 
Here, $Z_t$  is the unique weak solution to{~\eqref{X-control}.} 
Similarly, for $0<\veps<\veps_0$, 
\begin{align}\label{eqn-erg-cont-var-rep-pert}
	J(x,U)[r^\veps]= \limsup_{T\to\infty}\sup_{w\in\cA} \E_x^{U,w}\Bigg[ \frac{1}{T}\int_0^T \Big( r^\veps\big(Z_t, U_t(W_{[0,t]}+ \widetilde w_{[0,t]})\big) - \frac{1}{2}\|w_t\|^2\Big)\D t\Bigg]\,.
	\end{align} 
	In particular, the above statement holds for $v\in \Usm^{*,\beta}$. 
	\item [(ii)] Suppose that~\eqref{eqn-X} under $U$ is only known to have a unique weak solution. Then,~\eqref{eqn-erg-cont-var-rep} and~\eqref{eqn-erg-cont-var-rep-pert} still hold with $=$ replaced by $\geq$.
	\end{enumerate}
\end{proposition} 
\begin{proof} For part (i),  we provide the proof of~\eqref{eqn-erg-cont-var-rep} as the proof of~\eqref{eqn-erg-cont-var-rep-pert} can be argued similarly. For any $U\in \Uadm^{*,\beta}, $ it is clear that $U_t= U_t(W_{[0,t]})$ is a Borel measurable functional of $W$. From the conditions on $b$ and $\Sigma$ in the beginning of Section~\ref{sec-model} and \cite[Theorem 2.2.4]{arapostathis2012ergodic}, we can infer that the process $X$ is the unique strong solution to~\eqref{eqn-X}. This means that for $U_t=u$, we can express the process $X$ as follows: for every $t>0$,
$$ X_{t}= \mathscr{X}_t(W_{[0,t]},u),$$
for some measurable function $\mathscr{X}_t:\frC^d_t\times u\rightarrow \RR^d$. In the case where $U\in \Uadm^{*,\beta}$, the associated process $X$ can be  expressed as 
$$ X_{t}= \mathscr{X}_t\big(W_{[0,t]},U_t(W_{[0,t]})\big),$$
for $t>0$. In other words, the pair $(X,U)$ is a Borel measurable functional of $W$. This subsequently means that for every $T>0$,    $\frac{1}{T}\int_0^T r\big(X_t (W_{[0,t]}),U_t(W_{[0,t]})\big)\D t$
is also a Borel measurable functional of $W$. Hence, applying Lemma~\ref{lem-var-rep-BM-gen} to 
$$G(W)= \frac{1}{T}\int_0^T r\big(\mathscr{X}_t\big(W_{[0,t]},U_t(W_{[0,t]})\big),U_t(W_{[0,t]})\big)\D t,$$
we obtain 
\begin{align}\nonumber
&\frac{1}{T} \log \E\Big[\exp\Big(\int_0^T r(X_t,U_t)\D t\Big)\Big]\\\nonumber
&=\frac{1}{T} \log \E\Big[\exp\Big(\int_0^T r\big(\mathscr{X}_t \big(W_{[0,t]},U_t(W_{[0,t]})\big),U_t(W_{[0,t]})\big)\D t\Big)\Big]\\\label{eq-prop-var-1}
&= \sup_{w\in\cA} \E\bigg[ \frac{1}{T}\int_0^T r\big(\mathscr{X}_t \big(W_{[0,t]}+\widetilde w_{[0,t]},U_t(W_{[0,t]}+\widetilde w_{[0,t]})\big),U_t(W_{[0,t]}+\widetilde w_{[0,t]})\big)\D t - \frac{1}{2T}\int_0^T\|w_t\|^2\D t\bigg],
\end{align}
where $\widetilde w_t\doteq \int_0^t w_s\D s$. We now identify the process $\mathscr{X}_t \big(W_{[0,t]}+\widetilde w_{[0,t]},U_t(W_{[0,t]}+\widetilde w_{[0,t]})\big)$ on the right hand above: since the map $\mathscr{X}_t(W_{[0,t]},u)$ is a solution to~\eqref{eqn-X} with $U_t=u$, we can infer that the process $Z_t\doteq \mathscr{X}_t \big(W_{[0,t]}+\widetilde w_{[0,t]},U_t(W_{[0,t]}+\widetilde w_{[0,t]})\big)$ is a solution to
$$ d{Z}_t=  b\big({Z}_t, U_t(W_{[0,t]}+ \widetilde w_{[0,t]})\big)\D t +\Sigma(Z_t) w_t\D t + \Sigma(Z_t) \D W_t\,.$$
Therefore, in terms of the process $Z$,~\eqref{eq-prop-var-1} becomes 
\begin{align*}
\frac{1}{T} \log \E\Big[\exp\Big(\int_0^T r(X_t,U_t)\D t\Big)\Big]= \sup_{w\in\cA} \E\bigg[ \frac{1}{T}\int_0^T r\big(Z_t,U_t(W_{[0,t]}+\widetilde w_{[0,t]})\big)\D t - \frac{1}{2T}\int_0^T\|w_t\|^2\D t\bigg]\,.
\end{align*}
Now taking $T\to\infty$, proves the part (i).

To prove part (ii), we again only show the inequality analogous to~\eqref{eqn-erg-cont-var-rep} as the  inequality analogous to~\eqref{eqn-erg-cont-var-rep-pert} follows similarly.  Recall that there exists a measurable function (depending on $U_{\cdot}= v(\cdot,X_{[0,\cdot]})$; we suppress the dependence, however) $F: \frC^d\times [0,1] \rightarrow \frC^d$ such that 
$$ X_{\cdot}= F(W_{\cdot}, \vartheta )$$
with $\vartheta$ being a uniform random variable on $[0,1]$ that is independent of Brownian motion $W$.  Now consider 
$$ \widetilde G(W,\vartheta)\doteq \frac{1}{T}\int_0^T r\big(F_t(W_{[0,t]},\vartheta),U_t(F_t(W_{[0,t]},\vartheta))\big)\D t\,.$$
Then,
\begin{align*}
\frac{1}{T}\log \E_x^{U}\big[e^{T\widetilde G(W,\vartheta)}\big]&= \frac{1}{T}\log \int_0^1\E\big[e^{T\widetilde G(W,z)}\big] \D z\\ 
&= \frac{1}{T}\log \int_0^1e^{T \frac{1}{T}\log\E\big[e^{T\widetilde G(W,z)}\big]} \D z\\
&= \frac{1}{T}\log \int_0^1e^{T\sup_{w\in\cA} \E\big[ \widetilde G(W+\int_0^\cdot w_t\D t,z) - \frac{1}{2T}\int_0^T\|w_t\|^2\D t\big]} \D z\\
&\geq \int_0^1\sup_{w\in\cA} \E\bigg[ \widetilde G(W+\int_0^\cdot w_t\D t,z) - \frac{1}{2T}\int_0^T\|w_t\|^2\D t\bigg] \D z\\
&\geq \sup_{w\in\cA}  \int_0^1\E\bigg[ \widetilde G(W+\int_0^\cdot w_t\D t,z) - \frac{1}{2T}\int_0^T\|w_t\|^2\D t\bigg] \D z\\
&= \sup_{w\in\cA} \E_x^{U,w}\Bigg[ \frac{1}{T}\int_0^T \Big( r^\veps\big(Z_t, U_t(W_{[0,t]}+ \widetilde w_{[0,t]})\big) - \frac{1}{2}\|w_t\|^2\Big)\D t\Bigg]\,.
\end{align*}
In the above, we use Jensen's inequality to get the first inequality and we move the supremum outside the integral to obtain the second inequality. Finally, to obtain the last line above, we identify that the pair  $$\Big(W, F(W+\int_0^\cdot w_t\D t,\vartheta)\Big)$$ 
 is a weak solution to~\eqref{X-control}. This completes the proof of part (ii), after taking the limit as $T\to\infty$.
\end{proof}

We  refer to the process $Z$ defined above as the `extended' process under control $U$ and $w$, to make the distinction from the original process $X$ defined in~\eqref{eqn-X}.
\begin{remark}
In what follows, we always restrict ourselves to auxiliary controls $w$ that satisfy the hypotheses of either of Lemmas~\ref{lem-u-e-ext} or~\ref{lem-u-e-ext-sm}. For instance, see Lemmas~\ref{lem-comp-X-pert} and~\ref{lem-comp-X} of this section; In Section~\ref{sec-proof-pert}, we predominantly work with auxiliary controls $w\in \Wsm(l)$ (from the definition of $\Wsm(l)$ it is clear that $w$ then satisfies the hypothesis of Lemma~\ref{lem-u-e-ext-sm}). This in particular also includes Proposition~\ref{prop-2p-game} where a two-person zero-sum game (see~\eqref{def-rho-min-max} and~\eqref{def-rho-max-min} for its definition) is analyzed; In Section~\ref{sec-proof-main}, since we use the results from this section and Section~\ref{sec-proof-pert}, the auxiliary controls involved in that section automatically satisfy the hypotheses of Lemmas~\ref{lem-u-e-ext} or~\ref{lem-u-e-ext-sm}. Due to this reason, from hereon, we are not concerned with the problem of existence and uniqueness of the process $Z$. Also, in lieu of the above discussion, we simply work with process $Z$ without invoking/mentioning Lemmas~\ref{lem-u-e-ext} or~\ref{lem-u-e-ext-sm}.
\end{remark}

\begin{remark}In the above, the dependence on $U\in {\Uadm^{*,\beta}}$ and $w\in \cA$ is expressed only through $\E^{U,w}_x$ and to avoid confusion, we use $Z$ whenever the controls $w\in \cA$ and $U\in {\Uadm^{*,\beta}}$ are involved. We reserve $X$ whenever only $U\in {\Uadm^{*,\beta}}$ is involved. \end{remark} 
\begin{remark} It is important to note that the process $\big(Z_t,U_t(W_{[0,t]}+\widetilde w_{[0,t]})\big)$ on the right hand side of~\eqref{eqn-erg-cont-var-rep} is very different from the process $\big(X_t,U_t\big)$ involved in the definition of $J(x,U)[r]$. The main and the only difference being that the driving noise which is $W$ in the case of $\big(X_t,U_t\big)$ is replaced by $W+\int_0^\cdot w_s\D s$ in the case of $\big(Z_t,U_t(W_{[0,t]}+\widetilde w_{[0,t]})\big)$. To avoid cumbersome notation and lengthy expressions, we simply write $U_t(W_{[0,t]}+\widetilde w_{[0,t]})$ as $U_t$, whenever there is no confusion. 
\end{remark}

\begin{remark}\label{rem-non-negative} Even though we assumed that $U\in {\Uadm^{*,\beta}}$ {(instead of $U\in \Uadm$)} in  Proposition~\ref{prop-var-cost}, it is not necessary as Lemma~\ref{lem-var-rep-BM-gen} can always be applied for a non-negative Borel measurable $G$. In particular, we do not require the left hand side of~\eqref{eq-var-rep} to be finite. Consequently, we do not require $J(x,U)[r]$ and $J(x,U)[r^\veps]$ to be finite.   The biggest challenge that lies ahead is to switch
the `limsup' (in $T$) and `supremum' (over $\cA$) in~\eqref{eqn-erg-cont-var-rep} and~\eqref{eqn-erg-cont-var-rep-pert}. One can intuitively see that without certain uniform (in $T$) estimates, this is difficult to show. 
\end{remark}

In the next two subsections, we prove some crucial lemmas that give the aforementioned uniform (in $T$) estimates in the cases of the perturbed ERSC problem and the original ERSC problem. The proofs in the two cases are different - this owes to the fact that in the case of the perturbed ERSC problem, we have Lemma~\ref{lem-ui} and it is not clear if one can show that such an analogous result holds in the case of the original ERSC problem.

For $u\in \bU$ and $w\in \RR^d$, let $\widehat \Lg^{u,w} f(x)\doteq  \Lg^uf(x)+ (\Sigma(x) w)\cdot \nabla f(x)\,, \text{ for }x\in \RR^d $
	and for $v\in \Usm$, $\widehat \Lg^{v,w}f(x)\doteq \Lg^vf(x)+ (\Sigma(x) w)\cdot \nabla f(x)$.

\subsection{Key lemmas in  $0<\veps<\veps_0$ case}\label{sec-vepsgeq0}		
		 In the rest of the section, we fix $0<\veps<\veps_0$.  The next lemma can be regarded as the exponential analogue to uniform integrability. This helps us work with a truncation version of the integral of the running cost $r^\veps$. 
		 \begin{lemma}\label{lem-tail-est} The following holds:
	\begin{align*}
	\limsup_{L\to\infty}\sup_{U\in {\Uadm^{*,\beta}}}\limsup_{T\to\infty}\frac{1}{T}\log\E_x^U\left[\exp\left (\int_0^Tr^\veps(X_t,U_t) \D t\right) \Ind_{ [LT,\infty)}\Big(\int_0^T r^\veps(X_t,U_t)\D t\Big)\right]=-\infty\,.
	\end{align*}
	\end{lemma}

\begin{proof} Fix $U\in {\Uadm^{*,\beta}}$. For $L>0$,  define a random variable $$\cZ^\veps_T\doteq \exp\Big(\int_0^T r^\veps (X_t,U_t)\D t - LT\Big)\,.$$ 
From Lemma~\ref{lem-pert-well-defined}, we know that 
$$ \sup_{0\leq \veps<\veps_0}\sup_{U\in {\Uadm^{*,\beta}}} \limsup_{T\to\infty}\frac{1}{T}\log \E_x^U\Big[\exp\Big( \int_0^Tr^\veps (X_t,U_t)\D t\Big)\Big]\leq 2+ \max\{C_1\wedge C_2,\beta\}\,.$$
Using this and Lemma~\ref{lem-ui}, for large enough $T$, we clearly have 
\begin{align*}
e^{-LT} \E_x^U\Big[ \exp\Big( & \int_0^Tr^\veps(X_t,U_t)\D t\Big)\Ind_{[ LT,\infty)}\Big( \int_0^T r^\veps(X_t,U_t)\D t\Big)	\Big]= \E_x^U\Big[\cZ^\veps_T\Ind_{ [1,\infty)}(\cZ^\veps_T)\Big] \\
&\leq  \E_x^U\Big[ (\cZ^\veps_T)^{1+\eta_\veps }\Big]{=} e^{-(1+\eta_\veps)  LT} \E_x^U\Big[ \exp\Big((1+\eta_\veps) \int_0^Tr^\veps(X_t,U_t)\D t\Big)\Big]\,.
\end{align*}
Here, $\eta_\veps>0$ is the constant from Lemma~\ref{lem-ui}. This gives us
\begin{align*}
\sup_{U\in {\Uadm^{*,\beta}}}&\limsup_{T\to\infty}\frac{1}{T}\log\E_x^U\Big[\exp\Big( \int_0^T r^\veps(X_t,U_t)\D t\Big)\Ind_{[ LT,\infty)}\Big( \int_0^T r^\veps(X_t,U_t)\D t\Big)	\Big]\\
&\leq -\eta_\veps L +\sup_{U\in {\Uadm^{*,\beta}}} \limsup_{T\to\infty} \frac{1}{T}\log \E_x^U\left[\exp\left((1+\eta_\veps)\int_0^T r^\veps(X_t,U_t)\D t\right)\right]\\
&\leq-\eta_\veps L+ 2+\max\{C_1\wedge C_2,\beta\}\,.
\end{align*}
In the last inequality, we again use Lemma~\ref{lem-ui}. Now taking $L\uparrow \infty$, we have the desired result.
	\end{proof}
	
From the above lemma, we have the following very important corollary.
\begin{corollary}\label{cor-trunc-limit-L} For $U\in {\Uadm^{*,\beta}}$,  define 
$$J_L(x,U)[r^\veps]\doteq \limsup_{T\to\infty}\frac{1}{T}\log\E_x^U\left[\exp\left(\int_0^T r^\veps (X_t,U_t)\D t\right)\Ind_{[0,LT)}\Big( \int_0^T r^\veps (X_t,U_t)\D t\Big)	\right]\,. $$
Then the following holds: 
	\begin{align*}
	\lim_{L\to\infty} \sup_{U\in {\Uadm^{*,\beta}}}\Big|J(x,U)[r^\veps]- J_L(x,U)[r^\veps]\Big|=0\,.
	\end{align*}
	
\end{corollary}
\begin{proof} Fix $U\in {\Uadm^{*,\beta}}$. It is clear that for $L>0$, 
	\begin{align} \label{eq-trunc-low-bound}
	\limsup_{T\to\infty}\frac{1}{T}\log\E_x^U\left[\exp\left( \int_0^Tr^\veps (X_t,U_t)\D t\right) \Ind_{{[0,LT)}}\Big( \int_0^T r^\veps(X_t,U_t)\D t\Big)	\right]\leq J(x,U)[ r^\veps]\,.
	\end{align}
	It is also easy to see that 
	\begin{align}\nonumber 
	J(x,U)[ r^\veps]&\leq \max \Bigg\{ 	\limsup_{T\to\infty}\frac{1}{T}\log\E_x^U\left[\exp\left(\int_0^T r^\veps (X_t,U_t)\D t\right) \Ind_{[ 0,LT)}\Big( \int_0^T r^\veps (X_t,U_t)\D t\Big)	\right],\\\label{eq-trunc-up-bound}
	&\qquad\limsup_{T\to\infty}\frac{1}{T}\log\E_x^U\left[\exp\left( \int_0^T r^\veps (X_t,U_t)\D t\right)\Ind_{[ LT,\infty)}\Big(\int_0^T r^\veps (X_t,U_t)\D t\Big)	\right]\Bigg\}\,.
	\end{align}
	From here, taking supremum over $U\in {\Uadm^{*,\beta}}$, then $L\to \infty$ and using Lemma~\ref{lem-tail-est}, we have the desired result by combining~\eqref{eq-trunc-low-bound} and~\eqref{eq-trunc-up-bound}. 
	\end{proof} 
The following lemma states that there are nearly optimal controls whose {family of MEMs} is tight. 
	\begin{lemma}\label{lem-comp-w-pert}
	Suppose $U\in {\Uadm^{*,\beta}}$ and the hypothesis of Proposition~\ref{prop-var-cost}(i) holds. Then, for any  $\delta>0$, $T>0$, there exists $w^*= w^*(\delta,T,U)\in \cA$ such that   
	\begin{align}\label{eq-max-U} \sup_{w\in\cA} \E_x^{U,w}\Bigg[ \frac{1}{T}\int_0^T \Big(  r^\veps \big(Z_t, U_t\big) - \frac{1}{2}\|w_t\|^2\Big)\D t\Bigg]\leq  \E_x^{U,w^*}\Bigg[ \frac{1}{T}\int_0^T \Big(  r^\veps \big(Z_t, U_t\big) - \frac{1}{2}\|w^*_t\|^2\Big)\D t\Bigg]+\delta\,,\end{align}
	and  a constant $M_1=M_1(\veps,\delta,{\beta})>0$ such that 
	\begin{align}\label{eq-tight-1}\limsup_{T\to\infty} \frac{1}{T}\E_x^{U,w^*}\Big[\int_0^T \|w^*_t\|^2 \D t\Big]\leq M_1\,.\end{align}
	In particular, $M_1$ is independent of $U\in {\Uadm^{*,\beta}}$. As a consequence,   the family of MEMs of the process $w^*$ is tight. 
	\end{lemma}
	\begin{proof} Fix $\delta>0$. Using Corollary~\ref{cor-trunc-limit-L} and the fact that $U\in {\Uadm^{*,\beta}}$, we can choose $L=L(\veps,\delta,\beta)>0$ (independent of $U\in {\Uadm^{*,\beta}}$) such that 
	\begin{align}\nonumber &J(x,U)[ r^\veps]\\\nonumber
	&\leq  \limsup_{T\to\infty}\frac{1}{T}\log\E_x^{U}\left[\exp\left( \int_0^T r^\veps(X_t,U_t)\D t\right) \Ind_{[ 0,LT]}\Big(\int_0^T r^\veps (X_t,U_t)\D t\Big)	\right]+\frac{\delta}{2}\\\nonumber
	&= \limsup_{T\to\infty}\frac{1}{T}\log\E_x^{U}\left[\exp\left(\Big( \int_0^T r^\veps(X_t,U_t)\D t\Big)\wedge (LT)\right) \Ind_{[ 0,LT]}\Big(\int_0^T r^\veps (X_t,U_t)\D t\Big)	\right]+\frac{\delta}{2}\\\label{eq-exp-1}
	&\leq \limsup_{T\to\infty}\frac{1}{T}\log\E_x^{U}\left[\exp\left(\Big( \int_0^T r^\veps(X_t,U_t)\D t\Big)\wedge (LT)\right) 	\right]+\frac{\delta}{2} \,.\end{align}
	{To get the last line, we bound the indicator function by $1$.} Applying Lemma~\ref{lem-var-rep-BM-gen}, we have 
	\begin{align*}
	{\limsup_{T\to\infty}}&{\frac{1}{T}\log\E_x^{U}\left[\exp\left(\Big( \int_0^T r^\veps(X_t,U_t)\D t\Big)\wedge (LT)\right) 	\right]}\\
	&=\limsup_{T\to\infty}\sup_{w\in\cA} \E_x^{U,w}\Bigg[ \Big(\frac{1}{T}\int_0^T    r^\veps \big(Z_t, U_t\big)\D t\Big)\wedge L -\frac{1}{T}\int_0^T \frac{1}{2}\|w_t\|^2\D t\Bigg]\,.
	\end{align*}
	From above equation, for any $T>0$ and $w^*=w^*(\delta,T,U)$ such that
	\begin{align*}\sup_{w\in\cA} &\E_x^{U,w}\Bigg[ \Big(\frac{1}{T}\int_0^T  r^\veps \big(Z_t, U_t\big) \D t\Big)\wedge L- \frac{1}{2}\int_0^T \|w_t\|^2\D t\Bigg]\\
	&\leq  \E_x^{U,w^*}\Bigg[ \Big(\frac{1}{T}\int_0^T  r^\veps \big(Z_t, U_t\big) \D t\Big)\wedge L- \frac{1}{2}\int_0^T \|w^*_t\|^2\D t\Bigg]+\frac{\delta}{2}\,.\end{align*}
	 It is clear that 
	$$ \limsup_{T\to\infty}\E_x^{U,w^*}\Big[\frac{1}{T} \int_0^T \|w^*_t\|^2 \D t\Big]\leq 2(L+\delta)\,.$$ 
	 This completes the proof of~\eqref{eq-tight-1} with {$M_1(\veps,\delta, \beta)\doteq 2(L+\delta)$}. The fact that $M_1$ is independent of $U\in {\Uadm^{*,\beta}}$ follows from the fact that $L$ is independent of $U$. 
	 
	{ Finally, the tightness of the family of MEMs of the process $w^*$ follows from~\eqref{eq-tight-1} and an application of Markov's inequality. This completes the proof.}
	\end{proof}
	
	\begin{lemma}\label{lem-comp-X-pert} Suppose $U\in {\Uadm^{*,\beta}}$ and $w^*$ be as in Lemma~\ref{lem-comp-w-pert}. Then, 	there exists a constant $M_2=M_2(\veps, \delta,{\beta})>0$ such that 
		\begin{align*}
		\limsup_{T\to\infty} \frac{1}{T}\E_x^{U,w^*}\Big[\int_0^T h(Z_t,U_t)\D t\Big]\leq M_2\,.
		\end{align*}
		In particular, $M_2$ is independent of $U\in {\Uadm^{*,\beta}}$. As a consequence, the family of MEMs of the process $Z$ (under $U$ and $w^*$) is tight. 
		\end{lemma}

		\begin{proof}  
		
		Writing $\sV(x)= e^{ \frV(x)}$, we can easily see that 
		\begin{align*}
		\Lg^u\sV(x)=\Lg^ue^{\frV(x)}= e^{\frV(x)}\Lg^u \frV(x) + \frac{1}{2}e^{\frV(x)}\nabla \frV(x)\transp A(x)\nabla \frV(x)\,.
		\end{align*}
		From~\eqref{eq-inf-comp-2}, we know that 
		\begin{align*}
		\Lg^u \sV(x)\leq \Big(C_1\wedge C_2-{\bar h}(x,u)\Ind_{\cH^c}(x,u)+C_3r(x,u)\Ind_{\cH}(x,u)\Big)\sV(x)\,.
		\end{align*}
		In terms of $\frV$, the above display reduces to 
		\begin{align}\label{eq-lyap-var}
		\Lg^u \frV(x) + \frac{1}{2}\|\Sigma(x)\transp\nabla \frV(x)\|^2 \leq C_1\wedge C_2-{\bar h}(x,u)\Ind_{\cH^c}(x,u)+C_3r(x,u)\Ind_{\cH}(x,u)\,.
		\end{align}
		Since $ \widehat \Lg^{u,w} f= \Lg^uf+ \Sigma w\cdot \nabla f$ and $\Ind_{\cH}(x,u)\leq 1$,~\eqref{eq-lyap-var} becomes
		\begin{align}\label{eq-lyap-var-lin}
		\widehat \Lg^{u,w}\frV(x)\leq C_1\wedge C_2-{\bar h}(x,u)\Ind_{\cH^c}(x,u)+C_3r(x,u)-\frac{1}{2}\|\Sigma(x)\transp\nabla \frV(x)\|^2 + (\Sigma(x) w)\cdot \nabla \frV(x)\,.
		\end{align}
		Using the well-known inequality: $|x\cdot y|\leq \frac{\gamma}{2} \|x\|^2 + \frac{1}{2\gamma} \|y\|^2$, for $x,y\in \RR^d$ and $\gamma>0$ (which will be chosen later), we have
		\begin{align*}
		\widehat \Lg^{u,w}\frV(x) &\leq C_1\wedge C_2-{\bar h}(x,u)\Ind_{\cH^c}(x,u)+C_3r(x,u)-\frac{1}{2}\|\Sigma(x)\transp\nabla \frV(x)\|^2 \\ & \qquad + \frac{1}{ 2\gamma }\|w\|^2 + \frac{\gamma}{2}\|\Sigma(x) \transp\nabla \frV(x)\|^2\,.
		\end{align*}
		 Choosing $\gamma=\frac{1}{2}$, we have
		\begin{align}\label{eq-lyap-extended}
		\widehat \Lg^{u,w}\frV(x)&\leq C_1\wedge C_2-{\bar h}(x,u)\Ind_{\cH^c}(x,u)+C_3r(x,u) - \frac{1}{4}\|\Sigma(x)\transp\nabla \frV(x)\|^2   + { \|w\|^2}\, .
		\end{align}
		{For $U\in\Uadm^{*,\beta}$}, applying It{\^o}'s formula to $\frV(Z_{T\wedge \tau_R})$ with $w\equiv w^{*}_t$ and $u\equiv U_t$, we have 
		\begin{align}\label{eq-comp-1-x}
		\E_x^{U,w^{*}}\Big[\frV(Z_{T\wedge \tau_R})\Big]& \leq  \frV(x) + \E_x^{U,w^{*}}\Big[ \int_0^{T\wedge \tau_R} \Big(C_1\wedge C_2-{\bar h}(Z_{t}, U_t)\Ind_{\cH^c}(Z_{t}, U_t) \nonumber \\
		&+C_3r(Z_{t},U_t)-\frac{1}{4}\|\Sigma(Z_t)\transp\nabla \frV(Z_{t})\|^2   + {\|w^{*}_t\|^2}\Big)\D t\Big]\,.
		\end{align}
		Since $\frV\geq 0$, dividing by $T$ and taking $R\to\infty$, we have
		\begin{align*}
		\frac{1}{T}&\E_x^{U,w^{*}}\Big[\int_0^T{\bar h}(Z_{t}, U_t)\Ind_{\cH^c}(Z_{t}, U_t)\D t\Big]  + \frac{1}{4T} \E_x^{U,w^{*}}\Big[\int_0^T\|\Sigma(Z_t)\transp\nabla \frV(Z_{t})\|^2\D t \Big]\\\
		&\leq C_1\wedge C_2 +\frac{C_3}{T} \E_x^{U,w^{*}}\Big[\int_0^Tr(Z_{t},U_t)\D t\Big] +  {\frac{1}{ T}}
		\E_x^{U,w^{*}}\Big[ \int_0^T \|w^{*}_t\|^2\D t \Big]\,. 
		\end{align*}
		Adding on both sides the term $ T^{-1}\E_x^{U,w^{*}}\Big[\int_0^T r(Z_{t},U_t)\Ind_{\cH}(Z_{t}, U_t)\D t\Big],$
		 we have 
		\begin{align}\nonumber
		&\frac{1}{T}\E_x^{U,w^{*}}\Big[\int_0^T{\bar h}(Z_{t},U_t)\Ind_{\cH^c}(Z_{t}, U_t)\D t\Big]  + \frac{1}{4T}{ \E_x^{U,w^{*}}}\Big[\int_0^T\|\Sigma(Z_t)\transp\nabla \frV(Z_{t})\|^2\D t \Big] \nonumber\\
		& + \frac{1}{T}\E_x^{U,w^{*}}\Big[\int_0^T r(Z_{t},U_t)\Ind_{\cH}(Z_{t}, U_t)\D t\Big] \nonumber\\
		& \leq  \big(C_1\wedge C_2\big)  +\frac{C_3}{T} \E_x^{U,w^{*}}\Big[\int_0^Tr(Z_{t},U_t)\D t\Big] +\frac{1}{T}\E_x^{U,w^{*}}\Big[\int_0^T r(Z_{t},U_t)\Ind_{\cH}(Z_{t}, U_t)\D t\Big] \nonumber\\
		&\qquad+ {\frac{1}{ T}}
		\E_x^{U,w^{*}}\Big[ \int_0^T \|w^{*}_t\|^2\D t \Big] \nonumber \\\label{eq-comp-X-1}
		&\leq   \big(C_1\wedge C_2\big)  +\frac{C_3+1}{T} \E_x^{U,w^{*}}\Big[\int_0^Tr(Z_{t},U_t)\D t\Big]+ {\frac{1}{ T}}
		\E_x^{U,w^{*}}\Big[ \int_0^T \|w^{*}_t\|^2\D t \Big] \,. 
		\end{align}
		It is clear that from Proposition~\ref{prop-var-cost}(i),
		\begin{align}\label{eq-change-1}\limsup_{T\to\infty} \frac{1}{T} \E_x^{U,w^{*}}\Big[\int_0^Tr ^\veps (Z_{t},U_t)\D t\Big]\leq J(x,U)[r^\veps] +\limsup_{T\to\infty} \frac{1}{2T}\E_x^{U,w^*}\Big[ \int_0^T \|w_t^*\|^2 \D t\Big],\end{align}
		which from the definition of $r^\veps$ implies 
		\begin{align*}
		& \limsup_{T\to\infty}\frac{1}{T} \E_x^{U,w^{*}}\Big[\int_0^Tr (Z_{t},U_t)\D t\Big]\\
		 &\leq (1-\frac{\veps}{\veps_0})^{-1}\Bigg(J(x,U)[r^\veps] +\limsup_{T\to\infty}\Big( \frac{1}{2T}\E_x^{U,w^*}\Big[ \int_0^T \|w_t^*\|^2 \D t\Big]\\
		 &\qquad\qquad\qquad\qquad\qquad-\veps \frac{1}{T}\E_x^{U,w^*}\Big[\int_0^Th(Z_t,U_t)\D t\Big]\Big)\Bigg)\\
		 &\leq  (1-\frac{\veps}{\veps_0})^{-1}\Bigg(J(x,U)[r^\veps] + \limsup_{T\to\infty}\frac{1}{2T}\E_x^{U,w^*}\Big[ \int_0^T \|w_t^*\|^2 \D t\Big]\Bigg)\,.
		\end{align*} 
		From Lemma~\ref{lem-comp-w-pert},~\eqref{eq-inf-comp-1} and the last display,~\eqref{eq-comp-X-1} becomes 
	\begin{align*}
	\limsup_{T\to\infty}&\frac{1}{2T}\E_x^{U,w^{*}}\Big[\int_0^Th(Z_{t},U_t)\D t\Big] \\
	&\leq 1+ \big(C_1\wedge C_2\big) + \limsup_{T\to\infty}{\frac{1}{T}} \E_x^{U,w^{*}}\Big[ \int_0^T \|w^{*}_t\|^2\D t \Big]\\
		& \quad  +(C_3+1)(1-\frac{\veps}{\veps_0})^{-1} \Big(J(x,U)[r^\veps] + \limsup_{T\to\infty}\frac{1}{2T}\E_x^{U,w^*}\Big[ \int_0^T \|w_t^*\|^2 \D t\Big]\Big)\\
		&\leq 1+ \big(C_1\wedge C_2\big)  +(1+C_3)(1-\frac{\veps}{\veps_0})^{-1} \big({\beta} + {\frac{M_1}{2}}\big)+{M_1}\,.
	\end{align*}	
In the last equation, we use Lemma~\ref{lem-comp-w-pert} and the fact that {$J(x,U)[r^\veps]\leq \beta$}. With $M_2\doteq 2\big(1+ \big(C_1\wedge C_2\big)  +({C_3+1})(1-\frac{\veps}{\veps_0})^{-1} \big({\beta} + {\frac{M_1}{2}}\big)+{M_1}\big)$, we have the result.
		\end{proof}

\subsection{ Key lemmas in $\veps=0$ case.}\label{sec-e-0}

	Before we state the next lemma, we make the following simple observation. Since $w\in \cA$ if and only if $kw\in \cA$, for all $k\neq 0$,  for every $U\in {\Uadm}$ we have 
	\begin{align}\label{eq-equiv-cost} \sup_{w\in\cA} \E_x^{U,w}\Bigg[ \frac{1}{T}\int_0^T \Big( r\big(Z_t, U_t\big) - \frac{1}{2}\|w_t\|^2\Big)\D t\Bigg]=\sup_{w\in\cA} \E_x^{U,kw}\Bigg[ \frac{1}{T}\int_0^T \Big( r\big(Z_t, U_t\big) - \frac{k^2}{2}\|w_t\|^2\Big)\D t\Bigg]\,.\end{align}

	\begin{lemma}\label{lem-comp-X} For every $\delta>0$ and $v\in {\Usm^{*,\beta}}$,  there exists a $\widetilde w=\widetilde w(\delta,T)\in \cA$ such that 
	\begin{align}\label{def-d-opt}\sup_{w\in\cA} \E_x^{v,w}\Bigg[ \frac{1}{T}\int_0^T \Big( r^v\big(Z_t\big) - \frac{1}{2}\|w_t\|^2\Big)\D t\Bigg]\leq \E_x^{v,\widetilde w}\Bigg[ \frac{1}{T}\int_0^T \Big( r^v\big(Z_t\big) - \frac{1}{2}\|\widetilde w_t\|^2\Big)\D t\Bigg]+\delta\end{align}
	 	 and a constant $M_3={ M_3(\delta,\beta)}>0$ such that
		\begin{align}\label{eq-rec}
		\limsup_{T\to\infty} \frac{1}{T}\E_x^{v,\widetilde w}\Big[\int_0^T  h^v(Z_t)\D t\Big]\leq M_3\,.
		\end{align}
		As a consequence,  the family of MEMs of the process $Z$ (under $v$ and $\widetilde w$) is tight.
		\end{lemma}

		\begin{proof}  
		 Fix $v\in {\Usm^{*,\beta}}$ and let $\sW_v$ be the inf-compact function from Corollary~\ref{cor-stable}.  Writing $\widetilde \sW_v(x)\doteq \log \sW_v(x)$, we can easily see that 
		\begin{align*}
		\Lg^v\sW_v(x)=\Lg^ve^{\widetilde \sW_v(x)}= e^{\widetilde \sW_v(x)}\Lg^v \widetilde \sW_v(x) + \frac{1}{2}e^{\widetilde \sW_v(x)}\| \Sigma(x)\transp\nabla \widetilde \sW_v(x)\|^2\,.
		\end{align*}
		From here, using the definition of $\widehat \Lg^{u,w}$ and Corollary~\ref{cor-stable}, we have
		$$ \widehat \Lg^{v,w} \widetilde \sW_v(x) + \frac{1}{2}\| \Sigma(x)\transp\nabla \widetilde \sW_v(x)\|^2= \big(\lambda^*_v[\veps_0 h^v]- \veps_0h^v(x)\big)+ (\Sigma(x)w)\cdot \nabla\widetilde \sW_v(x)\,,$$
		for $w\in \RR^d.$
		Using the well-known inequality: $|x\cdot y|\leq \frac{\gamma}{2} \|x\|^2 + \frac{1}{2\gamma} \|y\|^2$, for $x,y\in \RR^d$ and $\gamma>0$, we have
		\begin{align*}
		\widehat \Lg^{v,w}{ \widetilde \sW_v}(x)&\leq \big(\lambda^*_v[\veps_0 h^v]- \veps_0h^v(x)\big) - \frac{(1-\gamma)}{2}\|\Sigma(x)\transp\nabla \widetilde \sW_v(x)\|^2   + \frac{1}{2\gamma} \|w\|^2\, .
		\end{align*}
		Now fix $\delta>0$ and $\gamma\doteq \frac{\veps_0}{2}$. From~\eqref{eq-equiv-cost}, we can choose $w^*=w^*(\delta,T)\in \cA$ such that 
		 $$ \sup_{w\in\cA} \E_x^{v,w}\Bigg[ \frac{1}{T}\int_0^T \Big( r^v\big(Z_t\big) - \frac{1}{2}\|w_t\|^2\Big)\D t\Bigg]\leq \E_x^{v,2\veps_0^{-1} w^*}\Bigg[ \frac{1}{T}\int_0^T \Big( r^v\big(Z_t\big) - \frac{2}{\veps_0^2}\|w^*_t\|^2\Big)\D t\Bigg] +\delta\,.$$
From the above display, we immediately have (along a subsequence  again denoted by $T$)
\begin{align}\label{eq-lem-comp-X-1} \frac{2}{\veps_0^2T}\E_x^{v,2\veps_0^{-1} w^*}\Big[ \int_0^T \|w^*_t\|^2\D t\Big] \leq  \frac{1}{T}\E_x^{v,2\veps_0^{-1} w^*}\Big[\int_0^T  r^v\big(Z_t\big)\D t\Big] -\Lambda_v[r]+ {\delta}\,.\end{align}		 
		Applying It{\^o}-Krylov's formula to $\widetilde \sW_v(Z_{T\wedge \tau_R})$, we have 
		\begin{align}\nonumber
		&\E_x^{v,2\veps_0^{-1}w^{*}}\Big[\widetilde \sW_v(Z_{T\wedge \tau_R})\Big] \\\nonumber
		& = \widetilde \sW_v(x) + \E_x^{v,2\veps_0^{-1}w^{*}}\Big[ \int_0^{T\wedge \tau_R} \Big(\lambda^*_v[\veps_0 h^v] -\veps_0 h^v(Z_t)-\frac{(1-\frac{\veps_0}{2})}{2}\|\Sigma(Z_t)\transp\nabla\widetilde \sW_v(Z_{t})\|^2   \Big)\D t\Big]\\\label{eq-comp-1-x2}
		&\qquad  +\frac{1}{\veps_0}\E_x^{v,2\veps_0^{-1}w^{*}}\Big[ \int_0^{T\wedge \tau_R} \|w^{*}_t\|^2\D t\Big]\,.
		\end{align}
		Since $ K_1\doteq \inf_{x\in \RR^d}\widetilde \sW_v(x)>-\infty$, dividing by $T$ and taking $R\to\infty$, we have
		\begin{align*}
		\frac{\veps_0}{T}&\E_x^{v,2\veps_0^{-1}w^{*}}\Big[\int_0^T h^v(Z_t)\D t\Big]  + \frac{(1-\frac{\veps_0}{2})}{2T} \E_x^{v,2\veps_0^{-1}w^{*}}\Big[\int_0^T\|\Sigma(Z_t)\transp\nabla\widetilde \sW_v(Z_{t})\|^2\D t \Big]\\\
		&\leq \lambda^*_v[\veps_0h^v]+\frac{\widetilde \sW_v(x)-K_1}{T}  +  \frac{1}{\veps_0T}
		\E_x^{v,2\veps_0^{-1}w^{*}}\Big[ \int_0^T \|w^{*}_t\|^2\D t \Big]\,. 
		\end{align*}
		Substituting~\eqref{eq-lem-comp-X-1}, we have
		\begin{align*}
		\frac{\veps_0}{T}&\E_x^{v,2\veps_0^{-1}w^{*}}\Big[\int_0^T h^v(Z_t)\D t\Big]  + \frac{(1-\frac{\veps_0}{2})}{2T} \E_x^{v,2\veps_0^{-1}w^{*}}\Big[\int_0^T\|\Sigma(Z_t)\transp\nabla\widetilde \sW_v(Z_{t})\|^2\D t \Big]\\\
		&\leq \lambda^*_v[\veps_0h^v]+\frac{\widetilde \sW_v(x)-K_1}{T}   + \frac{\veps_0 }{2T}\E_x^{v,2\veps_0^{-1} w^*}\Big[\int_0^T  r^v\big(Z_t\big)\D t\Big] -\frac{\veps_0}{2} \Lambda_v[r]+ {\frac{\veps_0 \delta}{2}}\,. 
		\end{align*}
		Using the fact that $r\leq h$, we can get the following from the last display:
		{\begin{align*}
		\frac{\veps_0}{2T}&\E_x^{v,2\veps_0^{-1}w^{*}}\Big[\int_0^T h^v(Z_t)\D t\Big]  + \frac{(1-\frac{\veps_0}{2})}{2T} \E_x^{v,2\veps_0^{-1}w^{*}}\Big[\int_0^T\|\Sigma(Z_t)\transp\nabla\widetilde \sW_v(Z_{t})\|^2\D t \Big]\\\
		&\leq \lambda^*_v[\veps_0h^v]+\frac{\widetilde \sW_v(x)-K_1}{T}    -\frac{\veps_0}{2} \Lambda_v[r]+ \frac{\veps_0 \delta}{2}\\
		&\leq \lambda^*_v[\veps_0h^v]+\frac{\widetilde \sW_v(x)-K_1}{T}   + \frac{\veps_0 \delta}{2}\\
		&\leq \Lambda_v[\veps_0h^v]+\frac{\widetilde \sW_v(x)-K_1}{T}   + \frac{\veps_0 \delta}{2}\\
		&\leq 2+ \max\big\{ C_1\wedge C_2, \beta \big\}+ \frac{\widetilde \sW_v(x)-K_1}{T}   + \frac{\veps_0 \delta}{2} \,.		\end{align*}}
		To get the fourth line, we use Corollary~\ref{cor-stable} and to get the last line we use Corollary~\ref{cor-inf-comp-finite}. Hence taking $T\to\infty$, we have the desired result with {$M_3= 2+ \max\big\{ C_1\wedge C_2, \beta \big\} + \frac{\veps_0 \delta}{2}$} and $\widetilde w=2\veps_0^{-1} w^*$.
				\end{proof}
				
From the previous lemma, we have the following immediate corollary. 		
		\begin{corollary}\label{cor-comp-w}
	For $\delta>0$ and $v\in {\Usm^{*,\beta}}$, let $\widetilde w\in \cA$ be as in the hypothesis of Lemma~\ref{lem-comp-X}. Then, 
	$$\limsup_{T\to\infty} \frac{1}{2T}\E_x^{v,\widetilde w}\Big[\int_0^T \|\widetilde w_t\|^2 \D t\Big]\leq M_3+\delta\,.$$
	{Here, $M_3$ is the constant from Lemma~\ref{lem-comp-X}.} In particular,  the family of MEMs of the process $\widetilde w$ is tight. 
	\end{corollary}
	\begin{proof} Fix $\delta>0$ and $v\in {\Uadm^{*,\beta}}$. From~\eqref{eq-lem-comp-X-1} and the fact that $r\leq h$ and using Lemma~\ref{lem-comp-X}, we have 
	\begin{align*} \frac{1}{2T}\E_x^{v,\widetilde  w}\Big[ \int_0^T \|\widetilde w_t\|^2\D t\Big] \leq  \frac{1}{T}\E_x^{{v},\widetilde  w}\Big[\int_0^T  r\big(Z_t, {v(Z_t)}\big)\D t\Big] + \delta
	\leq M_3 +\delta\,.
	\end{align*}
	 This completes the proof.	\end{proof}
	 The following gives us the estimates analogous to those in Lemma~\ref{lem-comp-X} and Corollary~\ref{cor-comp-w}, but for $r^\veps$ and importantly, these estimates are uniform in $\veps$. 
	 \begin{lemma}\label{lem-comp-X-w-pert}
	 For every $\delta>0$ and $v\in {\Usm^{*,\beta}}$,  there exists a {$\widetilde w^\veps=\widetilde w^\veps(T,\delta)\in \cA$} such that 
	\begin{align}\label{def-d-opt-pert}\sup_{w\in\cA} \E_x^{v,w}\Bigg[ \frac{1}{T}\int_0^T \Big( r^{\veps,v}\big(Z_t\big) - \frac{1}{2}\|w_t\|^2\Big)\D t\Bigg]\leq \E_x^{v,{\widetilde w^\veps}}\Bigg[ \frac{1}{T}\int_0^T \Big( r^{\veps,v}\big(Z_t\big) - \frac{1}{2}\|{\widetilde w^\veps_t}\|^2\Big)\D t\Bigg]+\delta\end{align}
	 	 and a constant {$M_4=M_4(\delta,\beta)>0$} such that
		\begin{align*}
		\sup_{0\leq\veps<\veps_0}\limsup_{T\to\infty} \frac{1}{T}\E_x^{v,{\widetilde w^\veps}}\Big[\int_0^T  h^v(Z_t)\D t\Big]\leq {M_4}\,,
		\end{align*}
		and
		\begin{align*}
		\sup_{0\leq\veps<\veps_0}\limsup_{T\to\infty} \frac{1}{2T}\E_x^{v,{\widetilde w^\veps}}\Big[\int_0^T \|{\widetilde w^\veps_t}\|^2 \D t\Big]\leq {M_4}+\delta\,.
		\end{align*}
	 \end{lemma}
	We omit the proof as it follows using the same arguments as those in the proofs of Lemma~\ref{lem-comp-X} and Corollary~\ref{cor-comp-w}.

	 \begin{lemma}\label{lem-exp-frv}
		The following relation holds for $\frV= \log \sV$, $v\in {\Usm^{*,\beta}}$ and $\widetilde w\in \cA$ as in the Lemma~\ref{lem-comp-X}:
		$$ \limsup_{T\to\infty}\frac{1}{T}\E_x^{v,\widetilde w}\Big[\frV(Z_T)\Big]\leq {M_5}\,,$$
		for some positive constant {$M_5=M_5(\delta,\beta)$}.
		\end{lemma}
		\begin{proof} {From the proof of Lemma~\ref{lem-comp-X}, we recall that $\widetilde w=2\veps_0^{-1} w^*$.}The proof begins by considering~\eqref{eq-comp-1-x}: 
		\begin{align*}
		\E_x^{v,\widetilde w}\Big[\frV(Z_{T\wedge \tau_R})\Big]& \le \frV(x) + \E_x^{v,\widetilde w}\Big[ \int_0^{T\wedge \tau_R} \Big(C_1\wedge C_2-{\bar h}(Z_{t}, v(Z_t))\Ind_{\cH^c}(Z_{t}, v(Z_t))\\
		&+C_3{r(Z_{t},v(Z_t))}-\frac{1}{4}\|\Sigma(Z_t)\transp\nabla \frV(Z_{t})\|^2   +  \|\widetilde w_t\|^2\Big)\D t\Big]\\
		&\leq  \frV(x) + \E_x^{v,\widetilde w}\Big[ \int_0^{T\wedge \tau_R} \Big(C_1\wedge C_2+C_3{r(Z_{t},v(Z_t))}   +  \|\widetilde w_t\|^2\Big)\D t\Big]\,.
		\end{align*}
		{Using Lemma~\ref{lem-comp-X} and Corollary~\ref{cor-comp-w},} dividing by $T$ and taking $R\to\infty$, we have
		\begin{align*}
		\limsup_{T\to\infty}\frac{1}{T}&\E_x^{v,\widetilde w}\Big[\frV(Z_T)\Big] \leq  {M_5}\,,
		\end{align*}for some constant {$M_5=M_5(\delta,\beta)>0$}.  	
		\end{proof}
		\begin{remark}\label{rem-pert-nearly-opt} 
		Lemma~\ref{lem-comp-X} and Corollary~\ref{cor-comp-w} imply that for any $v\in {\Usm^{*,\beta}}$ and $\delta>0$, there exists  $w\in \cA$ that is $\delta$--optimal for ${J(x,v)[r ]}$ such that  the family of  MEMs of the joint process {$(Z, v(Z), w)$ with $v(Z)\doteq v(Z_t)$} is tight.  
		\end{remark}	
				
			We end this section by introducing some notation that will be frequently used and results in concise expressions. By now, it is evident that for any $U\in {\Uadm^{*,\beta}}$, we are interested in analyzing 
		$$\limsup_{T\to\infty} \E_x^{v,w}\Bigg[ \frac{1}{T}\int_0^T \Big( r^\veps (Z_t,U_t) - \frac{1}{2}\|w_t\|^2\Big)\D t\Bigg] \text{ as $w$ varies over $\cA$}\,. $$
Therefore, we define
\begin{align*}
J(x,U,w,T)[r^\veps]&\doteq \E_x^{U,w}\Bigg[ \frac{1}{T}\int_0^T \Big( r^\veps(Z_t,U_t) - \frac{1}{2}\|w_t\|^2\Big)\D t\Bigg]\,,\\
\Lambda_{x,v,w,T}[r^\veps]&\doteq  \E_x^{v,w}\Bigg[ \frac{1}{T}\int_0^T \Big( r^{\veps,v}\big(Z_t\big) - \frac{1}{2}\|w_t\|^2\Big)\D t\Bigg] \,, 
\end{align*}	
and
\[
 J(x,U,w)[r^\veps] \doteq  \limsup_{T\to\infty} J(x,U,w,T)[r^\veps]\quad \text{and}\quad \Lambda_{x,v,w}[r^\veps] \doteq \limsup_{T\to\infty} \Lambda_{x,v,w,T}[r^\veps]\,. 
\]
Here, the process $Z$ is as defined in~\eqref{X-control}. Similarly, we define $J(x,U,w)[r]$ and $\Lambda_{x,v,w}[r]$. However, we drop $x$ (as it turns out to be irrelevant) and write $J(U,w)[r]$ and $\Lambda_{v,w}[r]$. At few places (see~\eqref{eq-inst-1} for instance), we encounter expressions that are similar to 
$$  \limsup_{T\to\infty}\sup_{w\in \cA} \frac{1}{T}\E_x^{v,w}\Big[\int_0^T\Big(r^v (Z_t)-\frac{1}{2}\|w_t\|^2\Big)\D t\Big]$$
where, the limit superior and supremum operations appear in the order given above. In such instances, we explicitly give the full expression to avoid any confusion.

\medskip

\section {Analysis of the perturbed ERSC problem}\label{sec-proof-pert}
In this section, we state and prove all the necessary results for the perturbed ERSC problem that are later useful in studying the limiting behavior  as $\veps \to 0$ in Section~\ref{sec-proof-main}.  {For $v\in {\Usm^{*,\beta}}$,  we show that $\Lambda_v[r^\veps]$ can be} represented as the optimal cost of the associated CEC problem for the extended diffusion $Z$.  To that end, we recall the notion of ergodic occupation measure from \cite[Section 3.2.1]{arapostathis2012ergodic}. A measure $\pi_v\in \calP(\RR^d\times \RR^d)$ is said to be an ergodic occupation measure associated with $v\in {\Usm^{*,\beta}}$, if for every $f\in \cC^2_c(\RR^d)$, the following holds: $ \int_{\RR^d\times \RR^d}\widehat \Lg^{v,w}f(x) \pi_v(\D x,\D w)=0$.
The set of ergodic occupation measures associated with $v\in {\Usm^{*,\beta}}$ is denoted by $\frG_v$. For $l>0$, $\frG^l_v$ be the set of ergodic occupation measures $\pi_v$ of the form $\pi_v(\D x,\D w)=\eta_v (\D x)\mu(\D w|x)$  with $\eta_v$ being the invariant measure of $Z$ under $v$ and relaxed auxiliary stationary Markov control $\mu(\D w|x)$ such that  
{\begin{enumerate}
\item [(i)] for every $x\in B_l^c$,  $\mu(\D w|x)\equiv \delta_{0}(\D w)$ with $\delta_z$ being the Dirac delta measure at $z\in \RR^d$,
\item [(ii)] $\int_{\RR^d}\|w\|\mu(\D w|x)\leq l$, for every $x\in B_l$.
\end{enumerate}}
		\begin{lemma}\label{lem-cec-cost}
		For $v\in {\Usm^{*,\beta}}$, the following statements hold.
		\begin{enumerate}
		\item[(i)] For every $\delta>0$ and $L>0$, there exists $l=l(L,\delta,{M_4})>0$ such that 
		$$ \Lambda_v[r^\veps] \leq \sup_{\pi\in \frG^v_l}\int_{\RR^d\times \RR^d} \Big(r^{\veps,v}(x)\wedge L -\frac{1}{2}\|w\|^2\Big)\D \pi(x,w) +\delta\,.$$
		
		\item [(ii)] $$ \Lambda_v[r^\veps] \leq \sup_{\pi\in \frG^v}\int_{\RR^d\times \RR^d} \Big(r^{\veps,v}(x) -\frac{1}{2}\|w\|^2\Big)\D \pi(x,w)\,.$$
		\end{enumerate}
		\end{lemma}
		\begin{proof} Fix $\delta>0$. {From Corollary~\ref{cor-trunc-limit-L}, there exists $L= L(\delta)>0$ such that 
		\begin{align*}\Lambda_v[ r^\veps]&\leq  \limsup_{T\to\infty}\frac{1}{T}\log\E_x^{v}\left[\exp\left( \int_0^T r^{\veps,v}(X_t)\D t\right) \Ind_{[ 0,LT]}\Big(\int_0^T r^{\veps,v} (X_t)\D t\Big)	\right]+\frac{\delta}{3}\\
		&=  \limsup_{T\to\infty}\frac{1}{T}\log\E_x^{v}\left[\exp\left(\Big( \int_0^T r^{\veps,v}(X_t)\D t\Big)\wedge (LT)\right) \Ind_{[ 0,LT]}\Big(\int_0^T r^{\veps,v} (X_t)\D t\Big)	\right]+\frac{\delta}{3}\\
		&\leq \limsup_{T\to\infty}\frac{1}{T}\log\E_x^{v}\left[\exp\left(\Big( \int_0^T r^{\veps,v}(X_t)\D t\Big)\wedge (LT)\right) 	\right]+\frac{\delta}{3}\\
		&=\limsup_{T\to\infty}\sup_{w\in\cA} \E_x^{v,w}\Bigg[ \Big(\frac{1}{T}\int_0^T    r^{\veps,v} \big(Z_t\big)\D t\Big)\wedge L -\frac{1}{T}\int_0^T \frac{1}{2}\|w_t\|^2\D t\Bigg] +\frac{\delta}{3} \,.\end{align*}
	To get the last line, we use Proposition~\ref{prop-var-cost}(i)}. 
	Choose $w^*=w^*(\delta, T)\in \cA$ such that 
	\begin{align}\nonumber  \sup_{w\in\cA} &\E_x^{v,w}\Bigg[ \Big(\frac{1}{T}\int_0^T    r^{\veps,v} \big(Z_t\big)\D t\Big)\wedge L -\frac{1}{T}\int_0^T \frac{1}{2}\|w_t\|^2\D t\Bigg]\\\nonumber
	&\leq \E_x^{v,w^*}\Bigg[ \Big(\frac{1}{T}\int_0^T    r^{\veps,v} \big(Z_t\big)\D t\Big)\wedge L -\frac{1}{T}\int_0^T \frac{1}{2}\|w^*_t\|^2\D t\Bigg]+{ \frac{\delta}{3}}\\\label{eq-exp-2}
	&\leq \E_x^{v,w^*}\Bigg[ \frac{1}{T}\int_0^T  \Big(  r^{\veps,v} \big(Z_t\big)\wedge L -\frac{1}{2}\|w^*_t\|^2\Big)\D t\Bigg]+{ \frac{\delta}{3}}\,.\end{align}
	
	{ We now analyze the first term on the right hand side of the above display.}
	From Lemma~\ref{lem-comp-X-w-pert}, we can conclude that there exists an ergodic occupation measure $\pi^*\in \calP(\RR^d\times \RR^d)$ such that MEMs $\pi_T$ of $(Z_{[0,T]},w^*_{[0,T]})$ converge weakly to $\pi^*(\D x,\D w)= \eta^*_v(\D x)\mu^*(\D w|x)$, along a subsequence $T_k$. This means that 
	$$ \limsup_{T_k\to\infty} \E_x^{v,w^*}\Bigg[ \frac{1}{T_k}\int_0^{T_k} \Big(  r^{\veps,v} \big(Z_t\big)\wedge L - \frac{1}{2}\|w^*_t\|^2\Big)\D t\Bigg]\leq \int_{\RR^d\times \RR^d}\Big( r^{\veps,v}(x)\wedge L-\frac{1}{2}\|w\|^2\Big) \D \pi^*(x,w)\,.$$
	To summarize, we have shown that 
	\begin{align}\label{eq-occu-meas}
	 \Lambda_v[r^\veps] \leq  \int_{\RR^d\times \RR^d}\Big( r^{\veps,v}(x)\wedge L-\frac{1}{2}\|w\|^2\Big) \D \pi^*(x,w) +{\frac{2\delta}{3}}\,.
	\end{align}
	{ However, it is not clear if $\pi^*$ obtained above lies in $\frG^v_l$. Below we construct $\widetilde w^*\in \cA$ that is also nearly optimal and is such that the family of MEMs of $(Z_{[0,T]},w^*_{[0,T]})$ is tight and the limit points lie in $\frG^v_l$. To that end, we} define  $\widetilde w^*\in \cA$ as follows: let $$ \tau_l^*\doteq \inf\{t>0: \|w^*_t\|>l \text { or } \|Z_t\| >l\}\, \text{ and } \widetilde w^*_t \doteq w^*_t \Ind_{[0,\tau_l^*]}(t)\,.$$ 
Here, $w^*$ is as chosen above.	It is trivial to see that whenever either $\|w^*_t\|>l$ or $\|Z_t\|>l$, $\widetilde w^*_t=0$ and for $t\leq \tau^*_l$, $\widetilde w^*_t=w^*_t$. This consequently implies that for $t\leq \tau_l^*$, $Z_t$ is identical under both the pairs $v,w^*$ and $v,\widetilde w^*$.  We now show that for large enough $l$, $\widetilde w^*$ is also nearly optimal. To that end, observe that 
	\begin{align}\label{eq-greater-cost} \E_x^{v,\widetilde w^*}\Bigg[\frac{1}{T}\int_0^T \frac{1}{2}\|\widetilde w^*_t\|^2\D t\Bigg]\leq \E_x^{v,w^*}\Bigg[ \frac{1}{T}\int_0^T \frac{1}{2}\|w^*_t\|^2\D t\Bigg]\,.\end{align}
	\begin{align*} 
	&\E_x^{v,\widetilde w^*}\Bigg[ \frac{1}{T}\int_0^T  \Big(  \big(r^{\veps,v} \big(Z_t\big)\wedge L\big) -\frac{1}{2}\|\widetilde w^*_t\|^2\Big)\D t\Bigg]-\E_x^{v,w^*}\Bigg[ \frac{1}{T}\int_0^T  \Big(  \big(r^{\veps,v} \big(Z_t\big)\wedge L\big) -\frac{1}{2}\|w^*_t\|^2\Big)\D t\Bigg]\\&\geq \E_x^{v,\widetilde w^*}\Bigg[ \frac{1}{T}\int_0^T  \big(r^{\veps,v} \big(Z_t\big)\wedge L\big) \D t\Bigg]-\E_x^{v,w^*}\Bigg[ \frac{1}{T}\int_0^T   \big(r^{\veps,v} \big(Z_t\big)\wedge L\big) \D t\Bigg],
	\end{align*} 
	where we use~\eqref{eq-greater-cost} to get the inequality. Since $Z$ under $v,w^*$ and $v,\widetilde w^*$ are identical for $t\in [0,\tau^*_l]$, we subsequently have
	\begin{align*}
	&\E_x^{v,\widetilde w^*}\Bigg[ \frac{1}{T}\int_0^T  \Big( \big( r^{\veps,v} \big(Z_t\big)\wedge L\big) -\frac{1}{2}\|\widetilde w^*_t\|^2\Big)\D t\Bigg]-\E_x^{v,w^*}\Bigg[ \frac{1}{T}\int_0^T  \Big(  \big(r^{\veps,v} \big(Z_t\big)\wedge L\big) -\frac{1}{2}\|w^*_t\|^2\Big)\D t\Bigg]\\
	&\geq  \frac{1}{T}\int_0^T  \E_x^{v,\widetilde w^*}\Bigg[ \big(r^{\veps,v} \big(Z_t\big)\wedge L\big) \Ind_{[\tau^*_l,\infty)}(t) \Bigg]\D t-  \frac{1}{T}\int_0^T\E_x^{v,w^*}\Bigg[   \big(r^{\veps,v} \big(Z_t\big)\wedge L\big)\Ind_{[\tau^*_l,\infty)}(t)\Bigg] \D t\\
	&\geq  -\frac{2L}{T}\int_0^T\PP(\tau^*_l<t)\D t= -\frac{2L}{T} \E_x^{v,w^*}\Big[\int_0^T \Ind_{[\tau^*_l,\infty)}(t)\D t\Big]\,.
	\end{align*}
	From Lemma~\ref{lem-comp-X-w-pert}, we can choose $l=l(L,\delta,{M_4})$ large enough such that $$2LT^{-1}\E_x^{v,w^*}\Big[\int_0^T \Ind_{[\tau^*_l,\infty)}(t)\D t\Big]<{\frac{\delta}{3}},$$ uniformly in $T$. Therefore, we have shown that 
	$$ \Lambda_v[r^\veps]\leq \limsup_{T\to\infty}\E_x^{v,\widetilde w^*}\Bigg[ \frac{1}{T}\int_0^T  \Big(  \big(r^{\veps,v} \big(Z_t\big)\wedge L\big) -\frac{1}{2}\|\widetilde w^*_t\|^2\Big)\D t\Bigg]+\delta\,.$$
	Again from Lemma~\ref{lem-comp-X-w-pert}, we can conclude that there exists an ergodic occupation measure $\widetilde \pi^*\in \calP(\RR^d\times \RR^d)$ such that MEMs  $\widetilde \pi_T$ of $(Z_{[0,T]},\widetilde w^*_{[0,T]})$ converge weakly to $\widetilde \pi^*(\D x, \D w)= \widetilde  \eta^*_v(\D x)\widetilde \mu^*(\D w|x)$, along a subsequence $T_k$. This means that 
	\begin{align*} \limsup_{T_k\to\infty}& \E_x^{v,\widetilde w^*}\Bigg[ \frac{1}{T_k}\int_0^{T_k} \Big(  \big(r^{\veps,v} \big(Z_t\big)\wedge L\big) - \frac{1}{2}\|\widetilde w^*_t\|^2\Big)\D t\Bigg]\\
	&\leq \int_{\RR^d\times \RR^d}\Big( \big(r^{\veps,v}(x)\wedge L\big)-\frac{1}{2}\|w\|^2\Big) \D \widetilde \pi^*(x,w) +\delta\,.\end{align*}
	
It can be easily seen that $\widetilde \pi^*$ is an ergodic occupation measure and moreover, $\widetilde \pi^*\in \frG^v_l$. This proves part (i).
	
	To prove part (ii), we take $L\to\infty$ in~\eqref{eq-occu-meas} and use the monotone convergence theorem. This gives us
	$$ \Lambda_v[r^\veps] \leq  \int_{\RR^d\times \RR^d}\Big( r^{\veps,v}(x)-\frac{1}{2}\|w\|^2\Big) {\D \widetilde \pi^*(x,w)} {+\delta}\,. $$
	Taking the supremum over all the ergodic occupation measures associated with $v\in {\Usm^{*,\beta}}$ and taking $\delta\downarrow 0$  completes the proof.
		\end{proof}
		{\begin{remark}\label{rem-inv} Without loss of generality, the constant $l=l(L,\delta,M_4)$ from the above lemma  for every $L>0$ can be taken to be such that as $L\uparrow \infty$, $l\uparrow\infty$  and $l$ is strictly increasing in $L$.  For a fixed $\delta,\beta$, we know that $M_4=M_4(\delta,\beta)$ is fixed. Hence, in this case we can conclude that $l(L,\delta,M_4)$ is invertible in $L$. We denote that inverse by $L^*(l,\delta,M_4)$.  The reason behind defining this inverse is our desire to keep the expressions  appearing in subscript in what follows short. 
		\end{remark}}
	The above lemma will be used in what follows to ensure that for $v\in {\Usm^{*,\beta}}$, there are nearly optimal controls for 
	$ \sup_{w\in\cA} \Lambda_{v,w}[r^\veps]$
	that are in $\Wsm$ (in other words, there are nearly optimal stationary Markov controls).

		\begin{remark} In the rest of the paper, we use the following notation: 
\begin{align*}
 \widetilde V^\veps\doteq \log V^\veps \text{ and }
\omega^\veps(\cdot)\doteq \Sigma(\cdot)\transp\nabla \widetilde V^\veps(\cdot)\,. 
\end{align*}

\end{remark}
		Before we state the next result, we define a family of CEC problems. For $l>0$, let {$ L^*\doteq L^*(l,\delta,M_4)$ be the inverse defined from Remark~\ref{rem-inv}} and define $\chi_l:\RR^d\rightarrow \RR$ to be a continuous function that satisfies  	$ \chi_l(x)=0$, whenever {$x\in B_{ l}^c$} and $\chi_{ l}(x)=1$, whenever {$x\in B_{\frac{ l}{2}}$}.	
		Following the techniques of \cite[Section 3]{anugu24ergodic}, define $f^\veps_l:\RR^d\times \bU\times \RR^d\rightarrow \RR$ and $\Delta_l:\RR^d\times \RR^d\rightarrow \RR$ by
\begin{align}\label{def-game-cost}
f^{\veps}_{l}(x,u,w)\doteq
r^{\veps}(x,u)\wedge {L^*}-\frac{1}{2}\|\chi_l(x) w\|^2  \quad \text{ and } \quad \Delta_l(x,w)\doteq \chi_l(x)
\Sigma(x) w\,.
\end{align}
Also, for $v\in {\Usm^{*,\beta}}$, define
\begin{align}\label{eqn-erg-cont-aug}
\Lambda_v^{l}[r^\veps]\doteq sup_{w\in \Wsm(l)} \Lambda_{v,w}[r^\veps\wedge L^*]\quad \text{and}\quad \widetilde \Lambda_v[r^\veps]\doteq \sup_{w\in \Wsm}\Lambda_{v,w}[r^\veps]\,.
\end{align}
For $v\in {\Usm^{*,\beta}}$, also set $f^{v,\veps}_l(x,w) = f^\veps_l(x,v(x),w)$. We remark that from the definition of $\Wsm(l)$,
$$\Lambda_v^{l}[r^\veps]= \sup_{w\in \Wsm(l)} \Lambda_{v,w}[r^\veps\wedge L^*]= \sup_{w\in \Wsm}\limsup_{T\to\infty} \E_x^{v,w}\Bigg[ \frac{1}{T}\int_0^T f^{v,\veps}_l \big(Z_t,w(Z_t)\big) \D t\Bigg]\,.  $$
We then consider the limit as $l\to\infty$. We state and prove a modified version of \cite[Theorem 4.1]{ABP15}. 
\begin{theorem}\label{app-thm-trunc} For $v\in {\Usm^{*,\beta}}$
	and  every $l>0$, there exists a function $\upu^{v,\veps}_{l}\in W^{2,p}_{\text{loc}}(\RR^d)$, $p>d$, with $\upu^{v,\veps}_l(0)=0$  such that 
	\begin{align}\label{eq-trunc}
	\Lg^v \upu^{v,\veps}_{l}(x) + \max_{w:\|w\|\leq l}\big\{ f^{v,\veps}_{l}(x,w) + \Delta_{l}(x,w)\cdot\grad \upu^{v,\veps}_{l}(x)\big\}= \Lambda_v^{l}[r^\veps]\,, \text{ for } x\in \RR^d\,.
	\end{align}
	Moreover,  we have the following. 
	\begin{enumerate}
	\item[(i)]$\upu^{v,\veps}_{l}\in \calC^2(B_l)$ and
	\item [(ii)]  $\Lambda_v^l[r^\veps]$ is non-decreasing in $l$.	
\end{enumerate}
\end{theorem}
\begin{proof} From Corollary~\ref{cor-stable}, we know that $X$ (under $v\in {\Usm^{*,\beta}}$) is positive recurrent. From here, for every $l>0$, $Z$ (under $v\in {\Usm^{*,\beta}}$ and $w\in \Wsm(l)$) is positive recurrent. Indeed, the generators of the processes $X$ (under $v\in {\Usm^{*,\beta}}$) and $Z$ (under $v\in {\Usm^{*,\beta}}$ and $w\in \Wsm(l)$)  coincide outside $B_l$. Therefore the hypothesis of \cite[Theorem 4.1]{ABP15} is satisfied and this proves the existence of solutions in~\eqref{eq-trunc} and part (i). Finally, part (ii) follows trivially.
\end{proof}

The following lemma studies how  $\{\Lambda_v^{l}[r^\veps]\}_{l>0}$ behaves as $l\uparrow \infty$. 
\begin{lemma}\label{lem-limit-truc-pert} For every $v\in {\Usm^{*,\beta}}$, the following hold: 
 	\begin{enumerate}
	\item [(i)]$\lim_{l\to\infty} \Lambda_v^{l}[r^\veps]= \Lambda_v[r^\veps] =  \widetilde  \Lambda_v[r^\veps]=\sup_{w\in \Wsm} \Lambda_{v,w}[r^\veps]\,. $
	\item[(ii)] There exists a positive function $\widetilde \upu^{v,\veps}\in W^{2,p}_{\text{loc}}(\RR^d)$ with $p>d$ such that 
	\begin{align}\label{eq-poisson-veps}
	\Lg^v\widetilde \upu^{v,\veps}(x)+ r^{\veps,v}(x)\widetilde \upu^{v,\veps}(x)=\Lambda_v[r^\veps] \widetilde \upu^{v,\veps}(x), \text{ for $x\in \RR^d$.}
	\end{align}
	\item [(iii)] \begin{align}\label{eq-inf-sup-pert}
\Lambda_{\text{SM}}[r^\veps]&=\inf_{v\in \Usm^*} \sup_{w\in \Wsm} \Lambda_{v,w}[r^\veps]\,.
\end{align}
	
	\end{enumerate}
\end{lemma}

\begin{proof}
	We first prove that $\{\Lambda_v^{l}[r^\veps]\}_{l>0}$ is a bounded sequence in $\RR$. Since  $\Wsm(l)\subset \Wsm$, using the definition of $\Lambda^l_v[r^\veps]$ and $\widetilde \Lambda_v[r^\veps]$, we have $\Lambda^l_v[r^\veps]\leq \sup_{\Wsm} \Lambda_{v,w}[r^\veps\wedge L^*]\leq \sup_{\Wsm} \Lambda_{v,w}[r^\veps]=\widetilde \Lambda_v[r^\veps]$, for every $l>0$. Also, recall that $\Lambda_v[r^\veps]= J(x,U)[r^\veps]$ for $U_t=v(Z_t)$, from the definition of $\Lambda_v$.  From here, using~\eqref{eqn-erg-cont-var-rep-pert}, and the fact that $v\in{\Usm^{*,\beta}}$ and $\Wsm(l)\subset \cA$,
 we have $\Lambda^l_v[r^\veps]\leq\widetilde \Lambda_v[r^\veps]\leq  \Lambda_{v}[r^\veps]<\infty$. In other words, $\{\Lambda_v^{l}[r^\veps]\}_{l>0}$ is a bounded  sequence in $\RR$.  Let $\Lambda_v^{*}[r^\veps]$ be a limit point of the sequence $\{\Lambda_v^{l}[r^\veps]\}_{l>0}$. Then, we have $\Lambda_v^{*}[r^\veps] \leq \Lambda_v[r^\veps]$. Moreover, from Lemma~\ref{lem-cec-cost}, we can conclude that $\Lambda_v^{*}[r^\veps] \geq \Lambda_v[r^\veps] $. This proves that the sequence $\{\Lambda_v^{l}[r^\veps]\}_{l>0}$ is convergent with limit $\Lambda_v[r^\veps]$ which also satisfies $\Lambda_v[r^\veps] =  \widetilde  \Lambda_v[r^\veps]$. This completes the proof of part (i). 
 
 To prove part (ii), we use the fact that $r^{\veps,v}$ is inf-compact and $\Lambda_v[r^\veps]<\infty$ (as $v\in{\Usm^{*,\beta}}$), in conjunction with \cite[Lemma 2.3(ii)--(iii)]{AB18} to conclude the existence of a positive function $\widetilde \upu^{v,\veps}\in W^{2,p}_{\text{loc}}(\RR^d)$ with $p>d$ that satisfies~\eqref{eq-poisson-veps}. 
 
 Finally,  part (iii) follows trivially from part (i). This completes the proof.
	\end{proof}

\begin{remark}\label{rem-eigen-value-pert}
Since $r^\veps$ is inf-compact, using \cite[Theorem 1.4]{AB18} and Lemma~\ref{lem-limit-truc-pert}(i), we can conclude that $ \Lambda_v[r^\veps]=\lambda^*_v[r^\veps]=\widetilde \Lambda_v[r^\veps]$. 
\end{remark}
\begin{remark}\label{rem-v-h-pert} It is easy to conclude (using the same proof as above) that for any $f\in \sC_0$ and for every $v\in {\Usm^{*,\beta}}$,  $\Lambda_v[r^\veps+f]= \sup_{w\in \Wsm}\Lambda_{v,w}[r^\veps+f]$.
\end{remark}

The next theorem (Theorem~\ref{thm-lin-rep-pert}) is the most important result of this section because it will help us with the characterization of the optimal stationary Markov controls for the original ERSC problem. Before we state and prove it, we give { an important existing result} that is used in its proof and later on in the proof of Theorem~\ref{thm-diffusion}(ii).  This is  taken from \cite{borkar1992stochastic} and pertains to an ergodic two-person zero-sum game, where the minimizing and maximizing strategies are compact-space valued. 
 {To state this result,}  we first define a family of two-person zero-sum games.    Fix $v_0\in \Usm^{*,\beta}$, define $\Usm^{*,\beta}(l)$ as the set of all $v\in \Usm$ such that $v(x)\equiv v_0$ on $B_l^c$. Now let
\begin{align}\label{def-lambda-e-l}
\Lambda_{v,w}^l&\doteq \limsup_{T\to\infty}\frac{1}{T} \E_x^{v,w}\Big[\int_0^T  f_l^{v,\veps}(Z_t,w(Z_t))\D t\Big]\,,\\\label{def-rho-min-max}
\overline \rho^\veps_l&\doteq \inf_{v\in {\Usm^{*,\beta}(l)}} \sup_{w\in  \Wsm(l)} \limsup_{T\to\infty}\frac{1}{T} \E_x^{v,w}\Big[\int_0^T  f_l^{v,\veps}(Z_t,w(Z_t))\D t\Big]\,,\\\label{def-rho-max-min}
\underline \rho^\veps_l&\doteq  \sup_{w\in  \Wsm(l)} \inf_{v\in {\Usm^{*,\beta}(l)}}\limsup_{T\to\infty}\frac{1}{T} \E_x^{v,w}\Big[\int_0^T  f_l^{v,\veps}(Z_t,w(Z_t))\D t\Big]\,. 
\end{align}
Here, $f_l^{\veps}(x,v(x),w)=f^\veps_l(x,v(x),w)$ as defined in~\eqref{def-game-cost} for every $v\in {\Usm^{*,\beta}}$.

\begin{remark}
Even though the two problems \emph{viz.}~\eqref{def-rho-max-min} and~\eqref{def-rho-min-max} involve infinimum over a set that is possibly a subset of $\Usm$, we can recast both problems in~\eqref{def-rho-max-min} and~\eqref{def-rho-min-max} as involving infimum over $\Usm$ by redefining process $Z$ (denote the redefined process by $Z_l$) and running cost $f^{v,\veps}_l$ (denote the redefined running cost by $\widehat f^{v,\veps}_l$ as follows: $Z^l$ is the  strong solution to~\eqref{X-control} with drift $b_l$ (instead of $b$) given by 
\begin{align*}b_l(x,u)= \begin{cases}b(x,u)&\text{ if $x\in B_l$,}\\
b(x,\overline v(x)) &\text{ if $x\notin B_l$,} \end{cases}\quad \text{and}\quad\widehat f^{v,\veps}_l(x, w)= \begin{cases}f_l^{\veps}(x,v(x),w)&\text{ if $x\in B_l$,}\\
 f_l^{\veps}(x,v_0(x),w) &\text{ if $x\notin B_l$}\,. \end{cases} \end{align*}

\end{remark}
Using this equivalence and results from \cite{borkar1992stochastic} (mainly, Theorems 4.4 and 4.5 of that paper), we analyze problems~\eqref{def-rho-min-max} and~\eqref{def-rho-max-min}. Observe that $\Wsm(l)$  can be treated as the set of compact-space valued controls.   
\begin{lemma}\label{lem-lyap-prop} For given $\veps,l>0$ and $\varkappa\geq 2$, define $\widetilde \sW(x)=\widetilde\sW_{\varkappa,\overline v}(x)\doteq  (\log \sW_{\overline v}(x))^{\frac{1}{\varkappa}}$. Then, the following holds:  there exists $\theta>0$ such that for every $v\in\Usm^{*,\beta}(l)$ and $w\in \Wsm(l)$, 
\begin{align}\label{eq-lyap-2p-game}
\Lg^v \widetilde \sW(x) + \big(\Sigma(x) w(x)\big)\cdot \nabla \widetilde\sW(x)\leq -\theta, \text{ for every $x\in B_l^c$.}
\end{align}
Moreover, for any $\varkappa> 2$, we have 
\begin{align}\label{eq-exp-zero}
\lim_{T\to\infty}\frac{1}{T}\E_x^{v,w}\big[\widetilde \sW(Z_T)\big]=0\,.
\end{align}
\end{lemma} 
\begin{proof} To begin with, for $v\in\Usm^{*,\beta}(l)$ and $w\in \Wsm(l)$ we substitute $\sW_{\overline v}(x)= \exp\big(\big(\widetilde \sW (x)\big)^\varkappa\big)$ in the left hand side of~\eqref{eq-h-wv} and simplify to get
\begin{align*}
\frac{1}{\sW_{\overline v}(x)}&\Big(\Lg^{v}\sW_{\overline v}(x) +\veps_0 h^v(x) \sW_{\overline v}(x)\Big)\\
&= {\varkappa\big(\widetilde \sW(x)\big)^{\varkappa-1}} \Lg^v \widetilde \sW(x) + \frac{1}{2}\Big(\varkappa^2\big(\widetilde \sW(x)\big)^{2\varkappa-2} +{\varkappa(\varkappa-1)\big(\widetilde \sW(x)\big)^{\varkappa-2}}\Big)\|\Sigma(x)\transp\nabla \widetilde \sW(x)\|^2\\
& \qquad + \veps_0 h^v(x)\,.
\end{align*}
Therefore, from~\eqref{eq-h-wv} and the above display, we have
\begin{align*}
&\Lg^v \widetilde \sW(x) +\big(\Sigma(x)w(x)\big)\cdot  \nabla \widetilde \sW(x)\\
& =- \frac{1}{2 {\varkappa\big(\widetilde \sW(x)\big)^{\varkappa-1}}}\Big(\varkappa^2\big(\widetilde \sW(x)\big)^{2\varkappa-2} +{\varkappa(\varkappa-1)\big(\widetilde \sW(x)\big)^{\varkappa-2}}\Big)\|\Sigma(x)\transp\nabla \widetilde \sW(x)\|^2\\
&\qquad - \frac{\big(\veps_0 h^v(x)-\lambda^*_v [\veps_0 h^v]\big)}{{\varkappa\big(\widetilde \sW(x)\big)^{\varkappa-1}}} +\big(\Sigma(x) w(x)\big)\cdot  \nabla \widetilde \sW(x)\\
&\leq  - \frac{1}{2 {\varkappa}}\Big(\varkappa^2\big(\widetilde \sW(x)\big)^{\varkappa-1} +{\varkappa(\varkappa-1)\big(\widetilde \sW(x)\big)^{-1}}\Big)\|\Sigma(x)\transp\nabla \widetilde \sW(x)\|^2\\
&\qquad - \frac{\big(\veps_0 h^v(x)-\lambda^*_v [\veps_0 h^v]\big)}{{\varkappa\big(\widetilde \sW(x)\big)^{\varkappa-1}}} + 2\|w(x)\|^2 + 2 \|\Sigma(x)\transp \nabla \widetilde \sW(x)\|^2\,.
\end{align*}
From the inf-compactness of $\sW_{\overline v}$ (and thereby, the inf-compactness of $\widetilde \sW$), we can ensure that outside a large enough closed ball, the expression in the second inequality above is strictly lesser than any given $-\theta$ (for $\theta>0$).

To prove~\eqref{eq-exp-zero}, we first apply  It\^o-Krylov's lemma to $\widetilde \sW_{2,\overline v}(Z_t)$ to obtain 
$$ \limsup_{T\to\infty}\frac{1}{T}\E_x^{v,w}\big[\widetilde \sW_{2,\overline v}(Z_T)\big]<\infty\,.$$
For $\varkappa>2$, we note that $$\limsup_{\|x\|\to\infty}\frac{\widetilde \sW_{\varkappa,\overline v}(x)}{ \widetilde \sW_{2,\overline v}(x) }=0\,.$$
From \cite[Lemma 3.7.2(ii)]{arapostathis2012ergodic},  we obtain~\eqref{eq-exp-zero}. This completes the proof.
\end{proof}

The above lemma verifies the conditions of Theorems 4.5 and 4.6 of \cite{borkar1992stochastic}. Below, we state the combination of them.

\begin{proposition}\label{prop-2p-game} 
There exist a unique function $\Psi^\veps_l\in W^{2,p}_{\text{loc}}(\RR^d)\cap\srO(\widetilde W_{\varkappa,\overline v}) $, $2\leq p<\infty$, $\varkappa>2$ and a constant $\rho^\veps_l\in \RR$ such that the following hold with $\Delta_l$ as defined in~\eqref{def-game-cost}:
\begin{enumerate}
\item [(i)] \begin{align}\label{eq-hjb-game}
&\min_{u\in \bU}\max_{w:\|w\|\leq l} \Big\{\Lg^u \Psi^\veps_l(x) + f^\veps_l (x,u,w)+ \Delta_l(x,w)\cdot \nabla \Psi^\veps_l(x)\Big\} \nonumber\\
&\qquad=\max_{w:\|w\|\leq l} \min_{u\in \bU} \Big\{\Lg^u \Psi^\veps_l(x)+ f^\veps_l(x,u,w)+ \Delta_l(x,w)\cdot \nabla \Psi^\veps_l(x)\Big\}=\rho^\veps_l\,, \text{ for } x\in \RR^d\,. 
\end{align}
\item[(ii)] $\rho^\veps _l=\overline \rho^\veps_l=\underline \rho^\veps_l$. 
\item [(iii)] For any $v\in {\Usm^{*,\beta}}$ and $w\in \Wsm(l)$, 
$$ \lim_{T\to\infty} \frac{1}{T}\E_x^{v,w}\big[ \Psi^\veps_l(Z_T)\big]=0 \text{ and } \lim_{R\to\infty}\E_x^{v,w}\big[\Psi^{\veps}_l(Z_{T\wedge \tau_R})\big]=\E_x^{v,w}\big[\Psi^{\veps}_l(Z_{T})\big]\,.$$
\item [(iv)] $v^*\in {\Usm^{*,\beta}} $ satisfies $\sup_{w\in \Wsm(l)}\Lambda^l_{v^*,w}= \rho^\veps_l$ if and only if for a.e. $x\in \RR^d$,
\begin{align*} \min_{u\in \bU}\max_{w:\|w\|\leq l} \Big\{\Lg^u \Psi^\veps_l (x)&+ f^\veps_l (x,u,w)+ \Delta_l(x,w)\cdot \nabla \Psi^\veps_l(x)\Big\}\\
&= \max_{w:\|w\|\leq l} \Big\{\Lg^{v^*} \Psi^\veps_l(x) + f^{v^*,\veps}_l (x,w)+ \Delta_l(x,w)\cdot \nabla \Psi^\veps_l(x)\Big\}\,.\end{align*}
\item [(v)] $w^*\in \Wsm(l)$ satisfies $\inf_{v\in {\Usm^{*,\beta}}} \Lambda^l_{v,w^*}=\rho^\veps_l$ if and only if for a.e. $x\in \RR^d$,
\begin{align*} \min_{u\in \bU} \Big\{\Lg^u \Psi^\veps_l(x)&+ f^\veps_l(x,u,w^*(x))+ \Delta_l(x,w^*(x))\cdot \nabla \Psi^\veps_l(x)\Big\}\\
&=\max_{w:\|w\|\leq l} \min_{u\in \bU} \Big\{\Lg^u \Psi^\veps_l(x)+ f^\veps_l(x,u,w)+ \Delta_l(x,w)\cdot \nabla \Psi^\veps_l(x)\Big\}\,.\end{align*}
Moreover, { if $w^*$ satisfies the above display, then it is} bounded and continuous on $\RR^d$.
\end{enumerate}
\end{proposition}
We use the above result to prove the following result {concerning} $\widetilde V^\veps$ and $\lsm[r^\veps]$.

\smallskip

		\begin{theorem}\label{thm-lin-rep-pert}
For any $R>0$ and $x\in B_R^c$,	the following statements hold.
\begin{itemize}
\item[(i)] $\lsm[r^\veps]$ has the following  characterizations:
\begin{equation}\label{eq-alt-1}
\begin{aligned}
\lsm[r^\veps] &= \lim_{l\to \infty}\rho^\veps_l \\
&=\sup_{w\in \Wsm}\inf_{v\in {\Usm^{*,\beta}}} \Lambda_{v,w}[r^\veps] \\
& =\lim_{l\to\infty}\sup_{w\in \Wsm(l)}\inf_{v\in {\Usm^{*,\beta}}} \Lambda_{v,w}[r^\veps]\,.
\end{aligned}
\end{equation}
\item[(ii)]  For $l>0$, $v\in\Usm^{o,\veps}$ and $w\in \Wsm(l)$,
\begin{align}\label{eq-lin-rep}
		\widetilde V^\veps(x)&\geq \E_x^{v,w}\Big[\int_0^{\widecheck \tau_R} \Big(r^{\veps,v}(Z_t)-\frac{1}{2}\|w(Z_t)\|^2- \lsm[r^\veps]\Big)\D t+\widetilde V^\veps(Z_{\widecheck \tau_R})\Big]\nonumber\,.		\end{align}
		
\end{itemize}

\end{theorem}
\begin{proof}
We first prove part (i). Since $\inf_{x}\sup_{y} f(x,y)\geq \sup_y\inf_x f(x,y)$, we have $$ \lsm[r^\veps]=\inf_{v\in{\Usm^{*,\beta}}} \sup_{w\in \Wsm}\Lambda_{v,w}[r^\veps] \implies \lsm[r^\veps]\geq \sup_{w\in \Wsm}\inf_{v\in{\Usm^{*,\beta}}}\Lambda_{v,w}[r^\veps]\,.$$
Therefore, it suffices to show that for any $\delta>0$, there exists $w^*\in \Wsm$ such that 
\begin{equation}\label{eq-sup-inf-rep-1}
\lsm[r^\veps]\leq \inf_{v\in\Usm^{*,\beta}}\Lambda_{v,w^*}[r^\veps]+\delta\,.
\end{equation}
We now show that $ \limsup_{l\to\infty}\rho^\veps_l\leq \lsm[r^\veps]. $
Observe that $\Wsm(l)\subset \Wsm$ for any $l>0$ and  choose $v^\veps$ that is optimal for $\lsm[r^\veps]$. Now, define 
$$ \overline v^\veps_l(x)\doteq \begin{cases} v^\veps(x) &\text{ if $x\in B_l$}\\
v_0 &\text{ if $x\in B_l^c$}\,.
\end{cases}$$
Here, $v_0\in \Usm^{*,\beta}$ is as chosen in  the definition of $\Usm^{*,\beta}(l)$.
$$ \limsup_{l\to\infty} \rho^\veps_l=\limsup_{l\to\infty}\overline \rho^\veps_l \leq \limsup_{l\to\infty} \sup_{w\in \Wsm(l)}\Lambda^l_{v^\veps_l,w}\leq \limsup_{l\to\infty} \sup_{w\in \Wsm}\Lambda^l_{v^\veps_l,w}\leq \lsm[r^\veps]\,. $$
The last inequality above follows from the proof of Lemma~\ref{lem-cec-cost}(i),. Now, let us show that 
\begin{align}\label{eq-sup-inf-pert-2}
\liminf_{l\to\infty}\rho^\veps_l=\liminf_{l\to\infty}\overline\rho^\veps_l \geq \lsm[r^\veps]\,.
\end{align} 
This will also show that for any $\delta>0$, there exists $w^*_l \in \Wsm(l)$ for large enough $l$ such that 
$$ \lsm[r^\veps]\leq \inf_{v\in {\Usm^{*,\beta}}} \Lambda^l_{v,w^*_l}+\delta\,.$$
This is exactly~\eqref{eq-sup-inf-rep-1} which is what we were set out to prove.  To prove~\eqref{eq-sup-inf-pert-2}, we make the observation that from the proof of Lemma~\ref{lem-cec-cost}(i), Lemma~\ref{lem-tail-est} and Lemma~\ref{lem-comp-X-w-pert}, we can choose $l>0$ uniformly in $v\in {\Usm^{*,\beta}}$. In other words, for any $\delta>0$, there exists $l=l(\delta)$ large enough, such that 
\begin{align*}
\Lambda_v[r^\veps]\leq \sup_{w\in \Wsm(l)} \Lambda^l_{v,w}[r^\veps]+\delta, \text{ for $v\in {\Usm^{*,\beta}}$}\,.
\end{align*}
Now choose $v_l\in \Usm^{*,\beta}(l)$ that is $\delta-$optimal {for $\overline \rho^\veps_l$ \emph{i.e.,}} $\sup_{w\in \Wsm(l)}\Lambda^l_{v_l,w}\leq \overline \rho^\veps_l +\delta $. From the above display, we have
$$ \lsm[r^\veps]\leq \Lambda_{v_l}[r^\veps]\leq \sup_{w\in \Wsm(l)}\Lambda^l_{v_l,w}+\delta\leq \overline \rho^\veps_l+\delta\,.$$
Since the  other two equalities of part (i) now follow trivially, we completed the proof of part (i).

We now proceed with proof of part (ii). Since $\rho^\veps_l\to \lsm[r^\veps]$, as $l\to\infty$, using the standard elliptic regularity theory as $\veps\to 0$, uniqueness of $\widetilde V^\veps$ from Theorem~\ref{thm-HJB-pert} will help us conclude that $\Psi^\veps_l\to \widetilde  V^\veps$ strongly in {$W^{1,p}_{\text{loc}}(\RR^d)$, for $p\geq 2$ and in particular for $p>d$. From the continuous embedding of    $W^{1,p}(K)$ in $ \cC^{0,\gamma}(K)$ with $p>d$ and $\gamma=1-\frac{d}{p}$
and for every compact set $K\subset\RR^d$,  we can also conclude that   $\Psi^\veps_l\to \widetilde  V^\veps$ converges  uniformly on compact sets of $\RR^d$.} We already know that $\widetilde V^\veps$ satisfies 
$$ \min_{u\in \bU} \Big\{\Lg^u\widetilde V^\veps (x)+ r^\veps(x,u)+\frac{1}{2}\|\omega^{\veps}(x)\|^2\Big\}=\lsm[r^\veps]\,, \text{ for } x\in \RR^d\,.$$
This in turn can be re-written as 
$$ \min_{u\in \bU}\max_{w\in \RR^d} \Big\{\Lg^u\widetilde V^\veps(x) + r^\veps(x,u)+ (\Sigma (x) w)\cdot \nabla V^\veps(x)-\frac{1}{2}\|w\|^2\Big\}=\lsm[r^\veps]\,, \text{ for } x\in \RR^d\,.$$
Now choose $u\equiv v\in \Usm^{o,\veps}$ and $w=w(\cdot)\in \Wsm(l)$ and this gives us
$$ \Lg^v\widetilde V^\veps(x) + r^{\veps,v}(x)+ \big(\Sigma (x) w(x)\big)\cdot \nabla V^\veps(x)-\frac{1}{2}\|w(x)\|^2\leq \lsm[r^\veps]\,, \text{ for } x\in \RR^d\,.
$$
For $\hat R> R$ {and $x\in B_{\hat R}\setminus B_{R}$}, applying It{\^o}'s formula to the above display, we get
\begin{align*}
\widetilde V^\veps(x) \geq \E_x^{v,w}\Big[\int_0^{\widecheck \tau_{R}\wedge \tau_{\hat R}} \Big( r^{\veps,v}(Z_t)-\lsm[r^\veps]-\frac{1}{2}\|w(Z_t)\|^2 \Big)\D t + \widetilde V^\veps(Z_{\widecheck \tau_{R}\wedge \tau_{\hat R}})\Big]\,.
\end{align*}
From the fact that $w\in \Wsm(l)$, we have $ \lim_{\hat R\to\infty}\E_x^{v,w}\Big[ \widetilde V^\veps(Z_{\widecheck \tau_{R}\wedge \tau_{\hat R}})\Big]= \E_x^{v,w}\Big[ \widetilde V^\veps(Z_{\widecheck \tau_{R}})\Big]$.
This and the application of monotone convergence theorem give the result.
\end{proof}

We are ready to prove Lemma~\ref{lem-monotonicity}. 
\begin{proof}[Proof of Lemma~\ref{lem-monotonicity}] Let $\delta>0$ and $\frB\subset \RR^d$ be an open ball. Let $v^*\in \Usm$ be an optimal Markov control corresponding to $\Lambda_{\text{SM}}[r^\veps+\delta\Ind_{{\frB}}]$. Since $\Lambda_{v^*}[r^\veps]\leq \Lambda_{v^*}[r^\veps+\delta \Ind_{\frB}]\leq \Lambda_{v^*}[r^\veps] +\delta $, without loss in generality, we can suppose that $v^*\in \Usm^{*,\beta}$, for a larger $\beta$, if necessary. From Proposition~\ref{prop-var-cost}(i) (which can also be applied to $r^\veps+\delta\Ind_{{\frB}}$ instead of $r^\veps$), we have
\begin{align*}
\Lambda_{\text{SM}}[r^\veps+\delta\Ind_{{\frB}}] &=\Lambda_{v^*}[r^\veps+\delta\Ind_{{\frB}}]\geq  \Lambda_{v^*,w}[r^\veps +\delta \Ind_{{\frB}}]\,,
\end{align*}
for every $w\in \cA$. Using Theorem~\ref{thm-lin-rep-pert}(i), it is clear that for a given $n$, we can choose a large enough $l$ and  $w^*=w^*(n)\in \Wsm(l)$ such that $$ \lsm[r^\veps]\leq \Lambda_{v,w^*}+\frac{1}{n}, \text{ for every $v\in \Usm^{*,\beta}$}\,.$$
 Then, from the above two displays, we have
\begin{align*}
\Lambda_{\text{SM}}[r^\veps+\delta\Ind_{{\frB}}]&\geq \Lambda_{v^*,w^*}[r^{\veps}+\delta\Ind_{{\frB}}]\geq \lsm[r^\veps] -\frac{1}{n}+ \delta\limsup_{T\to\infty} \pi^n_T({\frB}) \,.
\end{align*}
The second inequality follows from the fact that $\Lambda_{v^*,w^*}[g]$ is linear in $g$. Here, $\pi^n_T$ is the MEM of the process $Z$ under $v^*$ and $w^*$. 
Since $w^* \in \Wsm(l)$, $w^*$ is a bounded function. Hence, we can follow the arguments similar to those in the proof of Lemma~\ref{lem-comp-X-pert} to conclude that 
 \begin{align*}
		\limsup_{T\to\infty} \frac{1}{T}\E_x^{v^*,w^*}\Big[\int_0^T h(Z_t,v(Z_t))\D t\Big]<\infty\,.
		\end{align*} In particular,   we know that $\{\pi^n_{T}\}_{n,T}$ is tight in both $n$ and $T$. Therefore, along a subsequence $T_k$, $\pi^n_{T_k}$ converges weakly to some invariant measure $\pi^n$ of $Z$ under $v^*$ and $w^*$. From the Portmanteau theorem, $$ \liminf_{k\to\infty} \pi_{T_k}^n({\frB})\geq \pi^n({\frB})\,.$$ 
Next along a subsequence $n_k$, $\pi^{n_k}$ further converges weakly to some other invariant measure $\pi^*$. From \cite[Theorem 2.6.16]{arapostathis2012ergodic}, we know that $\pi^*$ has a positive density with respect to Lebesgue measure. Therefore, again using the Portmanteau theorem gives us $\liminf_{k\to\infty} \pi^{n_k}({\frB})\geq \pi^*({\frB})>0$.
Now choosing $k$ large enough such that $n_k\geq  \frac{2}{{\delta\pi^*({\frB})}}$, we get
$$\Lambda_{\text{SM}}[r^\veps+\delta\Ind_{{\frB}}]\geq \Lambda_{\text{SM}}[r^\veps] + \frac{1}{2}\delta\pi^*({\frB})\,.$$
This completes the proof.
\end{proof}

The following result states that the infimum of the ERSC objective associated with $r^\veps$ is achieved by certain $v\in {\Usm^{*,\beta}}$. Although this result is not used in the rest of the paper and in particular, in the proof of Theorem~\ref{thm-diffusion}(iv), it reduces to Theorem~\ref{thm-diffusion}(iv) when $\veps=0$.  We state and prove this result for the following reason: {to} the best of the authors' knowledge, the existing results in the literature cannot be applied to prove this result.   In accordance with the theme of the rest of the paper, we use the tools developed in Section~\ref{sec-var-BM} in the proof below.

		\begin{proposition}\label{lem-cost-equal-pert}$ \lsm[r^\veps]=\Lambda[r^\veps].$ 
		\end{proposition}
		\begin{proof} 
		Since $\lsm[r^\veps]\geq \Lambda[r^\veps]$ trivially, we only show the reverse inequality. To that end, fix $\delta>0$ and choose a $\delta$--optimal $U^*$ for $\Lambda[r^\veps]$.  From Proposition~\ref{prop-var-cost}(ii), it is then clear that 
		\begin{align}\label{eq-1-cost-equal}
		J(x,U^*)[r^\veps]\geq  J(x,U^*,w)[r^\veps]\, ,
	\end{align}
	for $w\in \cA$. Using Theorem~\ref{thm-lin-rep-pert}(i), we now choose large enough $l$ such that 
	\begin{align}\label{eq-opt-cont} \lsm[r^\veps] \leq \sup_{w\in \Wsm(l)}\inf_{v\in \Usm} \Lambda_{v,w}[r^\veps] +\delta\,.
\end{align}
Taking $w=w^*$ from Proposition~\ref{prop-2p-game}(v) in~\eqref{eq-1-cost-equal}, we have
\begin{align}\label{eq-2-cost-equal}
		J(x,U^*)[r^\veps]\geq J(x,U^*,w^*)[r^\veps] \,.
	\end{align}
	Since $w^*\in \Wsm(l)$ and $\sup_{x\in \RR^d}\|w^*(x)\|\leq l$, we can conclude that 
	$$ \sup_{T>0}\frac{1}{T}\E_x^{U^*,w^*}\Big[\int_0^T \|w^*(Z_t)\|^2\D t \Big]\leq l \,.$$
	From here,~\eqref{eq-2-cost-equal} and the definition of $J(x,U^*,w^*)[r^\veps]$, we also have
	\begin{align*}
	\limsup_{T\to\infty}\frac{1}{T}\E_x^{U^*,w^*}\Big[\int_0^Tr^\veps (Z_t,U^*_t)\D t\Big] \leq J(x,U^*)[r^\veps] +l\,.
	\end{align*} 
	Following the arguments of the proof of Lemma~\ref{lem-comp-X-pert}, we get
	\begin{align}\label{eq-3-cost-equal} \limsup_{T\to\infty}\frac{1}{T}\E_x^{U^*,w^*}\Big[\int_0^Th(Z_t,U^*_t)\D t\Big]<\infty\,.\end{align}
	Therefore, the MEMs of $(Z_{[0,T]},U^*_{[0,T]})$ denoted by $\pi_T$ is a tight family of measures in $T>0$. This means that along a subsequence (still denoted by $T$), $\pi_T$ converges weakly to some measure $\pi_*$. From \cite[Lemma 3.4.6]{arapostathis2012ergodic}, $\pi_*$ is an ergodic occupation measure. With the decomposition of the measure, we can write $\pi_*(\D x,\D u)= \mu_v (\D x)v(\D u|x),$ for some $v\in {\Usm^{*,\beta}}$.  
Therefore, from the fact that $r$ is non-negative  and the fact that $w^*$ is bounded and continuous (from Proposition~\ref{prop-2p-game}(v)), we can conclude that 
$$\lim_{T\to\infty}\int_{\RR^d\times \bU} \Big(r^\veps(x,u)-\frac{1}{2}\|w^*(x)\|^2\Big)\D \pi_T(x,u)\geq \int_{\RR^d\times \bU} \Big(r^\veps(x,u)-\frac{1}{2}\|w^*(x)\|^2\Big)\D \pi_*(x,u)\,. $$
From here, together with~\eqref{eq-opt-cont}, we have 
{\begin{align*}
\Lambda[r^\veps] &\geq J(x,U^*)[r^\veps]-\delta\\
& \geq \int_{\RR^d\times \bU} \Big(r^\veps(x,u)-\frac{1}{2}\|w^*(x)\|^2\Big)\D \pi_*(x,u) -2\delta\\
&\geq \inf_{v\in {\Usm^{*,\beta}}} \Lambda_{v,w}[r^\veps]-2\delta\\
&\geq  \sup_{w\in \Wsm(l)}\inf_{v\in {\Usm^{*,\beta}}} \Lambda_{v,w}[r^\veps] -3\delta\\
& \geq \lsm[r^\veps]-3\delta\,.
\end{align*}
In the above, we obtain the fourth inequality  from  the definition.} From the arbitrariness of $\delta>0$, we have the result.
		\end{proof}

\medskip

\section{Analysis of the limiting behavior as $\veps \to 0$ and Proof of  Theorem~\ref{thm-diffusion} }\label{sec-proof-main}
In this section, we provide the proof of Theorem~\ref{thm-diffusion} by analyzing the behavior of $\lsm[r^\veps]$ and $\widetilde V^\veps$ (or equivalently, $V^\veps$) as $\veps \to 0$.  Recall that 
\begin{align*}
J(U,w)[r^\veps]&= \limsup_{T\to\infty}\E_x^{U,w}\Bigg[ \frac{1}{T}\int_0^T \Big( r^\veps(Z_t,U_t) - \frac{1}{2}\|w_t\|^2\Big)\D t\Bigg], \\
\Lambda_{v,w}[r^\veps]&= \limsup_{T\to\infty} \E_x^{v,w}\Bigg[ \frac{1}{T}\int_0^T \Big( r^{\veps,v}\big(Z_t\big) - \frac{1}{2}\|w_t\|^2\Big)\D t\Bigg]\,,
\end{align*}
and similarly, $J(U,w)[r]$ and $\Lambda_{v,w}[r]$. 
 The proof of Theorem~\ref{thm-diffusion} is divided into four subsections each of which contains the proof of one of the four parts of the theorem. 
 Before we begin, we present the following key convergence results that will be repeatedly used in the proof of Theorem~\ref{thm-diffusion}.  
 
 \medskip 
 \subsection{Convergence of the perturbed optimal ERSC cost}\label{sec-proof-cost-limit}
 The following theorem shows that  the perturbed ERSC problem indeed approximates the original ERSC problem in the sense that the perturbed optimal ERSC cost approaches the original optimal ERSC cost.  
 
\begin{theorem}\label{thm-pert-limit} The following  statements hold.
\begin{itemize}
\item[(i)] For $v\in {\Usm^{*,\beta}}$, 
$$ \lim_{\veps\to 0} \Lambda_v[r^\veps]=\Lambda_v[r]\,.$$
\item [(ii)] $$ \lim_{\veps\to 0} \lsm[r^\veps]=\lsm[r]\,.$$
\end{itemize} 

\end{theorem}
\begin{proof} 
Fix $v\in {\Usm^{*,\beta}}$ and $\delta>0$. Recall that from Lemma~\ref{lem-limit-truc-pert}(i), we have $ 
\Lambda_v[r^\veps]= \sup_{w\in \Wsm}\Lambda_{v,w}[r^\veps] $.
Let $w^\veps\in \Wsm$ be such that 
\begin{align} \label{eq-w-choice}
\sup_{w\in \Wsm} \Lambda_{v,w}[r^\veps]\leq \Lambda_{v,w^\veps}[r^\veps] +\delta\,. \end{align}
From Lemma~\ref{lem-comp-X-w-pert}, we know that 
\begin{align}\label{eq-thm-pert-limit-1}\sup_{0<\veps<\veps_0} \limsup_{T\to\infty} \frac{1}{T} \E_x^{v,w^\veps} \Big[\int_0^T h^v(Z_t)\D t\Big]\leq {M_4}\,.\end{align}
From the definition of $r^\veps$, 
we have
\begin{align*}
\Lambda_v[r^\veps]& \leq \Lambda_{v,w^\veps}[r^\veps] +\delta\\
&\leq \Lambda_{v,\omega^\veps}[r] 
       + \veps  \limsup_{T\to\infty}  \frac{1}{T} \E_x^{v,w^\veps} \Big[\int_0^T h^v(Z_t)\D t\Big] +\delta\\
 &\leq \Lambda_v[r] + \veps { M_4}  +\delta \,.             
\end{align*}
In the above display, to get the first line, we use Lemma~\ref{lem-limit-truc-pert}(i) in conjunction with~\eqref{eq-w-choice} and to get the second line, we use the definition of $r^\veps$.  To get the first term in the  third line, we use~\eqref{eqn-erg-cont-var-rep} and the fact that $w^\veps \in \Wsm \subset \cA$ (recall that $J(x,v)[r]=\Lambda_v[r]$ from our notation), and to get the second term in the third line, we use~\eqref{eq-thm-pert-limit-1}.
Arbitrariness of $\delta$ gives us $ \limsup_{\veps\to 0} \Lambda_v[r^\veps] \leq \Lambda_v[r]$.
To prove the reverse inequality, choose the $w^*\in \cA$ such that
\begin{align}\label{eq-inst-1} {\Lambda_v[r]=}\limsup_{T\to\infty}\sup_{w\in \cA} &\frac{1}{T}\E_x^{v,w}\Big[\int_0^T\Big(r^v (Z_t)-\frac{1}{2}\|w_t\|^2\Big)\D t\Big]{\leq  \Lambda_{v,w^*}[r]+\delta}\,. \end{align}
Applying Proposition~\ref{prop-var-cost}(i) (in particular,~\eqref{eqn-erg-cont-var-rep-pert}) for $v\in \Usm^{*,\beta}$, we  have
\begin{align*}
\Lambda_v[r^\veps]&=\limsup_{T\to\infty} \sup_{w\in \cA} \frac{1}{T}\E_x^{v,w}\Big[\int_0^T\Big(r^{\veps,v} (Z_t)-\frac{1}{2}\|w_t\|^2\Big)\D t\Big]\\
&=\limsup_{T\to\infty} \sup_{w\in \cA} \frac{1}{T}\E_x^{v,w}\Big[\int_0^T\Big(\big(1-\frac{\veps}{\veps_0}\big)r(Z_t)+ \veps h (Z_t)-\frac{1}{2}\|w_t\|^2\Big)\D t\Big]\\
&\geq \limsup_{T\to\infty} \sup_{w\in \cA} \frac{1}{T}\E_x^{v,w}\Big[\int_0^T\Big(\big(1-\frac{\veps}{\veps_0}\big)r(Z_t)-\frac{1}{2}\|w_t\|^2\Big)\D t\Big]\\
&\geq \Lambda_{v,w^*}[r]   -\frac{\veps}{\veps_0}  \limsup_{T\to\infty} \frac{1}{T}\E_x^{v,w^*}\Big[\int_0^Tr^v(Z_t)\D t\Big]\\
&\geq \Lambda_v[r]-\frac{\veps}{\veps_0}  \limsup_{T\to\infty} \frac{1}{T}\E_x^{v,w^*}\Big[\int_0^Tr^v(Z_t)\D t\Big]-\delta\,.
\end{align*}
 In the above, to get the second line, we use the definition of $r^{\veps}$ (see~\eqref{def-r-cost-pert}) and the fact that $r^{\veps,v}(x)=r^\veps\big(x,v(x)\big)$; to get the third line, we use the fact that $h\geq 0$; to get the fourth line, we use the definition of $\Lambda_{v,w^*}[r]$ and finally, to get the last line, we use~\eqref{eq-inst-1}. Using from Lemma~\ref{lem-comp-X}, we have 
$$ \limsup_{T\to\infty} \frac{1}{T} \E_x^{v,w^*} \Big[\int_0^T h^v(Z_t)\D t\Big]\leq {M_3}\,.$$
Since $r\leq h$, taking $\veps\to 0$, gives us  $ \liminf_{\veps\to 0} \Lambda_v[r^\veps]\geq \Lambda_v[r]{-\delta}$.
{Arbitrariness of $\delta$ gives us the reverse inequality.}

\smallskip
We next prove part (ii) in a similar fashion. 
To begin with, fix $\delta>0$ and choose a $v^\veps\in \Usm^{o,\veps}$  that is optimal for $\lsm[r^\veps]$: 
\begin{align}\label{eq-cost-equiv-1}
\lsm[r^\veps]&= \Lambda_{v^\veps}[r^\veps]=  \sup_{w\in \Wsm}\Lambda_{v^\veps,w}{[r^\veps]}\,. 
\end{align}
We then choose $w^\veps\in \Wsm$ such that 
\begin{align*}\Lambda_{v^\veps}[{r^\veps}]=&\sup_{w\in \Wsm}\Lambda_{v^\veps,w}[{r^\veps}]\leq\Lambda_{v^\veps,w^\veps}[{r^\veps}] +\delta\,.\end{align*}
Evaluating~\eqref{eq-cost-equiv-1} with $w^\veps\in \Wsm$, we have
\begin{align*}
\lsm[r^\veps]&\geq \Lambda_{v^\veps,w^\veps}[r^\veps] \\
&\geq\Lambda_{v^\veps,w^\veps}[r]   -\frac{\veps}{\veps_0} \limsup_{T\to\infty}\frac{1}{T}\E_x^{v^\veps,w^\veps }\Big[\int_0^Tr^{v^\veps}(Z_t)\D t\Big]\\
&\geq \Lambda_{v^\veps}[r]-\delta -\frac{\veps}{\veps_0} \limsup_{T\to\infty}\frac{1}{T}\E_x^{v^\veps,w^\veps }\Big[\int_0^Tr^{v^\veps}(Z_t)\D t\Big]\\
&\geq \lsm[r]-\delta -\frac{\veps}{\veps_0} \limsup_{T\to\infty}\frac{1}{T}\E_x^{v^\veps,w^\veps }\Big[\int_0^Tr^{v^\veps}(Z_t)\D t\Big]\,.
\end{align*}
Again using Lemma~\ref{lem-comp-X-w-pert}, we have 
$$\sup_{0<\veps<\veps_0} \limsup_{T\to\infty} \frac{1}{T} \E_x^{v^\veps ,w^\veps} \Big[\int_0^T h^{v^\veps}(Z_t)\D t\Big]\leq {M_4}$$
and consequently, 
$ \liminf_{\veps\to 0} \lsm[r^\veps]\geq \lsm[r]-\delta$.
 To prove the reverse inequality, choose $v^*\in {\Usm^{*,\beta}}$ that is $\delta$--optimal for $\lsm[r]$. Then we have  $\lsm[r^\veps]\leq \Lambda_{v^*}[r^\veps] = \sup_{w\in \Wsm}\Lambda_{v^*,w}[r^\veps]$. 
With $w^\veps$ being the $\delta$--optimal for $\Lambda_{v^*}[r^\veps]$, we obtain
\begin{align*}
\lsm[r^\veps]&\leq \Lambda_{v^*,w^\veps}[r^\veps]{+\delta}\\
&\leq \Lambda_{v^*,w^\veps}[r] + \veps  \limsup_{T\to\infty}\frac{1}{T}\E_x^{v^*,w^\veps}\Big[\int_0^Th^{v^*}(Z_t)\D t\Big]+\delta\\
&\leq  \limsup_{T\to\infty} \sup_{w\in \cA} \Lambda_{v^*,w,T}[r]+ \veps  \limsup_{T\to\infty}\frac{1}{T}\E_x^{v^*,w^\veps}\Big[\int_0^Th^{v^*}(Z_t)\D t\Big]+\delta\\
&\leq \Lambda_{v^*}[r] + \veps  \limsup_{T\to\infty}\frac{1}{T}\E_x^{v^*,w^\veps}\Big[\int_0^Th^{v^*}(Z_t)\D t\Big]+\delta\\
&\leq \lsm[r]+ \veps  \limsup_{T\to\infty}\frac{1}{T}\E_x^{v^*,w^\veps}\Big[\int_0^Th^{v^*}(Z_t)\D t\Big]+2\delta\,. 
\end{align*}
Again using Lemma~\ref{lem-comp-X-w-pert}, we have
$$ \sup_{0<\veps<\veps_0} \limsup_{T\to\infty}\frac{1}{T}\E_x^{v^*,w^\veps}\Big[\int_0^Th^{v^*}(Z_t)\D t\Big]\leq {M_4}\,,$$
which consequently gives us
$$ \limsup_{\veps\to 0} \lsm[r^\veps]\leq \lsm[r]+2\delta\,.$$
Arbitrariness of $\delta$ gives us the result.
\end{proof}

Using Theorem~\ref{thm-pert-limit} we can infer
 that the ERSC cost associated with $r$ under $v\in {\Usm^{*,\beta}}$ can be written as a CEC problem where we maximize the the running cost $r^v(x)-\frac{1}{2}\|w\|^2$ associated with the extended process $Z$ over auxiliary controls $w\in \Wsm$. 
Moreover, we also show that there exists a nearly optimal stationary Markov control $w(\cdot)\in \Wsm $ such that $w$ vanishes outside a large compact set \emph{i.e.,} $w\in \Wsm(l)$ for large  $l$.

\begin{proposition}\label{prop-sup-v} For $v\in {\Usm^{*,\beta}}$, 
\begin{align*}
\Lambda_v[r] &=\sup_{w\in \Wsm} \Lambda_{v,w}[r] \\
&=\lim_{l\to\infty}\sup_{w\in \Wsm(l)} \Lambda_{v,w}[r]\,.
\end{align*}
\end{proposition}

\begin{proof} Since we know that $\Wsm, \Wsm(l)\subset \cA$ for $l>0$, we can immediately infer that $\Lambda_v[r]\geq \sup_{w\in \Wsm} \Lambda_{v,w}[r] $ and $\Lambda_v[r]\geq \sup_{w\in \Wsm(l)} \Lambda_{v,w}[r]\,.$
 Therefore, it suffices to show that for every $\delta>0$, there exist $l>0$, $w^*\in \Wsm$, $\tilde w=\tilde w(l)\in\Wsm(l)$ such that 
 $$ \Lambda_v[r]\leq \Lambda_{v,w^*}[r]+\delta \quad \text{ and } \quad \Lambda_v[r]\leq \Lambda_{v,\tilde w}[r]+\delta\,.$$
 To that end, from Theorem~\ref{thm-pert-limit} we know that for every $0<\delta<1$ (without loss of generality), there exist $\widetilde \veps\leq \veps_0$ such that for $0<\veps<\widetilde \veps$, we have $ \Lambda_v[r]\leq \Lambda_v[r^\veps]+\frac{\delta}{3}$.
 Choosing $w^\veps\in \Wsm$ such that $\Lambda_v[r^\veps]\leq \Lambda_{v,w^\veps}[r^\veps]+\frac{\delta}{3}$ gives us
 \begin{align*}
 \Lambda_v[r]&\leq \Lambda_{v,w^\veps}[r^\veps]+\frac{2\delta}{3}\\
 &= \limsup_{T\to\infty}\frac{1}{T}\E_x^{v,w^\veps} \Big[\int_0^T \Big(r^{\veps,v}(Z_t)-\frac{1}{2}\|w\|^2\Big)\D t\Big]+\frac{2\delta}{3}\\
 &\leq \limsup_{T\to\infty}\frac{1}{T}\E_x^{v,w^\veps} \Big[\int_0^T \Big(r^{v}(Z_t)-\frac{1}{2}\|w\|^2\Big)\D t\Big]+\frac{2\delta}{3} + \veps \limsup_{T\to\infty}\frac{1}{T}\E_x^{v,w^\veps}\Big[\int_0^T h^{v}(Z_t)\D t\Big]\,.
 \end{align*}
We know that from Lemma~\ref{lem-comp-X-w-pert}, $\sup_{0<\veps<\veps_0}\limsup_{T\to\infty}\frac{1}{T}\E_x^{v,w^\veps}\Big[\int_0^T h^{v}(Z_t)\D t\Big]\leq  {M_4}$.
 Therefore, choosing $ \veps\leq \min\big\{{ \frac{\delta}{3M_4}},\widetilde\veps\big\}\,,$
 we have shown that  $ \Lambda_v[r]\leq  \Lambda_{v,w^\veps}[r] +\delta$, for $\veps$ sufficiently small. To prove the second equality in the statement of the proposition, we use Lemma~\ref{lem-cec-cost}(i) to infer that for a large enough $l>0$, there exists  a relaxed Markov auxiliary control $\tilde w= \mu(\D w|x)\in \Wsm(l)$ with $\pi(\D x,\D w)= \eta_v(\D x) \mu(\D w|x)$ such that 
 $$ \Lambda_v[r^\veps] \leq \int_{\RR^d\times \RR^d}  \Big(r^{\veps,v}(x)\wedge L -\frac{1}{2}\|w\|^2\Big)\D \pi(x,w) +2\delta \leq \int_{\RR^d\times \RR^d} \Big(r^{\veps,v}(x) -\frac{1}{2}\|w\|^2\Big)\D \pi(x,w) +2\delta\,. $$
   From here, arguing similarly,  we can show that there exist $l>0$ and $\tilde w\in \Wsm(l)$ such that  $ \Lambda_v[r]\leq \Lambda_{v,\tilde w}[r]+\delta$.
 This proves the result.
\end{proof}

\subsection{Stability of the ground diffusion} 
For $v\in{\Usm^{*,\beta}}$, let $Z^v$ be the process defined as a strong solution (if it exists) to the following SDE: 
\begin{align}\label{eq-ground-diff}
dZ^v_t= b\big(Z^v_t,v(Z^v_t)\big)\D t+ \Sigma(Z^v_t)\Sigma\transp(Z^v_t) \nabla \upu^v(Z^v_t)\D t +\Sigma (Z^v_t)\D W_t \,\,, Z^v_0=x
\end{align}
with  $\upu^v\doteq \log \widetilde \upu^v$  and  $\widetilde \upu^v\in W^{2,p}_{\text{loc}}(\RR^d)$ being a positive solution to the equation 
\begin{align}\label{def-poiss-eq}
\Lg^v \widetilde \upu^v(x)+r^v(x) \widetilde \upu^v(x)=\lambda^*_v[r]\widetilde \upu^v(x), \text{ for a.e. $x\in \RR^d$.}
\end{align}
Here, $\lambda_v^*[\cdot]$ is as defined in~\eqref{def-princp-eig}. 
The process $Z^v$ is referred to as ``ground" diffusion associated with $v\in {\Usm^{*,\beta}}$. Observe that this is a particular case of the extended diffusion $Z$ defined in~\eqref{X-control} with \begin{align}\label{eq-aux-opt-cont}w_t\equiv \Sigma\transp(Z_t)\nabla \upu^v(Z_t)\,.\end{align} See \cite{ari2018strict,AB18} for discussions on its relation with risk sensitive cost associated with $r^v$. 
Unlike in the case of uniform stability (see \cite[Lemma 2.4]{ari2018strict}), the stability (in particular, recurrence) of the ground diffusion  is not obvious from the general structural hypothesis (Assumption~\ref{a-main}). Suppose that $w_t$ defined in~\eqref{eq-aux-opt-cont}  {(which is clearly $\cG_t$-adapted) lies in $\cA$ and }satisfies~\eqref{def-d-opt} for small $\delta$. Then, an application of  Lemma~\ref{lem-comp-X} immediately implies that the process $Z^v$ satisfies~\eqref{eq-rec} which in particular, concludes the recurrence of $Z^v$. However, it is not at all clear a priori even if $w$ defined  {in~\eqref{eq-aux-opt-cont}} satisfies~\eqref{def-d-opt}.   We  prove the stability of process $Z^v$ below, which will be used in the proofs of Theorem~\ref{thm-diffusion}(i) and (iii). We first prove the existence of $\upu^v$ mentioned above.

 \begin{lemma}\label{lem-ground-diffusion}
 For $v\in {\Usm^{*,\beta}}$, with $\widetilde \upu^{v,\veps}$ as in Lemma~\ref{lem-limit-truc-pert}(ii), as $\veps\to 0$ , $\widetilde \upu^{v,\veps}$ converges to some  positive function $\widetilde \upu^v$ in $ W^{2,p}_{\text{loc}}(\RR^d)$ weakly and in $W^{1,p}_{\text{loc}}(\RR^d)$ strongly, for $p>d$. Moreover, $\widetilde \upu^{v}$ is a solution to
 \begin{align}\label{def-poiss-eq-v}
\Lg^v \widetilde \upu^v(x)+r^v(x) \widetilde \upu^v(x)=\Lambda_v[r]\widetilde \upu^v(x), \text{ for a.e. $x\in \RR^d$}.
\end{align}
 \end{lemma}

  \begin{proof} Fix $v\in {\Usm^{*,\beta}}$ and $\widetilde \upu^{v,\veps}$ be as in Lemma~\ref{lem-limit-truc-pert}(ii).  By that result, we  know that 
 $$ \Lg^v\widetilde \upu^{v,\veps} (x)+ r^{\veps,v}(x)\widetilde  \upu^{v,\veps} =\Lambda_v[r^\veps]\widetilde \upu^{v,\veps}(x),\,\, \text { for } x\in \RR^d\,.$$
  From Theorem~\ref{thm-pert-limit}(i), we know that $\Lambda_v[r^\veps]\to\Lambda_v[r]$, as $\veps\to 0$. Therefore, using standard elliptic regularity theory, we can infer that along a subsequence denoted again by $\veps$, $\widetilde \upu^{v,\veps} \to \widetilde \upu^v$ {strongly in }$W^{1,p}_{\text{loc}}(\RR^d)$ and that $\widetilde \upu^v$ satisfies 
  \begin{align}\label{eq-tilde-upu-1} \Lg^v \widetilde \upu^v(x)+r^v(x)\widetilde \upu^v(x)=\Lambda_v[r]\widetilde\upu^v(x), \,\,\text{ for }x\in \RR^d\,.\end{align}
 To prove that $\widetilde \upu^v$ is positive, we recall that for $p>d$, $W^{1,p}(B_R)$ is continuously embedded in $\cC^{0,\gamma}(B_R)$ with $\gamma=1-\frac{d}{p}$. 
 This implies that $\widetilde \upu^{v,\veps}$ converges to $\widetilde \upu^v$, uniformly on compact sets of $\RR^d$ and for any $R>0$, an application of  Harnack's inequality (\cite[Theorem 8.20]{gilbarg1977elliptic}) and the fact that $\widetilde \upu^\veps>0$, give us
$$ \sup_{x\in B_{\bar R}}\widetilde \upu^{v,\veps}(x)\leq C\inf_{x\in B_{\bar R}} \widetilde \upu^{v,\veps}(x), \,\, \text{ for $\bar R < \frac{R}{4}$}$$
with $C$ being  independent of $R$. This proves the positivity of $\widetilde \upu^v$. Hence, the claim of the lemma is proved. 
  \end{proof}

Let  $\upu^{v}=\log {\widetilde \upu^{v}}$ and $\varpi^v(\cdot)=\Sigma(\cdot)\transp\nabla \upu^v(\cdot)$  with ${\widetilde \upu^{v}}$ be as in Lemma~\ref{lem-ground-diffusion}.
   \begin{proposition}\label{prop-ground-diffusion}
The process $Z$ under $v\in\Usm^{*,\beta}$ and $\varpi^v$ is recurrent. 
 \end{proposition}

 \begin{proof} To begin with, let $\upu^{v,\veps}=\log {\widetilde \upu^{v,\veps}}$ and $\varpi^{v,\veps}(\cdot){\doteq}\Sigma(\cdot)\transp\nabla \upu^{v,\veps}(\cdot)$ with ${\widetilde \upu^{v,\veps}}$ be as in Lemma~\ref{lem-limit-truc-pert}(ii). Since $X$  under $v\in\Usm^{*,\beta}$ is positive recurrent, we can conclude using \cite[Lemma 2.1 (a $\Leftrightarrow$ b)]{AB18} that $\inf_{x\in \RR^d}\widetilde  \upu^{v,\veps}(x)>0$. It is easy to observe that $\upu^{\veps,v}$ satisfies 
\begin{align}\label{eq-widehat-1}\max_{w\in \RR^d}\big\{ \widehat \Lg^{v,w} \upu^{v, \veps} + r^{\veps,v}(x)-\frac{1}{2}\|w\|^2\big\}= \Lambda_{v}[r^\veps]\end{align}
and that $\varpi^{v,\veps}(\cdot)$ is the maximizer in the above display.   To establish the desired  recurrence of $Z$ under $v\in\Usm^{*,\beta}$ and $\varpi^v$, it suffices to establish
  \begin{align}\label{eq-pos-rec} \limsup_{T\to\infty}\frac{1}{T} \E_x^{v,\varpi^{v}}\Big[\int_0^T h^v(Z_t) \D t\Big]\leq {M_4}\,.\end{align}
This is what we show below.

Recall~\eqref{eq-widehat-1}. Applying It\^o's formula to $\upu^{v,\veps}(Z_t)$ under $U_t=v(Z_t)$ and $w_t= \varpi^{v,\veps}(Z_t)$ up to $t= T\wedge \tau_R$ (recall that $\tau_R$ is the first exit time of $B_R$ by the process $Z$) gives us
\begin{align}\label{eq-ito-int}
\E^{v,\varpi^{v,\veps}}_x\big[\upu^{v,\veps}(Z_{T\wedge \tau_R})\big]&= \upu^{v,\veps}(x)+ \E_x^{v,\varpi^{v,\veps}}\Big[\int_0^{T\wedge \tau_R} \big(\Lambda_v[r^\veps]- r^{\veps,v}(Z_t)-\frac{1}{2}\| \varpi^{v,\veps}(Z_t)\|^2\big) \D t\Big]\,.
\end{align} 
In the following, we first take $R\to\infty$, and then take $T\to\infty$. First, we argue that the process $Z$ under $v$ and $\varpi^{v,\veps}$ is exponentially ergodic. Recall that as $v\in \Usm^{*,\beta}$ and  $r^{\veps,v}$ is inf-compact with $\Lambda_v[r^\veps]<\infty$,  Lemma~\ref{lem-ui} holds. Then, applying \cite[Theorem 1.8]{AB18}, we immediately conclude the exponential ergodicity of the process $Z$ under $v$ and $\varpi^{v,\veps}$. Moreover, from the proof of \cite[Theorem 1.8]{AB18} (especially, (2.52) and (2.53) in that paper), we can also conclude that there exist an inf-compact function $\sV_{v,\veps}\in W^{2,p}_{\text{loc}}(\RR^d)$, $p>d$ and constants $\eta_1,\eta_2, K_1,K_2>0$ satisfying 
$$ \widehat \Lg^{v,\varpi^{v,\veps}} \sV_{v,\veps}(x)\leq -\eta_1 \sV_{v,\veps}(x)+K_1 $$
\begin{center}and\end{center}
$$ \sV_{\veps,v}(x)\geq K_2 \big(\widetilde\upu^{v,\veps}(x)\big)^{\eta_2},\text{ for $x\in \RR^d$}\,. $$ 
From the above lower bound on $\sV_{\veps,v}$, we can conclude that $\upu^{v,\veps}=\log \widetilde \upu^{v,\veps}= \fro\big(\sV_{\veps,v}\big)$. From here, using \cite[Lemma 3.7.2 (ii)]{arapostathis2012ergodic}, we can conclude that 
$$ \lim_{R\to\infty} \E^{v,\varpi^{v,\veps}}_x\big[\upu^{v,\veps}(Z_{T\wedge \tau_R})\big]= \E^{v,\varpi^{v,\veps}}_x\big[\upu^{v,\veps}(Z_{T})\big]\quad \text{and}\quad \lim_{T\to\infty}\frac{1}{T}\E^{v,\varpi^{v,\veps}}_x\big[\upu^{v,\veps}(Z_{T})\big]=0\,. $$
Going back to~\eqref{eq-ito-int}, we take $R\to\infty$ and conclude that 
\begin{align*}
\E^{v,\varpi^{v,\veps}}_x\big[\upu^{v,\veps}(Z_{T})\big]&= \upu^{v,\veps}(x)+ \E_x^{v,\varpi^{v,\veps}}\Big[\int_0^{T} \big(\Lambda_v[r^\veps]- r^{\veps,v}(Z_t)-\frac{1}{2}\| \varpi^{v,\veps}(Z_t)\|^2\big) \D t\Big]\,.
\end{align*} 
Now, rearranging and dividing the above display by $T$, and taking $T\to\infty$, gives us
$$ \Lambda_v[r^\veps] =\limsup_{T\to\infty} \frac{1}{T}\E_x^{v,\varpi^{v,\veps}}\Big[\int_0^{T} \big(r^{\veps,v}(Z_t)-\frac{1}{2}\| \varpi^{v,\veps}(Z_t)\|^2\big) \D t\Big]\,.$$

We observe that  using Lemma~\ref{lem-comp-X-w-pert} in conjunction with the above display, we have 
	$$\sup_{0<\veps<\veps_0} \limsup_{T\to\infty}\frac{1}{T} \E_x^{v,\varpi^{v,\veps}}\Big[\int_0^T h^v(Z_t) \D t\Big]\leq {M_4} \,.$$
	From the above display, denoting MEM of $Z$ (under $v$ and $\varpi^{v,\veps}$) by $\pi_{T,\veps}$, we can conclude that  $\{\pi_{T,\veps}\}_{T,\veps}$ is tight in both $T$ and $\veps$. From the lower-semicontinuity of $ \pi\mapsto \int_{\RR^d\times \RR^d} h(x) \D \pi(x,w)$, we can infer that 
along a subsequence again denoted by $T$, $\pi_{T,\veps}$ converges weakly to $\pi_\veps$ such that 
$$\int_{\RR^d\times \RR^d} h^v(x)\D \pi_\veps(x,w)\leq \liminf_{T\to\infty}\frac{1}{T} \E_x^{v,\varpi_\veps}\Big[\int_0^T h^v(Z_t) \D t\Big]\leq {M_4}\,.$$
Now, along a subsequence again denoted by $\veps$, $\pi_\veps$ converges weakly to $\pi$ such that 
 $$\int_{\RR^d\times \RR^d} h^v(x)\D \pi(x,w)\leq\liminf_{\veps\to 0} \int_{\RR^d\times \RR^d} h^v(x)\D \pi_\veps(x,w)\leq {M_4}\,.$$
It only remains to  show that $\pi$ is an invariant measure corresponding to the process $Z$ under $v$ and $\varpi^v$. To do this, it suffices to show (see \cite[Lemma 2.6.16]{arapostathis2012ergodic} that for every smooth function $f:\RR^d\rightarrow \RR$ with compact support, we have 
$$ \int_{\RR^d}\widehat \Lg^{v,\varpi^v} f(x)\pi(\D x)=0\,. $$
Fix a smooth function $f:\RR^d\rightarrow \RR$ with compact support. Since $\pi_\veps$ is the unique invariant measure of $Z$ under $v$ and $\varpi^{v,\veps}$, we clearly have 
$$ \int_{\RR^d}\widehat \Lg^{v,\varpi^{v,\veps}} f(x)\pi_\veps(\D x)=0\,. $$ 
Now consider 
\begin{align*} \int_{\RR^d}\widehat \Lg^{v,\varpi^{v,\veps}} f(x)\pi_\veps(\D x)&= \int_{\RR^d}  \Lg ^{v} f(x) \pi_\veps(\D x) + \int_{\RR^d} \varpi^{v,\veps}(x)\cdot \nabla f(x) \pi_\veps (\D x).
\end{align*}
Since the integrand in the  first term of the right hand side  is independent of $\veps$ and is smooth with compact support, the weak convergence of $\pi_\veps $ to $\pi$  gives us  $$\lim_{\veps\to 0 } \int_{\RR^d}  \Lg ^{v} f(x) \pi_\veps(\D x)=\int_{\RR^d}  \Lg ^{v} f(x) \pi(\D x)\,.$$ 
From Lemma~\ref{lem-ground-diffusion}, we know that $\widetilde \upu^{v,\veps}$ converges $\widetilde \upu^v$ weakly in $W^{2,p}_{\text{loc}}(\RR^d)$, for $p>d$. This means that the family $\{\widetilde \upu^{v,\veps}\}_{\veps>0}$ is bounded in $W^{2,p}_{\text{loc}}(\RR^d)$, for $p>d$. Since $W^{2,p}_{\text{loc}}(\RR^d)$ is continuously embedded in $C^{1,\gamma}$  with $\gamma=1-\frac{d}{p}$, we can conclude using Arzela-Ascoli theorem that $\varpi^{v,\veps}$ converges to $\varpi^v$, uniformly on compact sets of $\RR^d$. Consequently, this implies that 
$$ \lim_{\veps\to 0} \int_{\RR^d} \varpi^{v,\veps}(x)\cdot \nabla f(x) \pi_\veps (\D x)=  \int_{\RR^d} \varpi^{v}(x)\cdot \nabla f(x) \pi (\D x) \,. $$
From the above two displays, we conclude that $\pi$ is an invariant measure of $Z$ under $v$ and $\varpi^{v,\veps}$. This proves~\eqref{eq-pos-rec} and the proposition. 
\end{proof}

  We observe that applying It\^o-Krylov's lemma, from \eqref{eq-tilde-upu-1},  one can immediately show that 
  $$ \E_x^v\Big[\exp\Big(\int_0^{\widecheck \tau_R} (r^v(X_t) -\Lambda_v[r])\D t\Big)\widetilde \upu^v(X_{\widecheck \tau_R})\Big]\leq \widetilde \upu^v(x), \text{ for $x\in B_R^c$}\,.$$
However, this is not sufficient for our purposes \emph{i.e.,} proving uniqueness of solution to~\eqref{eqn-HJB} and Theorem~\ref{thm-diffusion}(iii), and {we} would require the equality to hold. 
The authors in \cite{ari2018strict} show that the equality holds if and only if $Z^v$ is recurrent. This fact is used in proving uniqueness of solution to~\eqref{eqn-HJB} and in the proof of Theorem~\ref{thm-diffusion}(iii).

\smallskip

\subsection{Proof of Theorem~\ref{thm-diffusion}(i)}\label{sec-p1}

 The proof of existence involves a direct application of the standard elliptic regularity theory and we provide it for completeness. 
\begin{proposition}\label{prop-limit}  As $\veps \to 0$, the pair {$(V^\veps, \lsm[r^\veps])$ (with $V^\veps$ obtained in Theorem~\ref{thm-HJB-pert})} converges {along a subsequence (again denoted by $\veps$)} to a pair $(V,\lsm[r])$ such that $V$ is a positive  function in $\calC^2(\RR^d)$, $V^\veps\to V$ {strongly in} $W^{1,p}_{\text{loc}}(\RR^d)$, { $p>1$}  and satisfies 
	\begin{align}\label{eq-hjb-limit}
	\min_{u \in \bU} \bigl[\Lg^{u} V(x) + r(x,u)\,V(x)\bigr] \;=\; \lsm[r]V(x)\,, \text{ for } x\in \RR^d\,.
	\end{align}
\end{proposition}
\begin{proof} From Theorem~\ref{thm-pert-limit}(ii), it is clear that $\lsm[r^\veps]\to \lsm[r]$. Therefore, it only remains to show that $V^\veps$ is convergent {strongly in} $W^{1,p}_{\text{loc}}(\RR^d)$, {and the limit point $V$ is positive, lies in $\cC^2(\RR^d)$ and} satisfies~\eqref{eq-hjb-limit}.
First recall that the pair $(V^\veps,\Lambda[r^\veps])$ satisfies 
\begin{align}\label{eq-hjb-pert}
	\min_{u \in \bU} \bigl[\Lg^{u} V^\veps(x) + r^\veps(x,u)\,V^\veps(x)\bigr] \;=\; {\lsm}[r^\veps]\,V^\veps(x)\,, \text{ for } x\in \RR^d\,.
	\end{align}
Now let $ v^\veps\in {\Usm^{*,\beta}}$ be  such that 
\begin{align*}
\min_{u \in \bU} \bigl[\Lg^{u} V^\veps(x) + r^\veps(x,u)\,V^\veps(x)\bigr]=\Lg^{v^\veps} V^\veps(x) + r^{\veps,v^\veps}(x)\,V^\veps(x), \text{ a.e. $x\in \RR^d$.}
\end{align*}
 Then 
$$ \Lg^{v^\veps} V^\veps(x)+r^{\veps,v^\veps} (x)V^\veps(x)= \lsm[r^\veps] V^\veps(x)\,, \text{ for } x\in \RR^d\,.$$
For $R>0$, from Harnack's inequality (\cite[Theorem 8.20]{gilbarg1977elliptic}) and the fact that $V^\veps>0$, we can conclude that 
$$ \sup_{x\in B_{\bar R}}V^\veps(x)\leq C_4\inf_{x\in B_{\bar R}} V^\veps(x), \text{ for $\bar R < \frac{R}{4}$}$$
with $C_4$ being  independent of $R$.
 Since $V^\veps(0)=1$, we have $\|V^\veps\|_{\infty, B_{\bar R}}$ bounded uniformly in $R>0$. Similarly, using the fact that $V^\veps>0$ on any compact set, we can also conclude that there is a uniform positive lower bound of $V^\veps$ on every compact set of $\RR^d$.  Using the above argument for $B_{\bar R+\delta}$ with $\delta>0$ such that $4\bar R+\delta< R$, we have that $\|V^\veps\|_{\infty, B_{\bar R+\delta}}$ is bounded uniformly in $R>0$ (say, by $C_5$).
 Now consider
$$ \|V^\veps\|_{p,B_{\bar R+\delta}}+\lsm[r^\veps]\|V^\veps\|_{p,B_{\bar R+\delta}}\leq \text{Vol}(B_{\bar R+\delta })\big(1+\lsm[r^\veps]\big) \|V^\veps\|_{\infty, B_{\bar R+\delta}}\,.$$
Here, $\text{Vol}(D)$ is the Lebesgue measure of domain $D$. Now using \cite[Theorem 9.11]{gilbarg1977elliptic}, we can conclude that for some $C_6>0$,
\begin{align}\nonumber \|V^\veps\|_{2,p, B_{\bar R}}&\leq C_6 ( \|V^\veps\|_{p, B_{R+\delta}} + \lsm[r^\veps] \|V^\veps\|_{p,B_{R+\delta}})\\\nonumber
&\leq C_6 \text{Vol}(B_{ R+\delta })\big(1+\lsm[r^\veps]\big) \|V^\veps\|_{\infty, B_{ R+\delta}}\\\label{eq-unif-w2p-bound}
&\leq \text{Vol}(B_{ R+\delta })\big(1+\lsm[r^\veps]\big) C_5 C_6\,.
 \end{align}

 Since  $\{\lsm[r^\veps]\}_{\veps<\veps_0}$ is convergent as $\veps\downarrow 0 $, $\{V^\veps\}_{\veps<\veps_0}$ is bounded in $W^{2,p}(B_{\bar R})$ (for every $1<p<\infty$), uniformly for $R> 0$.  Therefore, using a diagonalization argument, we can pick a subsequence denoted again by $\veps$ such that 
 $$ V^\veps \text{ converges weakly in $W^{2,p}_{\text{loc}}(\RR^d)$ to some } V\in W^{2,p}_{\text{loc}}(\RR^d) \text{ as $\veps\to 0$}\,. $$
  This means that for any $g\in L^{\frac{p}{p-1}}_{\text{loc}}(\RR^d)$, we have
$$\int_{K} g(x)\frac{\partial^2}{\partial x_i\partial x_j} V^\veps(x)\D x\to  \int_{K} g(x)\frac{\partial^2}{\partial x_i\partial x_j} V(x)\D x, \text{ for every compact set $K\subset \RR^d$}.$$
 Using Kondrachov's theorem, we know that $ W^{2,p}_{\text{loc}}(\RR^d)$ is compactly embedded in $W^{1,p}_{\text{loc}}(\RR^d)$. Therefore, $V^\veps$ converges to $V$ strongly in $W^{1,p}_{\text{loc}}(\RR^d).$ Now observe that, for any $f,h\in W^{1,p}_{\text{loc}}(\RR^d)$, we have 
\begin{align*}\Big| &\min_{u\in \bU}\big(b(x,u)\cdot \nabla f(x)+ r(x,u)f(x)\big) -\min_{u\in \bU}\big(b(x,u)\cdot \nabla h(x)+ r(x,u)h(x)\big)\Big|\\
&\leq \max_{u\in \bU}\Big| \big(b(x,u)\cdot \nabla f(x)+ r(x,u)f(x)\big) -\big(b(x,u)\cdot \nabla h(x)+ r(x,u)h(x)\big)\Big|\,. \end{align*}
Using the fact that $\{V^\veps\}_{\veps<\veps_0}$ converges strongly in $W^{1,p}_{\text{loc}}(\RR^d)$, the above observation gives us that 
$$ \min_{u\in \bU}\big(b(\cdot,u)\cdot \nabla V^\veps(\cdot)+ r(\cdot,u)V^\veps(\cdot)\big) \text{ converges in $L^p_{\text{loc}}(\RR^d)$}\,,$$
and  $\lsm[r^\veps]V^\veps$ converges {strongly in} $L^p_{\text{loc}}(\RR^d).$ Combining the above results, we have the following: {for $g\in L^{\frac{p}{p-1}}_{\text{loc}}(\RR^d)$,}
\begin{align*}
0&=\lim_{\veps\to 0} \int_K g(x)\Big( \sum_{i,j=1}^d {A_{ij}}(x)\frac{\partial^2}{\partial x_i\partial x_j} V^\veps(x)\\
&\qquad\qquad +\min_{u\in \bU}\big(b(x,u)\cdot \nabla V^\veps(x)+ r(x,u)V^\veps(x)\big)-\lsm[r^\veps]V^\veps(x)  \Big)\D x\\
&= \int_Kg(x)\Big( \sum_{i,j=1}^d {A_{ij}}(x)\frac{\partial^2}{\partial x_i\partial x_j} V(x) +\min_{u\in \bU}\big(b(x,u)\cdot \nabla V(x)+ r(x,u)V(x)\big)-\lsm[r] V(x)  \Big)\D x\,.
\end{align*}
From the arbitrariness of $g$, we know that $V$ satisfies 
\begin {align}\label{eq-minimizers}
\min_{u\in \bU}\Big\{\Lg^uV(x)+r(x,u)V(x)\Big\}=\lsm[r] V(x), \text{ a.e. in $x\in \RR^d$.}
\end{align}

We now improve the regularity of $V$ to $\cC^2(\RR^d)$. To do that, we note that $V\in W^{2,p}_{\text{loc}}(\RR^d)$, $\forall p>1$ and in particular, for $p>d$. Therefore, from the compact embedding of $W^{2,p}_{\text{loc}}(\RR^d)$ in $\cC^{1,\gamma}(K)$, $\gamma<1-\frac{d}{p}$, we have $V\in \cC^{1,\gamma}(K)$, for every compact set $K\subset \RR^d$. Using this fact, we can conclude that 
$$ \min_{u\in \bU} \big\{ b(\cdot,u)\cdot \nabla V(\cdot)+r(\cdot,u)V(\cdot)\big\}\in \cC^{0,\gamma}(K)\,.$$
Therefore, from \cite[ Theorem A.2.9]{arapostathis2012ergodic}, we conclude that $V\in \cC^2(K)$, for every compact set $K\subset \RR^d$ and thereby, conclude that $V\in \cC^2(\RR^d)$.  This now gives us the desired result.
\end{proof}
 We next prove the uniqueness of a solution $V$ to \eqref{eqn-HJB}. 
\begin{proposition}\label{prop-hjb-uniqueness}
Let ${\breve {V}}\in W^{2,p}_{\text{loc}}(\RR^d)$, {$p>d$} be a positive function that satisfies~\eqref{eqn-HJB} { with ${\breve V}(0)=1$, }
then ${\breve V}=V$. 
\end{proposition}
\begin{proof}
Recall that function $V$ from Proposition~\ref{prop-limit} satisfies~\eqref{eqn-HJB}. 
Choose $v^*\in \Usm$ such that 
\begin{align}
\Lg^{v^*} V(x)+r^{v^*}(x)V(x)=\min_{u\in \bU}\Big\{\Lg^uV(x)+r(x,u) V(x)\Big\}, \text{ for a.e. $x\in \RR^d$}\,.
\end{align}
Using Proposition~\ref{prop-minimizers-optimal} that follows later, we know that $\lsm[r]=\Lambda_{v^*}[r]$.

{Consider the process $Z$ under $v^*$ and $\varpi^{v^*}$. From Proposition~\ref{prop-ground-diffusion}, we have the recurrence of $Z$. Using \cite[Theorem 2.1 (c)]{ari2018strict}, we conclude that $\lambda_{v^*}^*[r]=\Lambda_{v^*}[r]=\lsm[r]$ and using \cite[Theorem 2.1(a)]{ari2018strict}, we conclude that }
\begin{align}\label{eq-stoc-rep-hjb}
V(x)= \E_x^{v^*}\Big[\exp\Big(\int_0^{\widecheck \tau_R} \big(r^{v^*}(X_t)-\lsm[r]\big) \D t\Big) V(X_{\widecheck \tau_R})\Big], \text{ for $x\in B_R^c$}\,.
\end{align} 
Suppose ${\breve V}$ is another positive solution to~\eqref{eqn-HJB} \emph{i.e.,} 
$$ \min_{u\in \bU}\Big\{\Lg^u {\breve V}(x)+r(x,u){\breve V}(x)\Big\}=\lsm[r] {\breve V}(x), \text{ for $x\in \RR^d$}\,.$$ 
It is clear that $v^*\in \Usm$ chosen above satisfies
$$ \Lg^{v^*} {\breve V}(x)+r^{v^*}(x){\breve V}(x)\geq \lsm[r]{\breve V}(x), \text{ for $x\in \RR^d$}\,.$$
From the above display, usual application of It\^o-Krylov's lemma and then followed by Fatou's lemma gives us
\begin{align}\label{eq-stoc-rep-sub}
{\breve V}(x)\geq  \E_x^{v^*}\Big[\exp\Big(\int_0^{\widecheck \tau_R} \big(r^{v^*}(X_t)-\lsm[r]\big) \D t\Big) {\breve V}(X_{\widecheck \tau_R})\Big], \text{ for $x\in B_R^c$}\,.
\end{align} 
From the fact that $V$ satisfies~\eqref{eq-stoc-rep-hjb} and the above display, we have
\begin{align*}
{\breve V}(x)\geq V(x) \min_{y\in B_R}\Big(\frac{{\breve V}(y)}{V(y)}\Big) , \text{ for $x\in B_R^c$}\,.
\end{align*}
Clearly, if ${\breve V}>V$ on $B_R$, then ${\breve V}>V$ on $\RR^d$. So multiplying {$V$} by $\min_{y\in B_R}\Big(\frac{{\breve V}(y)}{ V(y)}\Big)$ and denoting again by $V$, we can ensure that ${\breve V}$ touches $V$ from above at points in $\text{arg}\min_{y\in B_R}\Big(\frac{{\breve V}(y)}{ V(y)}\Big)$. This means that ${\breve V}\geq V$ on $\RR^d$ and its minimum is achieved in $B_R$. From~\eqref{def-poiss-eq} and the definition of ${\breve V}$, we have, for $x\in \RR^d$,
$$ \Lg^v({\breve V}-V)(x)- \big(r^{v^*}(x)-\lsm[r]\big)^- ({\breve V}-V)(x)= -\big(r^{v^*}(x)-\lsm[r]\big)^+({\breve V}-V)(x)\leq 0\,.$$
Therefore, using the strong maximum principle \cite[Theorem 9.6]{gilbarg1977elliptic}, we have ${\breve V}=V$. This completes the proof.
\end{proof}

\smallskip

\subsection{Proof of Theorem~\ref{thm-diffusion}(ii)}\label{sec-p2}
The content of this part of the theorem involves characterization of stationary Markov controls that are optimal.  In other words, we will show that a stationary Markov control is optimal if and only if it is a minimizer of~\eqref{eqn-HJB}, \emph{i.e.,} it satisfies~\eqref{eqn-optimality1}. 

To prove that the minimizers are optimal stationary Markov controls, it is sufficient to show that either $\inf_{x\in \RR^d}\widetilde V(x)>-\infty$ or that the negative part of $\widetilde V$ is appropriately small. 
 Recall that since $\inf_{x\in \RR^d}V^\veps(x)>0$, $\inf_{x\in\RR^d}\widetilde V^\veps(x)>-\infty$. In contrast, it is  not a priori clear if $\inf_{x\in \RR^d} {\widetilde V(x)}>-\infty$. However, it turns out that the expectation of  $T^{-1}\widetilde V(Z_T)$ is greater than zero for large $T$ - this is sufficient for us.  This is done by showing that the negative part of $T^{-1}\widetilde V(Z_T)$ has negligible expectation for large $T$. This is made precise in Lemma~\ref{lem-V-}. In Lemma~\ref{lem-minimizer-stable}, we analyze stationary Markov controls that are minimizers and show that all the results that are proved until now for $v\in {\Usm^{*,\beta}}$ are applicable with similar arguments. These include all the results in Section~\ref{sec-est} to Section~\ref{sec-p1}. This is important because it is not  a priori clear even if the ERSC cost for such controls is finite and if one can apply the analysis for ${\Usm^{*,\beta}}$ to these controls.

  From now on, we set $$ \widetilde V(\cdot) \doteq \log V(\cdot) \quad\text{ and } \quad\omega(\cdot) \doteq \Sigma(\cdot)\transp \nabla \widetilde V(\cdot)$$
 and recall that 
 $$ \widetilde V^\veps (\cdot) =\log V^\veps (\cdot) \quad\text{ and }\quad \omega^\veps (\cdot) = \Sigma(\cdot)\transp \nabla \widetilde V^\veps (\cdot)\,.$$
 {Recall that $ \frV(x)=\log \sV(x)$.}  
\begin{lemma}\label{lem-V-}  For the function $\widetilde V$, we have $\widetilde V^-\in \mathfrak{o}(\frV)$. 
\end{lemma}

\begin{proof} The proof of this argument follows closely the arguments in the proof of \cite[Lemma 3.10]{ABP15}.  Fix $l>0$ and choose $w^*\in \Wsm(l)$. From Theorem~\ref{thm-lin-rep-pert}(ii), we have the following: for $v^\veps\in \Usm^{o,\veps}$
\begin{align}\label{eq-lin-rep-a}
		\widetilde V^\veps(x)&\geq \E_x^{v^\veps,w^*}\Big[\int_0^{\widecheck \tau_R} \Big(r^{\veps,v^\veps}(Z_t)-\frac{1}{2}\|w^*(Z_t)\|^2- \lsm[r^\veps]\Big)\D t+\widetilde V^\veps(Z_{\widecheck \tau_R})\Big]\,.		\end{align}
		Since $w^*\in \Wsm(l)$, following the arguments of the proof of Lemma~\ref{lem-comp-X} gives us 
		$$ \sup_{{0<\veps<\veps_0}} \limsup_{T\to\infty}\frac{1}{T}\E_x^{v^\veps,w^*} \Big[\int_0^T h^{v^\veps}(Z_t) \D t\Big]<\infty\,.$$
		From  \cite[Lemma 3.3.4 (iii$\implies$i)  ]{arapostathis2012ergodic}, we can conclude that  $\sup_{{0<\veps<\veps_0}} \E_x^{v^\veps,w^*}[\int_0^{\widecheck \tau_R} h^{v^\veps}(Z_t) \D t]<\infty$. Taking $\veps\to 0$,  we know that  $v^\veps\to v^*$, for some $v^*\in \Usm$ in the topology of Markov controls (see \cite[Section 2.4]{arapostathis2012ergodic} for the definition)
such that 
$$ \Lg^{v^*} V(x)+r^{v^*}(x) V(x)= \min_{u\in \bU} \Big\{\Lg^uV(x)+ r(x,u) V(x)\Big\}\,, \text{ for a.e. } x\in \RR^d\,.$$ This follows from \cite[Lemma 2.4.3]{arapostathis2012ergodic}.  Using \cite[Lemma 3.8]{ABP15} it follows that $ \E_x^{v^\veps,w^*}[\widecheck \tau_R]\to  \E_x^{v^*,w^*}[\widecheck \tau_R]$, as $\veps\to 0$.
Therefore, taking $\veps\to 0$ in~\eqref{eq-lin-rep-a} and  using the fact that $\sup_{x\in \RR^d}\|w^*(x)\|\leq l$ (as $w^*\in \Wsm(l)$) gives us
		\begin{align}\nonumber
		\widetilde V(x)&\geq \E_x^{v^*,w^*}\Big[\int_0^{\widecheck \tau_R} \Big(r^{v^*}(Z_t)-\frac{1}{2}\|w^*(Z_t)\|^2- \lsm[r]\Big)\D t+\widetilde V(Z_{\widecheck \tau_R})\Big]\\\label{eq-lin-rep-r}
		&\geq \E_x^{v^*,w^*}\Big[\int_0^{\widecheck \tau_R} \Big(r^{v^*}(Z_t)-\frac{l^2}{2}- \lsm[r]\Big)\D t\Big]+\inf_{y\in {\partial B_R}} \widetilde V(y)\,.		\end{align}
Also using the fact that $r\geq 0$, we have 		
\begin{align}\label{eq-lin-rep-l}
\widetilde V(x)&\geq \E_x^{v^*,w^*}\Big[\int_0^{\widecheck \tau_R} \Big(-\frac{l^2}{2}- \lsm[r]\Big)\D t\Big] {+\inf_{y\in {\partial B_R}} \widetilde V(y)}\,. 
\end{align}
This in turn, gives us
$$ \widetilde V^-(x) \leq  \Big(\lsm[r]+\frac{l^2}{2} \Big)\E_x^{v^*,w^*}[\widecheck \tau_R] -\inf_{y\in {\partial B_R}} \widetilde V(y)\,.$$

Applying It\^o's formula to~\eqref{eq-lyap-extended} with $u\equiv v^*$ and $w\equiv w^*$, {and using the fact that $\frV\geq 0$}, we have
\begin{align*}
&\E_x^{v^*,w^*}\Big[ \int_0^{\wtau_R}{ \bar h^{v^*}}(Z_t) \Ind_{\cH^c}(Z_t,v^*(Z_t))\D t\Big] \\
&\leq C_1\wedge C_2 \E_x^{v^*,w^*}[\wtau_R]+ C_3 \E_x^{v^*,w^*}\Big[\int_0^{\wtau_R} r^{v^*}(Z_t)\D t\Big]-\frac{1}{4} \E_x^{v^*,w^*}\Big[\int_0^{\widecheck\tau_R} \|\Sigma(Z_t)\transp{\nabla\frV(Z_t)}\|^2\D t\Big]\\
&\quad+ {\E_x^{v^*,w^*}}\Big[ \int_0^{\wtau_R} \|w^*(Z_t)\|^2\D t\Big] +\frV(x)\,.
\end{align*}
Adding $ \E_x^{v^*,w^*}\Big[\int_0^{\wtau_R} r^{v^*}(Z_t)\Ind_{\cH}(Z_t,v^*(Z_t))\D t\Big]$ on both sides and using the fact that \\$\sup_{x\in \RR^d}\|w^*(x)\|\leq l$ gives us
\begin{align*}
\frac{1}{2}\E_x^{v^*,w^*}\Big[ \int_0^{\wtau_R} h^{v^*}(Z_t) \D t\Big]
&\leq  (C_1\wedge C_2 +{l^2}+1) \E_x^{v^*,w^*}[\wtau_R]\\
&\qquad+ (C_3 +1)\E_x^{v^*,w^*}\Big[\int_0^{\wtau_R} r^{v^*}(Z_t)\D t\Big]+\frV(x)\,.
\end{align*}
In the above, we also use~\eqref{eq-inf-comp-1}. Using~\eqref{eq-lin-rep-r},  we then have
\begin{align*}
&\frac{1}{2}\E_x^{v^*,w^*}\Big[ \int_0^{\wtau_R} h^{v^*}(Z_t) \D t\Big]\\
&\leq  (C_1\wedge C_2 +{l^2}+1) \E_x^{v^*,w^*}[\wtau_R]\\
&\qquad+ (C_3 +1)\Big(\widetilde V(x) + (\lsm[r]+\frac{l^2}{2}) \E_x^{v^*,w^*}[\widecheck \tau_R] -\inf_{y\in {\partial B_R}} \widetilde V(y) \Big)+\frV(x)\,.
\end{align*}
From here, following exactly the same arguments as those in the proof of \cite[Lemma 3.8]{ABP15}, we get the result.
\end{proof}

\begin{lemma}\label{lem-minimizer-stable}
Suppose $v\in \Usm$ satisfies~\eqref{eqn-optimality1}. Then $\Lambda_v[r]<\infty$ and the conclusion of Proposition~\ref{prop-sup-v} holds for $v$.
\end{lemma}
\begin{proof} Fix $v\in \Usm$ that satisfies~\eqref{eqn-optimality1}.  Suppose that $\kappa\doteq \Lambda_v[r]<\infty.$ Then, observe that all the results in Section~\ref{sec-est} hold with ${\beta}$ replaced by $\kappa$. Subsequently, conclusion of Proposition~\ref{prop-sup-v} holds. Therefore, it only remains to show that $\kappa<\infty$. This is achieved using Proposition~\ref{prop-var-cost}(i) and Remark~\ref{rem-non-negative} which give us 
\begin{align*}
\Lambda_v[r]=\limsup_{T\to\infty}\sup_{w\in \cA} \frac{1}{T}\E_x^{v,w}\Big[\int_0^T\Big(r^v(Z_t)-\frac{1}{2}\|w_t\|^2\Big)\D t\Big]\,.
\end{align*}
For any $\delta>0$ and $T>0$, let $w^*=w^*(\delta,T)\in \cA$ be such that 
$$\sup_{w\in \cA} \frac{1}{T}\E_x^{v,w}\Big[\int_0^T\Big(r^v(Z_t)-\frac{1}{2}\|w_t\|^2\Big)\D t\Big]\leq  \frac{1}{T}\E_x^{v,w^*}\Big[\int_0^T\Big(r^v(Z_t)-\frac{1}{2}\|w^*_t\|^2\Big)\D t\Big]+\delta\,. $$
Suppose the following holds. 
\begin{align}\label{eq-fin-val}\limsup_{T\to\infty} \frac{1}{T}\E_x^{v,w^*}\Big[\int_0^Tr^v(Z_t)\D t\Big]<\infty \text{ and }  \limsup_{T\to\infty} \frac{1}{2T}\E_x^{v,w^*}\Big[\int_0^T\|w_t^*\|^2\D t\Big]<\infty\,.\end{align}
Then it is clear that 
$$ \Lambda_v[r]\leq \limsup_{T\to\infty}\frac{1}{T}\E_x^{v,w^*}\Big[\int_0^T\Big(r^v(Z_t)-\frac{1}{2}\|w^*_t\|^2\Big)\D t\Big]+\delta<\infty$$
and also proves the lemma. In the following, we show~\eqref{eq-fin-val}. To begin with, observe that since $\Lambda_v[r]\geq 0$,
\begin{align}\label{eq-bound-1} \frac{1}{2T}\E_x^{v,w^*}\Big[\int_0^T\|w^*_t\|^2\D t\Big]\leq  \frac{1}{T}\E_x^{v,w^*}\Big[\int_0^Tr^v(Z_t)\D t\Big]+\delta,\text{ for large $T$}\,. \end{align}
Now define a function $\mathscr{F}:\RR^d\rightarrow \RR_+$ as $\mathscr{F}(x)\doteq \frV(x) +\widetilde V(x)$. Since $\widetilde V^-\in \mathfrak{o}(\frV)$ from Lemma~\ref{lem-V-} and $\frV$ is inf-compact,  we can conclude that $\mathscr{F}$ is inf-compact (and consequently, uniformly bounded from below).  Since $v\in \Usm$ satisfies~\eqref{eqn-optimality1}, we have 
\begin{align*} \Lg^{v} V(x)+ r^{v}(x) V(x)= \Lambda_\text{SM}[r]V(x)\,, \text{ for } x\in \RR^d\,.\end{align*}
	The above equation becomes
	$$ \Lg^{v} \widetilde V(x) + r^{v}(x) +\frac{1}{2} \|\omega(x)\|^2= \lsm[r]\,, \text{ for } x\in \RR^d.$$
From the above display and the fact that $\frV$ satisfies~\eqref{eq-lyap-var-lin},  we can conclude that 
\begin{align*}
\widehat \Lg^{v,w} \mathscr{F}(x)&= C_1\wedge C_2-{\bar h^v}(x)\Ind_{\cH^c}(x,u)+ C_3r^v(x)\Ind_{\cH}(x,u) -\frac{1}{2}\|\Sigma(x)\transp \nabla \frV(x)\|^2\\
&\qquad +(\Sigma(x) w )\cdot \nabla \frV(x) 
 +\lsm[r]-r^{v}(x)-\frac{1}{2}\|\omega(x)\|^2 + (\Sigma(x)w)\cdot \nabla \widetilde V(x)\\
&\leq  C_1\wedge C_2+\lsm[r]-{\bar h^v}(x)\Ind_{\cH^c}(x,u)-(1- C_3) r^v(x)\Ind_{\cH}(x,u)-r^v(x)\Ind_{\cH^c}(x,u)\\
&\qquad -\frac{1}{2}\|\Sigma(x)\transp \nabla \frV(x)\|^2 +(\Sigma(x) w )\cdot (\nabla \frV(x)+\nabla \widetilde V(x))  -\frac{1}{2}\|\omega(x)\|^2\,.
\end{align*}
Since $0<C_3<1$, using~\eqref{eq-inf-comp-1} the above display reduces to 
\begin{align}\nonumber
\widehat \Lg^{v,w} \mathscr{F}(x)
&\leq  C_1\wedge C_2+\lsm[r] +1 -{\frac{1}{2}} h^v(x) -\frac{1}{2}\|\Sigma(x)\transp \nabla \frV(x)\|^2\\\label{eq-lyap-final}
&\qquad +(\Sigma(x) w )\cdot (\nabla \frV(x)+\nabla \widetilde V(x))  -\frac{1}{2}\|\omega(x)\|^2\,.
\end{align}
From here, following the arguments of the proof of Lemma~\ref{lem-comp-X}, we can conclude that 
$$ \limsup_{T\to\infty} \frac{1}{T} \E_x^{v,w^*}\Big[\int_0^T h^v(Z_t)\D t\Big]<\infty. $$
Since $r\leq h$,~\eqref{eq-bound-1} immediately implies~\eqref{eq-fin-val}. This completes the proof of the lemma.
\end{proof}

\begin{proposition}\label{prop-minimizers-optimal}
	Suppose $v\in\Usm $ and satisfies~\eqref{eqn-optimality1}. Then, we have 
	$ \Lambda_v[r]=\Lambda_{\text{SM}}[r]$.
In other words, $v$ is an optimal stationary Markov control and $\Usm^o$ is non-empty.
\end{proposition}
\begin{proof} 
	Choose $v\in \Usm$ such that~\eqref{eqn-optimality1} holds. From Lemma~\ref{lem-minimizer-stable}, we can conclude that for $\delta>0$,  there exists  large enough $l$ such that  $\Lambda_v[r]\leq  \sup_{w\in \Wsm(l)} \Lambda_{v,w}[r]+\delta$. { We then choose $w^*\in \Wsm(l)$ such that 
$ \sup_{w\in \Wsm(l)} \Lambda_{v,w}[r]\leq \Lambda_{v,w^*}[r]+\delta$.}
 From~\eqref{eq-lyap-final} and following the computations in the proof of Lemma~\ref{lem-exp-frv}, we can infer  that $$ \limsup_{T\to\infty} \frac{1}{T}\E_x^{v,w^*}\Big[\mathscr{F}(Z_T)\Big]\leq {M_6}, \text{ for some constant {$M_6>0$}}\,. $$ 
Since $\mathscr{F}=\frV+\widetilde V$ and  $\widetilde V^-\in \mathfrak{o}(\frV)$, from Lemma~\ref{lem-V-}, we can also infer that $\widetilde V^-\in \mathfrak{o}(\mathscr{F})$.  Using \cite[Corollary  3.7.2]{arapostathis2012ergodic}, this implies that  \begin{align}\label{eq-exp-} \lim_{T\to\infty} \frac{1}{T} \E_x^{v,w^*}\Big[\widetilde V^-(Z_T)\Big]=0\,.\end{align}	
	Now, it is clear from~\eqref{eqn-HJB}  that 
	\begin{align}\label{eq-v-opt} \Lg^{v} V(x)+ r^{v}(x) V(x)= \Lambda_\text{SM}[r]V(x)\,, \text{ for } x\in \RR^d\,.\end{align}
	The above equation becomes
	$$ \Lg^{v} \widetilde V(x) + r^{v}(x) +\frac{1}{2} \|\omega(x)\|^2= \lsm[r]\,, \text{ for } x\in \RR^d$$
	which equivalently can be written as 
	$$  \Lg^{v} \widetilde V(x) + r^{v}(x) +\max_{w\in \RR^d} \Big\{(\Sigma(x)w)\cdot \nabla \widetilde V(x) -\frac{1}{2} \|w\|^2\Big\}= \lsm[r]\,, \text{ for } x\in \RR^d\,.$$
	For a $\delta>0$, with $w^*\in \Wsm$ as chosen above, we have 
	$$  \widehat \Lg^{v,w^*} \widetilde V(x) + r^{v}(x) -\frac{1}{2} \|w^*(x)\|^2\leq  \lsm[r]\\,, \text{ for } x\in \RR^d\,.$$
 Applying It\^o's formula gives us
	\begin{align*}
	\E_x^{v,w^*}\big[\widetilde V (Z_{T\wedge \tau_R})\big]-\widetilde V(x)&{\leq} \E_x^{v, w^*}\Big[\int_0^{T\wedge \tau_R}\big( \lsm[r]- r^{v}(Z_t)+ \frac{1}{2}\|w^*(Z_t)\|^2\Big)\D t\Big]\,,
	\end{align*}
	and
		\begin{align*}
	-\E_x^{v,w^*}\Big[\widetilde V^-(Z_{T\wedge \tau_R})+ \int_0^{T\wedge\tau_R}r^{v}(Z_t)\D t\Big]-\widetilde V(x)&\leq \E_x^{v,w^*}\Big[\int_0^{T\wedge \tau_R}\big( \lsm[r]+ \frac{1}{2}\|w^*(Z_t)\|^2\Big)\D t\Big]\,.
	\end{align*}
	Taking $R\uparrow \infty$, we have
	$$ -\E_x^{v,w^*}[\widetilde V^-(Z_{T})]-\widetilde V(x)\leq \E_x^{v,w^*}\Big[\int_0^{T}\big( \lsm[r]- r^{v}(Z_t)+ \frac{1}{2}\|w^*(Z_t)\|^2\Big)\D t\Big]\,.$$
	From here dividing by $T$ and using~\eqref{eq-exp-}, we immediately have
	$$\Lambda_{v,w^*}[r]= \limsup_{T\to\infty}\frac{1}{T}\E_x^{v,w^*}\Big[\int_0^T\Big(r^{v}(Z_t)- \frac{1}{2}\|w^*(Z_t)\|^2\Big)\D t\Big]\leq \Lambda_\text{SM}[r]\,.$$
	But, from the choice of $w^*$, it is easy to see that $ \Lambda_{v}[r]-{2\delta \leq \Lambda_{v,w^*}}\leq \lsm[r]$.
Arbitrariness of $\delta>0$, then gives us the result.  
\end{proof}

To prove that the optimal stationary Markov controls are minimizers of~\eqref{eqn-HJB} \emph{i.e.,} satisfy~\eqref{eqn-optimality1}, we use Proposition~\ref{prop-2p-game} and Theorem~\ref{thm-lin-rep-pert} extensively. The proof {involves an} argument by contradiction.

\begin{lemma}\label{lem-tightness} Let $w^*=w^*(l)\in \Wsm(l)$ be as in Proposition~\ref{prop-2p-game}(v), for every $l>0$. Then
$$ \sup_{0<\veps<\veps_0}\limsup_{l\to\infty} \limsup_{T\to\infty} \frac{1}{T}\E_x^{v,w^*}\Big[\int_0^T h^v(Z_t)\D t\Big]<\infty\,, \text{ for every $v\in {\Usm^{*,\beta}}$}.$$ 
\end{lemma}
\begin{proof} First, observe that for $v\in {\Usm^{*,\beta}}$ and $w^*$ satisfying the hypothesis of the lemma, we have
\begin{align*} \rho^\veps_l&\leq  \limsup_{T\to\infty} \frac{1}{T}\E_x^{v,w^*} \Big[\int_0^T \Big(r^{\veps,v} (Z_t) -\frac{1}{2} \|w^*(Z_t)\|^2 \Big)\D t\Big]\end{align*}
which along a subsequence (again denoted by $T$) implies that 
 \begin{align*} \frac{1}{T}\E_x^{v,w^*} \Big[\int_0^T \frac{1}{2} \|w^*(Z_t)\|^2\D t\Big]& \leq \frac{1}{T}\E_x^{v,w^*} \Big[\int_0^T r^{\veps,v} (Z_t)\D t\Big]  -\rho^\veps_l\,.\end{align*}
 From Theorems~\ref{thm-lin-rep-pert}(i) and~\ref{thm-pert-limit}(ii), we know that $\lim_{\veps\to 0} \lim_{l\to\infty} \rho^\veps_l=\lsm[r]\geq 0\,.$
Therefore, using the arguments from the proof of Lemma~\ref{lem-comp-X}, gives us the result.
\end{proof}

\begin{proposition} \label{prop-optimal-minimizers}Suppose $v\in \Usm^{o}$ is optimal, \emph{i.e.,} $\Lambda_{v}[r]=\lsm[r]$. Then $v$ satisfies~\eqref{eqn-optimality1}.
\end{proposition} 

\begin{proof}
Fix $v\in\Usm^o$.  Suppose~\eqref{eqn-optimality1} fails to hold on an open set $\frB\subset \RR^d$. In other words, there exists a non-trivial non-negative function $\frK\in L^1(\frB)$ such that 
$$ \frK(x)\doteq \Big(\Lg^v V(x)+\big(r^v(x)-\Lambda_{\text{SM}}[r]\big)V(x)\Big)\Ind_{\frB}(x)\,.$$
In terms of $\widetilde V=\log V$, there exists another non-trivial non-negative function $\widetilde \frK\in L^1(\frB)$ such that 
$$ \widetilde \frK(x)\doteq \Big(\widehat \Lg^{v,\omega} \widetilde V(x)+r^v(x)-\Lambda_{\text{SM}}[r]-\frac{1}{2} \|\omega(x)\|^2\Big)\Ind_{\frB}(x)\,.$$
Recall that $\omega(x)=\Sigma(x)\transp \nabla \widetilde V(x).$
Since $\Psi^\veps_l$ converges to $\widetilde V^\veps$ ({strongly in $W^{1,p}_{\text{loc}}(\RR^d)$ for $p\geq 2$,} from Theorem~\ref{thm-lin-rep-pert}) and $\widetilde V^\veps $ converges to $\widetilde V$ 
 (as $V^\veps \to V$ from Proposition~\ref{prop-limit}), uniformly on compact sets of $\RR^d$, there exists { a family of non-trivial non-negative functions $\{\widetilde \frK^\veps_l\}_{\veps,l} \subset L^1(\frB)$} such that
$$ \frK^\veps_l(x)\doteq \Big(\Lg^v \Psi^\veps_l(x) -\rho^\veps_l+ \max_{w:\|w\|\leq l} \{ f^{v,\veps}_l (x,w) + {\Delta(x,w)}\cdot \nabla \Psi^\veps_l(x)\}\Big)\Ind_{\frB}(x)\,.$$ 
Moreover, $\lim_{\veps\to 0} \lim_{l\to\infty} \widetilde \frK^\veps_l=\widetilde \frK$ {strongly in} $L^1(\frB)$. 

Now let $w^*= w^*(\veps,l)$ be as in Proposition~\ref{prop-2p-game}(v).  Applying  It\^o-Krylov's formula to $\Psi^\veps_l(Z_t)$ with  $u\equiv v$ and $w\equiv w^*$  gives us
\begin{align*}
\E_x^{v,w^*}\Big[&\Psi^\veps_l(Z_{T\wedge \tau_R})\Big] -\Psi^\veps_l(x)=\E_x^{v,w^*}\Big[ \int_0^{T\wedge \tau_R} \Big(\rho^\veps_l- f^{v,\veps}_l (Z_t,w^*(Z_t)) + \widetilde \frK^\veps_l(Z_t)\Big)\D t\Big]\,.
\end{align*}
From Proposition~\ref{prop-2p-game}(iii), $$ \lim_{T\to\infty}\frac{1}{T}\E_x^{v,w^*}\Big[\Psi^\veps_l(Z_T)\Big]=0 \quad\text{ and } \quad \lim_{R\to\infty}\E_x^{v,w^*}\Big[\Psi^\veps_l(Z_{T\wedge\tau_R})\Big]=\E_x^{v,w^*}\Big[\Psi^\veps_l(Z_{T})\Big]\,.$$
Therefore, taking $R\to\infty$ and then $T\to\infty$ will give us
\begin{align*}
\Lambda_v[r^\veps]\geq \Lambda^l_{v,w^*} &\geq \rho^\veps_l + \limsup_{T\to\infty} \frac{1}{T}\E_x^{v,w^*}\Big[\int_0^T \widetilde \frK^\veps_l(Z_t)\D t\Big]\,.
\end{align*}
To get the first inequality above, we use  Proposition~\ref{prop-var-cost}(i) and the fact that $w^*\in \cA$.
With $\mu_v^{\veps,l}$ being the invariant measure of $Z$ under $v$ and $w^*$, we have $
\Lambda_v[r^\veps]\geq \rho^\veps_l+ \mu^{\veps,l}_v(\widetilde \frK^\veps_l)$. 
Since ${\{\mu^{\veps,l}_v\}_{\veps,l}}$ is tight in both $\veps$ and $l$ from Lemma~\ref{lem-tightness}, we can choose  a subsequence $ l_n$ along which $\mu^{\veps,l_n}_v$ converges weakly  for some measure $\mu^\veps_v$ and further choose a  subsequence $\veps_n$,  along which {$\mu^{\veps_n}_v$} converges weakly to some invariant measure $\mu^*_v$ of $Z$.   Since $v\in \Usm^o$,  we also {know} that  $\Lambda_v[r^\veps] \to \Lambda_v[r]$ (from Theorem~\ref{thm-pert-limit}(i)) and $\lim_{\veps\to 0} \lim_{l\to\infty}\rho^\veps_l \to \Lambda_v[r]$ (from Theorems~\ref{thm-lin-rep-pert}(i) and~\ref{thm-pert-limit}(ii)). This finally gives us  $ \lsm[r]=\Lambda_v[r]\geq \lsm[r]+ \mu^*_v(\widetilde \frK)$.
This is a contradiction as $\mu^{*}_v(\widetilde\frK)>0$ due to \cite[Theorem 2.6.16]{arapostathis2012ergodic} and this proves our result.
\end{proof}

\subsection{Proof of Theorem~\ref{thm-diffusion}(iii)} \label{sec-stoc-rep}
The proof of this part of Theorem~\ref{thm-diffusion} relies heavily on Lemma~\ref{lem-ground-diffusion} and Proposition~\ref{prop-ground-diffusion}.  
 
\begin{proposition}\label{prop-uniqueness}
For every $v\in \Usm^o$,  the statement of Theorem~\ref{thm-diffusion}(iii) holds.
 \end{proposition}

 \begin{proof} Fix $v\in \Usm^o$. 
 From Theorem~\ref{thm-diffusion}(ii), we know that for every $v\in \Usm^o$, it holds that 
$$ \Lg^v V(x)+r^v(x)V(x)=\lsm[r]V(x), \text{ for $x\in\RR^d$}\,.$$
From here and~\eqref{eq-stoc-rep-hjb} we have
\begin{align}\label{eq-V-rep-1}
V(x)= \E_x^{v}\Big[\exp\Big(\int_0^{\widecheck \tau_R} \big(r^v(X_t)-\lsm[r]\big) \D t\Big)V(X_{\widecheck \tau_R})\Big], \text{ for $x\in B_R^c$}\,.
\end{align}
This proves first part of Theorem~\ref{thm-diffusion}(iii).

 To prove the second part of Theorem~\ref{thm-diffusion}(iii), for $v\in \Usm^o$, suppose that $V$ satisfies~\eqref{eq-V-rep-1}, for some $R>0$. Let  $\widetilde \upu^v\in W^{2,d}_{\text{loc}}(\RR^d)$ be as given by Lemma~\ref{lem-ground-diffusion}. From Proposition~\ref{prop-ground-diffusion} and  \cite[Lemma 2.6]{ari2018strict}, we can conclude that $\Lambda_{v}[r]=\lsm[r]=\lambda_v^*[r]$. Moreover, using \cite[Theorem 2.3]{ari2018strict}, we have $\lambda^*_v[r+f]>\lambda_v^*[r], \text{ for $f\in \sC_0 $}$.
  Similarly,  we can infer also that $\Lambda_v[r+f]=\lambda_v^*[r+f]$,   for $f\in \sC_0$ and $v\in \Usm^o$. Combining this with the above display, we have $\Lambda_v[r+f]>\Lambda_v[r]$. This completes the proof.
 \end{proof}

\subsection{Proof of Theorem~\ref{thm-diffusion}(iv)} The main content here is to prove that optimal ERSC cost over all admissible controls $U\in \Uadm$ is equal to optimal ERSC cost over all stationary Markov controls $v\in \Usm$. 
The difficulty in proving this arises from the fact that we do not have any results that are analogous to Lemma~\ref{lem-comp-X} and Corollary~\ref{cor-comp-w} for $U\in {\Uadm^{*,\beta}}$. An implication of this is that we cannot make efficient use of variational formulation for $U\in {\Uadm^{*,\beta}}$ \emph{viz.,} Proposition~\ref{prop-var-cost}. Therefore, to overcome this, we prove that $\lsm[r]$ can be written as a limiting value of a family of TP-ZS games. The advantage of this is that for  the process $Z$ under $U\in {\Uadm^{*,\beta}}$ and a nearly optimal maximizing strategy of the TP-ZS game, the results analogous to Lemma~\ref{lem-comp-X} and Corollary~\ref{cor-comp-w} hold. This is sufficient for us to prove  Theorem~\ref{thm-diffusion}(iv).

\begin{lemma}\label{lem-s-i-0}
The following hold:
\begin{align}\label{eq-s-i-0}
\lsm[r]&=\sup_{w\in \Wsm}\inf_{v\in{\Usm^{*,\beta}}}\Lambda_{v,w}[r]\\\label{eq-l-s-i-0}
&=\lim_{l\to\infty}\sup_{w\in \Wsm(l)}\inf_{v\in {\Usm^{*,\beta}}}\Lambda_{v,w}[r]\,.\end{align}
\end{lemma}
\begin{proof} Since $\inf_{x}\sup_{y} f(x,y)\geq \sup_y\inf_x f(x,y)$, we have $$ \lsm[r]=\inf_{v\in{\Usm^{*,\beta}}} \sup_{w\in \Wsm}\Lambda_{v,w}[r] \implies \lsm[r]\geq \sup_{w\in \Wsm}\inf_{v\in{\Usm^{*,\beta}}}\Lambda_{v,w}[r]\,.$$
Therefore, it suffices to show that for any $\delta>0$, there exists $w^*\in \Wsm$ such that 
\begin{align}\label{eq-sup-inf-0-1} \lsm[r]\leq \inf_{v\in{\Usm^{*,\beta}}}\Lambda_{v,w^*}[r]+\delta\,.\end{align}
To do this, we use Theorem~\ref{thm-pert-limit}(ii) and infer that for $\delta>0$, there exists small enough $\veps$ such that 
\begin{align} \label{eq-adm-cost-1}\lsm[r]\leq \lsm[r^\veps]+{\frac{\delta}{3}}\,.\end{align}
From~\eqref{eq-alt-1} of Theorem~\ref{thm-lin-rep-pert}, we know that there exists $w^\veps\in \Wsm$ such that 
\begin{align}\label{eq-adm-cost-2} \lsm[r^\veps]\leq \inf_{v\in{\Usm^{*,\beta}}}\Lambda_{v,w^\veps}[r^\veps]+{\frac{\delta}{3}} \leq \Lambda_{v,w^\veps}[r^\veps]+{\frac{\delta}{3}}\,.\end{align}
From the above two displays, the definition of $r^\veps$ and $\Lambda_{v,w^\veps}[r^\veps]$, we can conclude that along a subsequence again denoted by $T$, we have
\begin{align*}
\limsup_{T\to\infty}\frac{1}{2T}\E_x^{v,w^\veps}\Big[ \int_0^T \|w^\veps(Z_t)\|^2\D t\Big]&\leq (-\lsm[r]+{\frac{2\delta}{3}}) + \limsup_{T\to\infty}\frac{1}{T}\E_x^{v,w^\veps}\Big[ \int_0^Tr^v(Z_t)\D t\Big]\\
&\qquad+ \veps \limsup_{T\to\infty}\frac{1}{T}\E_x^{v,w^\veps}\Big[ \int_0^Th^v(Z_t)\D t\Big]\,
\end{align*}
From here, using the arguments of proof of Lemma~\ref{lem-comp-X}, we conclude that third term on the right hand side goes to zero as $\veps \to 0$. {This means choosing $\veps$ small enough,  and  combining~\eqref{eq-adm-cost-1} and~\eqref{eq-adm-cost-2}, gives us}
\begin{align*}
\lsm[r]&\leq \Lambda_{v,w^\veps}[r] +\delta, \text{ for small enough $\veps$ and for $v\in {\Usm^{*,\beta}}$}\\
&\leq \inf_{v\in {\Usm^{*,\beta}}} \Lambda_{v,w^\veps}[r]+\delta\,.
\end{align*}
This proves~\eqref{eq-sup-inf-0-1}, and hence \eqref{eq-s-i-0}. Finally,~\eqref{eq-l-s-i-0} can be proved analogously. This completes the proof.
\end{proof}
The following proposition implies that the $\Lambda[r]$ can be achieved by $v\in \Usm^o$. 
\begin{proposition}\label{prop-cost-equal-0}$\Lambda[r]=\lsm[r]$.
\end{proposition}
\begin{proof}To begin with, it is clear that we have $\Lambda[r]\leq \lsm[r]$. For $\delta>0$, let $U^*$ be $\delta$--optimal for $\Lambda[r]$. 
 Recall that from  Theorems~\ref{thm-lin-rep-pert}(i) and~\ref{thm-pert-limit}(ii), $ \lim_{\veps\to 0}\lim_{l\to\infty} \rho^\veps_l=\lsm[r]$.
For small enough $\veps>0$ and large enough $l$, we can ensure that 
\begin{align}\label{def-le-choice} \lsm[r]\leq \rho^\veps_l+\delta=\sup_{w\in \Wsm(l)}\inf_{v\in \Usm^{*,\beta}}\Lambda_{v,w}^l + \delta\,.\end{align}
See~\eqref{def-lambda-e-l}, for the definition of $\Lambda_{v,w}^l$.    Let  $w^*\in \Wsm(l)$ be as in the statement of Proposition~\ref{prop-2p-game}(v), {\emph{i.e.,} it satisfies 
\begin{align}\label{def-w*} \sup_{w\in \Wsm(l)}\inf_{v\in {\Usm^{*,\beta}}}\Lambda^l_{v,w}= \inf_{v\in {\Usm^{*,\beta}}}\Lambda^l_{v,w^*}\,.\end{align}}
From the choice of $U^*$, we have
\begin{align}\nonumber
\Lambda[r]\geq J(x,U^*)[r]-\delta&\geq  \limsup_{T\to\infty} \sup_{w\in \cA} \frac{1}{T}\E_x^{{U^*},w}\Big[ \int_0^T\Big(r(Z_t,{U^*_t})-\frac{1}{2} \|w_t\|^2\Big)\D t\Big]-\delta\\\label{eq-l-lsm-1}
&\geq \limsup_{T\to\infty}  \frac{1}{T}\E_x^{{U^*},w^*}\Big[ \int_0^T\Big(r(Z_t,{U^*_t})-\frac{1}{2} \|w^*(Z_t)\|^2\Big)\D t\Big]-\delta\,.
\end{align}
In the above,  to get the second  inequality in the first line, we use Proposition~\ref{prop-var-cost} and to get the inequality in the second line, we use the fact that $w^*\in \Wsm(l)\subset \cA$ is sub-optimal for the above supremum. Since $w^*\in \Wsm(l)$ and $\sup_{x\in \RR^d}\|w^*(x)\|\leq l$, we can conclude that 
	$${ \limsup_{T\to\infty}}\frac{1}{T}\E_x^{U^*,w^*}\Big[\int_0^T \|w^*(Z_t)\|^2\D t \Big]\leq {l^2} \,.$$
	{From here,~\eqref{eq-l-lsm-1} and} the definition of $J(x,U^*,w^*)[r]$, we also have
	\begin{align*}
	\limsup_{T\to\infty}\frac{1}{T}\E_x^{U^*,w^*}\Big[\int_0^Tr (Z_t,U^*_t)\D t\Big] \leq J(x,U^*)[r] +\frac{l^2}{2}\,.
	\end{align*} 
	Following the arguments of the proof of Lemma~\ref{lem-comp-X-pert}, we get
	\begin{align}\label{eq-3-cost-equal} \limsup_{T\to\infty}\frac{1}{T}\E_x^{U^*,w^*}\Big[\int_0^Th(Z_t,U^*_t)\D t\Big]<\infty\,.\end{align}
	Therefore, the MEMs of $(Z_{[0,T]},U^*_{[0,T]})$ denoted by $\pi_T$ is a tight family of measures in $T>0$. This means that along a subsequence (still denoted by $T$), $\pi_T$ converges weakly to some measure $\pi_*$. From \cite[Lemma 3.4.6]{arapostathis2012ergodic}, $\pi_*$ is an ergodic occupation measure. With the decomposition of the measure, we can write $\pi_*(\D x,\D u)= \mu_v (\D x)v(\D u|x),$ for some $v\in {\Usm^{*,\beta}}$.  
Therefore, from the fact that $r$ is non-negative  and the fact that $w^*$ is bounded and continuous (this property follows from Proposition~\ref{prop-2p-game}(v)) , we can conclude that 
\begin{align}\label{eq-wc-1}\lim_{T\to\infty}\int_{\RR^d\times \bU} \Big(r(x,u)-\frac{1}{2}\|w^*(x)\|^2\Big)\D \pi_T(x,u)\geq \int_{\RR^d\times \bU} \Big(r(x,u)-\frac{1}{2}\|w^*(x)\|^2\Big)\D \pi_*(x,u)\,. \end{align}
From here, together with~\eqref{eq-opt-cont}, we have 
\begin{align*}
\Lambda[r] &\geq J(x,U^*)[r]-\delta\\
 &\geq \int_{\RR^d\times \bU} \Big(r(x,u)-\frac{1}{2}\|w^*(x)\|^2\Big)\D \pi_*(x,u) -\delta\\
 &= \int_{\RR^d\times \bU} \Big(\big(1-\frac{\veps}{\veps_0}\big)r(x,u) +\veps h(x,u)-\frac{1}{2}\|w^*(x)\|^2\Big)\D \pi_*(x,u) \\
 &\qquad\qquad+ \int_{\RR^d\times \bU}\Big(\frac{\veps}{\veps_0}r(x,u) -\veps h(x,u)\Big)\D \pi^*(x,u)  -\delta\\
&\geq   \int_{\RR^d\times \bU} \Big(r^\veps(x,u) -\frac{1}{2}\|w^*(x)\|^2\Big)\D \pi_*(x,u) - \veps\int_{\RR^d\times \bU} h(x,u)\D \pi^*(x,u)  -\delta\\
 &\geq   \int_{\RR^d\times \bU} \Big(r^\veps(x,u)\wedge L^* -\frac{1}{2}\|w^*(x)\|^2\Big)\D \pi_*(x,u) - \veps\int_{\RR^d\times \bU} h(x,u)\D \pi^*(x,u)  -\delta\\
&\geq \inf_{{v\in \Usm^{*,\beta}}} \Lambda^l_{v,w^*}-\veps\int_{\RR^d\times \bU} h(x,u)\D \pi^*(x,u) -\veps\int_{\RR^d\times \bU} h(x,u)\D \pi^*(x,u)-2\delta\\ 
&\geq \lsm[r]-\veps\int_{\RR^d\times \bU} h(x,u)\D \pi^*(x,u)-2\delta\,.
\end{align*}
In the above, to get the second inequality, we use~\eqref{eq-l-lsm-1} and~\eqref{eq-wc-1}; to get the third inequality, we use the fact that $r\geq 0$; 
to get the fourth inequality, we use the fact that $r^\veps\geq r^\veps \wedge L^*$ with $L^*$ as in~\eqref{def-game-cost}; to get the fifth inequality, we combine~\eqref{def-le-choice} and~\eqref{def-w*}; finally, to obtain the last inequality, we use Lemma~\ref{lem-tightness} to conclude that $\sup_{0< \veps<\veps_0}\limsup_{l\to\infty}\int_{\RR^d\times \bU} h(x,u)\D \pi^*(x,u)<\infty\,.$ 
Therefore, beginning with  $\veps>0$ sufficiently small, we can ensure that $\veps\int_{\RR^d\times \bU} h(x,u)\D \pi^*(x,u)<\delta\,.$ From the arbitrariness of $\delta>0$, we then have the result.
\end{proof}

{\appendix

\section{Proof of~\eqref{eq-limit-lambda} in Remark~\ref{rem-limit-lambda}} 
Here, we provide the main arguments of the  proof of~\eqref{eq-limit-lambda}.  Before we proceed, we note that for any $0<\kappa\leq 1$,   $$ J^\kappa(x,U)[ r]\doteq \limsup_{T\to\infty} \frac{1}{\kappa T} \log \E_x^U\left[ \exp\left( \kappa\int_0^T r({X}_t, U_t )\D t \right)\right]\leq \frac{1}{\kappa} J(x,U)[r]\,.$$  
For $0<\kappa\leq 1$ and $U\in \Uadm^{*,\beta}$, applying Proposition~\ref{prop-var-cost}(ii) to the functional $\kappa T^{-1} \int_0^T r(X_t,U_t)\D t$, we have 
\begin{align*}
J^\kappa(x,U)[r]&\geq  \limsup_{T\to\infty}\sup_{w\in\cA} \E_x^{U,w}\Bigg[ \frac{1}{T}\int_0^T \Big( r\big(Z_t, U_t(W_{[0,t]}+ \widetilde w_{[0,t]})\big) - \frac{1}{2\kappa}\|w_t\|^2\Big)\D t\Bigg]\\
&\geq \limsup_{T\to\infty} \E_x^{U,0}\Bigg[ \frac{1}{T}\int_0^T r\big(Z_t, U_t\big) \D t\Bigg]\\
&= J^0(x,U)[r]\\
&\geq  \Lambda^0[r]\,.
\end{align*}
To get the second line, we use the fact that $w_t\equiv 0$ is sub-optimal to the supremum in the first line; to get the third  and fourth lines, we use  definitions of $J^0(x,U)$ and $\Lambda^0[r]$, and the fact that the process $Z$ under controls $U$ and $w\equiv 0$ is the same as $X$ under control $U$. This proves that  $\Lambda^0[r]\leq \liminf_{\kappa\to 0} \Lambda^\kappa[r]$. It now remains to show that 
\begin{align}\label{app-eq-2}
\limsup_{\kappa\to 0} \Lambda^\kappa[r]\leq \Lambda^0[r]\,.
\end{align} 
To do this, we first prove~\eqref{eq-mixed-condition-rs}. Recall that $\sV^\kappa(x)=\exp\big(\kappa \frV(x)\big)$. We obtain
\begin{align}\label{app-eq-3}
\Lg^u \sV^\kappa(x) &= \frac{\kappa^2}{2} \exp\big(\kappa \frV(x)\big)\|\Sigma(x)\transp \nabla \frV(x)\|^2 +\kappa \exp\big(\kappa \frV(x)\big)\Lg^u \frV(x)\,.
\end{align}
Since $\sV(x)=\exp\big(\frV(x)\big)$, for  $(x,u)\in \cK^c\times \bU$, we have $\Lg^u\frV(x)\leq C_1-\bar h(x,u) -\frac{1}{2}\|\Sigma(x)\transp \frV(x)\|^2$. Substituting this in~\eqref{app-eq-3}, we get
\begin{align*}
\Lg^u \sV^\kappa(x) &\leq  \frac{\kappa^2}{2} \exp\big(\kappa \frV(x)\big)\|\Sigma(x)\transp \nabla \frV(x)\|^2 +\kappa \exp\big(\kappa \frV(x)\big)\Big(C_1-\bar h(x,u) -\frac{1}{2}\|\Sigma(x)\transp \frV(x)\|^2\Big)\\
&=  \exp\big(\kappa \frV(x)\big)\Big(\frac{\kappa^2-\kappa}{2} \|\Sigma(x)\transp \nabla \frV(x)\|^2 + \kappa C_1-\kappa \bar h(x,u) \Big)\\
&\leq \big( \kappa C_1-\kappa \bar h(x,u) \big)\sV^\kappa(x)\,.
\end{align*}
To get the third line,  we use the fact that $\kappa \in (0,1]$  and the definition of $\sV^\kappa(x)$. Similarly, we can show that  for $(x,u)\in \cK\times \bU$, 
$$\Lg^u\sV^\kappa(x)\leq \big( \kappa C_2+C_3\kappa r(x,u) \big)\sV^\kappa(x)\,.$$
This proves~\eqref{eq-mixed-condition-rs}. We now proceed to prove~\eqref{app-eq-2}. For every $\delta>0$, recall that there exists $v^\delta\in \Usm$ and $x^\delta\in \RR^d$ such that~\eqref{eq-mixed-1} holds. Consequently, for any $\kappa <\kappa^\delta$, we have 
$$ \limsup_{T\to\infty} \frac{1}{\kappa_\delta T} \log \E_{x^\delta}^{v^\delta}\left[ \exp\left( (1+\eta^\delta)\kappa\int_0^T r({X}_t, v^\delta(X_t) )\D t \right)\right]<\infty$$
with $\eta^\delta=\frac{\kappa^\delta}{\kappa}-1>0$. Thus, the analysis and the results from Section~\ref{sec-vepsgeq0} can be applied to the running cost $\kappa r$ with control $v^\delta$ with $\kappa<\kappa^\delta$.  Therefore, again applying  Proposition~\ref{prop-var-cost}(i) to the functional $\kappa T^{-1} \int_0^T r(X_t,v^\delta(X_t))\D t$,  we get
\begin{align}\nonumber
J^\kappa(x^\delta,v^\delta)[r]&= \limsup_{T\to\infty}\sup_{w\in\cA} \E_x^{v^\delta,w}\Bigg[ \frac{1}{T}\int_0^T \Big( r\big(Z_t, v^\delta (Z_t)\big) - \frac{1}{2\kappa}\|w_t\|^2\Big)\D t\Bigg]\\\label{app-eq-4}
&\leq \limsup_{T\to\infty}\E_{x^\delta}^{v^\delta,w}\Bigg[ \frac{1}{T}\int_0^T \Big( r\big(Z_t, v^\delta (Z_t)\big) - \frac{1}{2\kappa}\|w^*_t\|^2\Big)\D t\Bigg]+\delta
\end{align}
with $w^*=w^*(\delta,T)$ being $\delta$--optimal for the supremum in the first line. From results analogous to Lemmas~\ref{lem-comp-w-pert} and~\ref{lem-comp-X-pert}, we can infer that for some $M_7>0$,
	\begin{align*}\limsup_{T\to\infty} \frac{1}{\kappa T}\E_x^{v^\delta,w^*}\Big[\int_0^T \|w^*_t\|^2 \D t\Big]\leq M_7 \,,\end{align*}
	and that the family of MEMs $\{\pi_{\kappa,T}\}_{\kappa,T}$ associated with $(Z,w^*)$ is tight in $\kappa$ and $T$. Moreover, $r(x,u)$ is uniformly integrable with respect to $\{\pi_{\kappa,T}\}_{\kappa,T}$. Therefore, along a subsequence of $T$ (again denoted by $T$), $\pi_{\kappa,T}$ converges weakly to some measure $\widetilde \pi_\kappa$. It is easy to see from the above display that  $\{\int_{x\in \RR^d}\widetilde \pi_\kappa(\D x,\D w)\}_{0<\kappa\leq 1}$  converges weakly as $\kappa\to 0$, to $\delta_0$, Dirac delta measure at $0$. This consequently means that along a subsequence of $\kappa$ (again denoted by $\kappa$), $\widetilde \pi_\kappa$ converges to an invariant  measure $\widetilde \pi$ associated with $Z$ under $v^\delta $ and $w=0$ (which is the same as the process $X$ under $v$). To summarize,~\eqref{app-eq-4} becomes 
	\begin{align*}
	\Lambda^\kappa \leq J^\kappa(x^\delta,v^\delta)[r]&\leq \int_{\RR^d\times \bU} r^{v^\delta}(x) \widetilde \pi(\D x)- \liminf_{\kappa\to 0}\liminf_{T\to\infty}\E_{x^\delta}^{v^\delta,w^*}\Big[ \frac{1}{2\kappa}\|w^*_t\|^2\D t\Big]+\delta\\
	&\leq \int_{\RR^d\times \bU} r^{v^\delta}(x) \widetilde \pi(\D x)+\delta\\
	&\leq \Lambda^0[r]+2\delta\,.
	\end{align*}
We get the integral term in the first line from uniform integrability of $r(x,u)$ with respect to $\{\pi_{\kappa,T}\}_{\kappa,T}$.	 From the arbitrariness of $\delta$, we have~\eqref{app-eq-2} and also~\eqref{eq-limit-lambda}. }

\section*{Acknowledgement}
This work is funded by  the NSF Grant DMS 2216765.

\bibliographystyle{abbrv}
\bibliography{RS-control}

\end{document}